\documentclass[10pt,a4paper]{article}
\usepackage[utf8]{inputenc}
\usepackage[T1]{fontenc}
\usepackage[leqno]{amsmath}
\usepackage{amsmath,amsfonts,amsthm,amssymb}
\usepackage{epsfig}
\usepackage{graphicx}
\usepackage{subfig}
\usepackage{authblk}
\usepackage[linesnumbered,ruled,vlined]{algorithm2e}
\usepackage[title]{appendix}

\usepackage{cite}
\usepackage{afterpage,hyperref}
\usepackage{cleveref}
\usepackage{makecell}
\usepackage{diagbox}
\usepackage{comment}
\usepackage{float}
\usepackage{enumerate}
\usepackage[justification=centering]{caption}
\usepackage[left=2cm,right=2cm,top=2cm,bottom=2cm]{geometry}

\usepackage{makecell}
\usepackage{tablefootnote}
\usepackage{cancel}
\usepackage{diagbox}

\hypersetup{
  colorlinks,
  citecolor=blue!50!black,
  linkcolor=red,
  urlcolor=blue!50!black}
\usepackage{standalone}
\usepackage{pgf,ifthen}
\usepackage{tikz}
\usetikzlibrary{shapes,graphs,graphs.standard}
\usetikzlibrary{backgrounds,calc,positioning}
\usepackage{epsfig}
\usepackage{color}
\usepackage{amsmath}
\usepackage{amssymb}
\newtheorem{lemma}{Lemma}

\newtheorem{theorem}{Theorem}

\newtheorem{claim}{Claim}
\newtheorem{observation}{Observation}

\theoremstyle{definition}

\newcommand{\ru}[1]{\textbf{R#1}}
\newcommand{\db}[2]{#1\!\leftrightarrow\!#2}
\newcommand{\gs}[0]{$G[S]^2=G^2[S]$}

\newcommand{\smallqed}{{\tiny $\left(\Box\right)$}}
\newcommand{\figureproof}[1]{\noindent\emph{Proof of #1.} }

\tikzstyle{graphnode}=[draw,shape=circle,draw=black,minimum size=0.5pt,inner sep=1.5pt]

\DeclareMathOperator{\mad}{mad}

\makeatletter
\providecommand\phantomcaption{\caption@refstepcounter\@captype}
\makeatother

\title{Computer assisted discharging procedure on planar graphs: application to 2-distance coloring}
\author[1]{Hoang La}
\author[1]{Petru Valicov}
\affil[1]{LIRMM, Université de Montpellier, CNRS, Montpellier, France}

\begin{document}
\maketitle

\begin{abstract}
Using computational techniques we provide a framework for proving results on subclasses of planar graphs via discharging method. The aim of this paper is to apply these techniques to study the 2-distance coloring of planar subcubic graphs. Applying these techniques we show that every subcubic planar graph $G$ of girth at least 8 has 2-distance chromatic number at most 6.
\end{abstract}

The discharging method is a very common tool used for proving coloring results on sparse graphs. At heart, it is a counting argument that guarantees the existence of (easily) colorable structures in a given sparse graph. Such structures are commonly named \emph{reducible configurations} as they cannot appear in a minimal counterexample to a desired theorem. A typical counting argument in the discharging method consists in translating the global sparseness of the graph into local weights, called \emph{charges}. For instance, a charge can be the degree of a vertex or the size of a face (when the graph is planar). The goal then is to obtain, through a clever redistribution of these charges, a contradiction by showing that there exists a reducible configuration in a minimal counterexample. This redistribution is done via \emph{discharging rules}. See the survey of Cranston and West~\cite{cw17} for more detailed explanation.

The limit of this method is achieved when one needs to consider a large amount of case distinctions in a proof. This happens essentially for two main reasons: the coloring of a configuration involves a complicated case analysis, or the set of reducible configurations needed in the proof is (too) large. Hence, using computer assistance seems to be the most natural way to overcome this hurdle. Showing that a configuration is reducible is very dependent on the type of coloring. On the other hand, generating a set of unavoidable configurations is more dependent on the class of graphs. The most famous example of computer assistance in discharging is the proof of the Four Color Theorem~\cite{AH77,AHK77,ROBERTSON19972}. In this paper, we present an algorithm that, given a particular set of discharging rules, generates all to-be-reduced configurations for planar graphs. We implemented this algorithm and applied it to show the $2$-distance colorability of a subclass of subcubic planar graphs. The source code can be found at \url{https://gite.lirmm.fr/discharging/planar-graphs}.

Before going into the details of 2-distance coloring problems, we wish to highlight that, even though the majority of the paper deals with the technicality of this particular problem, our algorithm is independent from the coloring problem.

A $2$-distance $k$-coloring of a graph $G=(V,E)$ is a map $\phi:V\rightarrow\{1,2,\dots,k\}$ such that no pair of vertices at distance at most 2 receives the same color $c\in\{1,2,\dots,k\}$. Wegner~\cite{wegner} conjectured that subcubic planar graphs are 7-colorable. This conjecture was proved by two independent group of authors, the first one using a graph decomposition (Thomassen~\cite{tho18}), the second one using a computer-assisted discharging method (Hartke \textit{et al.}~\cite{har16}). The authors in~\cite{har16} used computer assistance to $2$-distance color a given set of large configurations. Our approach differs as we use the computer in order to generate the set of configurations needed to be reduced (according to the discharging rules) instead. Moreover, for our specific problem, the reducible configurations are not always colorable by computer with a naive exhaustive precoloring extension algorithm (see discussion after \Cref{lemma:reducible_cycles}).

The $2$-distance colorability of planar graphs with high girth is extensively studied in the literature. See \cite{la20212distance} for a detailed state of art. We focus on the case of subcubic graphs. In 2008, Cranston and Kim proved the following result:

\begin{theorem}[\!\!\cite{CK08}]
\label{thm:girth9}
Let $G$ be a planar subcubic graph with girth $g\geq 9$. Then $\chi^2(G)\leq 6$.
\end{theorem}

Note that their result also applies for the list version of the problem. We improve \Cref{thm:girth9} by lowering the bound on the girth. Our proof relies heavily on the assumption that the colors are taken from the same set of six colors, thus it does not seem to be extendable to the list version of the problem.

\begin{theorem}
\label{thm:girth8}
Let $G$ be a planar subcubic graph with girth $g\geq 8$. Then $\chi^2(G)\leq 6$.
\end{theorem}

The proof is done by induction on the order of the graph using the discharging method. We will assume a minimum counterexample and show a set of reducible configurations which it cannot contain (\Cref{sec:structural_properties}). Then, using Euler's formula, we define a distribution of charges on the vertices and faces of this hypothetical counterexample such that the total amount of charges is negative. In order to obtain a non-negative total amount of charges on the vertices, we use the same distribution of charges on the vertices as in the proof of \Cref{thm:girth9} (\Cref{subsec:discharging_vertices}). With this distribution, the only faces with negative charge are of length 8. With the assistance of a computer program, we list each possible close neighbourhoods around a face of length 8. For each of these neighborhoods, our algorithm shows that either it contains a reducible configuration or it can get enough charge from its incident vertices (\Cref{subsec:second_round_discharging}). This leads to a contradiction.

In \Cref{sec4}, we discuss the tightness of \Cref{thm:girth8} and possible extensions. Finally, in \Cref{sec5}, we explain how to use our algorithm to solve problems on other subclasses of planar graphs.

\paragraph{Notations:} In the following, we only consider plane graphs that is planar graphs together with their embedding into the plane. For a plane graph $G$, we denote $V$, $E$, $F$ the sets of vertices, edges and faces respectively. We denote $d(v)$ (resp. $d(f)$) the degree of vertex $v\in V$ (resp. the size of face $f\in F$).

Some more notations:
\begin{itemize}
\item A \emph{$d$-vertex} is a vertex of degree $d$.
\item A \emph{$d$-face} is a face of size $d$.
\item A \emph{$k$-path} is a path of length $k+1$ where the $k$ internal vertices are 2-vertices.
\item A \emph{$(k_1,k_2,k_3)$-vertex} is a 3-vertex incident to a $k_1$-path, a $k_2$-path and a $k_3$-path.
\end{itemize}

Recall that in the whole paper we do a 2-distance 6-coloring. Thus, for a vertex $v$, we denote $L(v)$ the set of available colors from $\{a,b,c,d,e,f\}$.
For convenience, in the figures a vertex $v$ will be represented by a circle labeled $v$. Additionally, when a lower bound on $|L(v)|$ is known, it will be depicted on the figure. For example, the graph depicted in \Cref{subfig:config1} is a path $v_1v_2v_3v_4$ with the following size of lists of available colors: $|L(v_1)|\geq 2$, $|L(v_2)|\geq 3$, $|L(v_3)|\geq 2$, $|L(v_4)|\geq 2$.

We will also say that a vertex $u$ \textit{sees} another vertex $v$ if $v$ is at distance at most 2 from $u$.

\section{Useful observations and lemmata}
\label{sec:useful}
Here we show some colorable and non-colorable configurations, that is graphs together with lists of available colors for each vertex.
These observations will be extensively used in \Cref{sec:structural_properties}.

\begin{lemma}
\label{lemma:observations}
The graphs depicted in \Crefrange{subfig:config1}{subfig:config15} are 2-distance colorable.
\end{lemma}
\begin{proof}
In the proofs of this section, whenever the size of a list $|L(v)|\geq k$ we assume that $|L(v)|=k$ by arbitrarily removing the extra colors from the list. One can easily observe that these proofs will hold for the case when $|L(v)|> k$. 

We will give the proofs for each figure in order:

\figureproof{\Cref{subfig:config1}}
If $v_1$ and $v_4$ can be colored with the same color, then finish by coloring $v_2$, $v_3$ in this order. Otherwise, since $L(v_1)\cap L(v_4)=\emptyset$ we have $|L(v_1)\cup L(v_4)|\geq 4$, so one can apply Hall's Theorem. 
\hfill\smallqed\medskip

\figureproof{\Cref{subfig:config2}}
If $L(v_4)\neq L(v_5)$, then color $v_5$ with $x\notin L(v_4)$ and get \Cref{subfig:config1}, so we are done. Otherwise, color $v_3$ with a color $y\notin L(v_5) \cup L(v_4)$. Then color $v_1$, $v_2$, $v_4$, $v_5$, in this order.
\hfill\smallqed\medskip

\figureproof{\Cref{subfig:config3}}
If $L(v_1)\neq L(v_3)$, then we color $v_1$ with $x\notin L(v_3)$ and get \Cref{subfig:config2}. Otherwise, color $v_2$ with a color $y\notin L(v_3) \cup L(v_1)$, then color $v_3$, $v_4$, $v_5$, $v_6$ using \Cref{subfig:config1} and finish by coloring vertex $v_1$.
\hfill\smallqed\medskip

\figureproof{\Cref{subfig:config4}}
Observe that $L(v_3)=L(v_4)$ because if not we color $v_4$ with $x\notin L(v_3)$ and we get \Cref{subfig:config1}. Thus color $v'_3$ with $y\notin L(v_3)$ and get \Cref{subfig:config1} again.
\hfill\smallqed\medskip

\figureproof{\Cref{subfig:config5}}
If $L(v_2)\neq L(v_4)$, then one could color $v_4$ with $x\notin L(v_2)$, then by \Cref{subfig:config1} we are done. Otherwise, since $|L(v'_3)|\geq 3$, color $v'_3$ with a color $y\notin L(v_4) \cup L(v_2)$. Then again by \Cref{subfig:config1} we are done.
\hfill\smallqed\medskip

\figureproof{\Cref{subfig:config6}}
Observe that there exists $ x\in L(v'_3)\setminus L(v_2)$. Thus $x\in L(v_4)$ as otherwise one could color $v'_3$ with $x$ and get \Cref{subfig:config2}. Hence $x\in L(v_5)$, as otherwise one could color $v_4$ with $x$, color vertices $v_1,v_2,v_3,v'_3$ by \Cref{subfig:config1} and finish by coloring vertex $v_5$. Therefore, we color $v'_3$ and $v_5$ with $x$ and we get \Cref{subfig:config1}.
\hfill\smallqed\medskip

\figureproof{\Cref{subfig:config7}}
First observe that $L(v_1)\subset L(v'_2)$. Otherwise, by coloring $v_3$ with $x\notin L(v_1)$ and coloring $v_4$, $v'_3$ and $v_2$ in this order, one could finish with vertices $v_1$ and $v'_2$ which see the same colored vertices while $L(v_1)\not\subset L(v'_2)$. Now, suppose $L(v_3)\neq L(v'_2)$ and color vertex $v_3$ with $y\notin L(v'_2)\supset L(v_1)$. Then color $v_4$, $v'_3$, $v_2$, $v_1$, $v'_2$ in this order. Therefore $L(v_3)=L(v'_2)\supset L(v_1)$ and we color $v_2$ with $z\notin L(v_3)$ and finish by coloring $v_4$, $v'_3$, $v_3$, $v_1$, $v'_2$ in this order.
\hfill\smallqed\medskip

\figureproof{\Cref{subfig:config8}}
First note that $L(v_4)=L(v_5)$ as otherwise by coloring $v_5$ with $x\notin L(v_4)$ we get \Cref{subfig:config7}. If $L(v_5)\subset L(v'_3)$, then we color vertex $v_3$ with $y\notin L(v'_3)$ and $v_1$, $v'_2$, $v_2$, $v_4$, $v_5$, $v'_3$ in this order. We conclude that $|L(v'_3)\setminus L(v_5)|\geq 2$. Thus by replacing $L(v'_3)$ with $L(v'_3)\setminus L(v_5)$ and $L(v_3)$ with $L(v_3)\setminus L(v_5)$, we can color vertices $v_1$, $v_2$, $v'_2$, $v_3$, $v'_3$ by \Cref{subfig:config5} and finish by coloring vertices $v_4$ and $v_5$.
\hfill\smallqed\medskip

\figureproof{\Cref{subfig:config9}}
Suppose $L(v_2)\neq L(v'_3)$. Then restrict the list of colors of $v_3$ to $L(v_3)\setminus L(v_1)$, color vertices $v_3$, $v_4$, $v'_4$, $v_5$ and $v_6$ by \Cref{subfig:config5} and finish by coloring $v'_3$, $v_2$ and $v_1$ in this order.
Therefore, we have $L(v_2) = L(v'_3)$. Now, if $L(v_5)\neq L(v_6)$, then we color vertex $v_4$ with $x\notin L(v'_3)$, color $v_5$ and $v_6$ (because theirs lists are different) and finish by coloring $v'_4$, $v_3$, $v_1$, $v_2$ and $v'_3$ in this order.
Thus  we have $L(v_5) = L(v_6)$. Color vertex $v_3$ with $y\notin L(v_2)=L(v'_3)$. If $y\in L(v_6)$, then color vertex $v_6$ with $y$ and finish by coloring $v_5$, $v'_4$, $v_4$, $v_1$, $v_2$, $v'_3$ in this order. If $y\notin L(v_6)=L(v_5)$, then color $v'_4$, $v_4$, $v_5$, $v_6$ by \Cref{subfig:config1} and finish by coloring $v_1$, $v_2$, $v'_3$ in this order.
\hfill\smallqed\medskip

\figureproof{\Cref{subfig:config10}}
If $L(v_1)\not\subset L(v_2)$, then by coloring $v_1$ with $y\notin L(v_2)$ we get \Cref{subfig:config8}. Hence, we have w.l.o.g. $L(v_1)=\{a,b\}$ and $L(v_2)=\{a,b,c\}$.

If $L(v_2)\not\subset L(v_3)$, then we restrict $L(v_3)$ to $L(v_3)\setminus L(v_2)$. Observe that $|L(v_3)\setminus L(v_2)|\geq 3$. Now, we look at the two following cases:
\begin{itemize}
\item When $L(v'_3)=L(v''_3)$, we color $v_3$ with $x\notin L(v'_3)$ and then $v_5$, $v'_4$, $v_4$, $v'_3$, $v''_3$, $v_2$, $v_1$ in this order.  
\item When $L(v'_3)\neq L(v''_3)$, we color $v''_3$ with $y\notin L(v'_3)$ and we obtain \Cref{subfig:config5}. We color $v_2$ and $v_1$ last.
\end{itemize}
So, $L(v_2)\subset L(v_3)$. We can thus assume w.l.o.g. that $L(v_3)=\{a,b,c,d,e\}$. 

If $d\notin L(v'_3)\cup L(v''_3)$, then we color $v_3$ with $d$, then $v_5$, $v'_4$, $v_4$, $v'_3$, $v''_3$, $v_2$, $v_1$ in this order. The same holds for $e$. So, we must have $\{d,e\}\subseteq L(v'_3)\cup L(v''_3)$.

If $L(v'_3)=L(v''_3)$, then due to the previous observation, $L(v'_3)=L(v''_3)=\{d,e\}$. In this case, we color $v_3$ with $c$, then $v_5$, $v'_4$, $v_4$, $v'_3$, $v''_3$, $v_2$, $v_1$ in this order. As a result, $L(v'_3)\neq L(v''_3)$.

If $L(v'_3)\subset L(v_2)$, then we must have $L(v''_3)=\{d,e\}$. We then color $v_3$ with $d$, then $v''_3$, $v_5$, $v'_4$, $v_4$, $v'_3$, $v_2$, $v_1$ in this order. 

If $L(v'_3)\not\subset L(v_3)$, then $f\in L(v'_3)$. We color $v'_3$ with $f$, then $v''_3$ and $v_5$. We can then finish coloring $v_1$, $v_2$, $v_3$, $v_4$, $v'_4$ by \Cref{subfig:config2}. We can thus assume w.l.o.g that $d\in L(v'_3)$.

If $c\notin L(v'_3)$, then we color $v_2$ with $c$, $v_4$ with $x\in L(v_4)\setminus L(v'_3)$, and $v_5$, $v'_4$, $v_3$, $v_1$ in this order. We can finish by coloring $v'_3$ and $v''_3$ since $L(v'_3)\neq L(v''_3)$. So, $c\in L(v'_3)$.

To summarize the previous observations, we have $L(v_1)=\{a,b\}$, $L(v_2) = \{a,b,c\}$, $L(v_3)=\{a,b,c,d,e\}$, $L(v'_3) = \{c,d\}$ and $e\in L(v''_3)$. We color $v''_3$ with $e$. We restrict $L(v_3)$ to $\{c,d\}$. We color $v'_3$, $v_3$, $v_4$, $v'_4$, $v_5$ by \Cref{subfig:config5}. Finally, we finish by coloring $v_2$ and $v_1$ in this order. 
\hfill\smallqed\medskip

\figureproof{\Cref{subfig:config11}}
If $L(v_2)\neq L(v_1)$, then color $v_2$ with $x\notin L(v_1)$, color vertices $v'_4$, $v_4$, $v_5$, $v_6$ by \Cref{subfig:config1} and finish with $v_3$ and $v_1$. If $L(v_2)=L(v_1)$, then by restricting the list of colors of $v_3$ to $L(v_3)\setminus L(v_2)$, we color vertices $v_3$, $v_4$, $v'_4$, $v_5$, $v_6$ by \Cref{subfig:config5} and finish with $v_2$ and $v_1$.
\hfill\smallqed\medskip

\figureproof{\Cref{subfig:config12}}
Observe that $L(v_1)=L(v_2)$ since otherwise one could color $v_1$ with $x\notin L(v_2)$ and get \Cref{subfig:config6}.
Therefore, we restrict the list of colors of $v_3$ to $L(v_3)\setminus L(v_2)$. We color then $v_3$, $v_4$, $v'_4$, $v_5$, $v_6$ by \Cref{subfig:config5} and finish with $v_2$ and $v_1$.
\hfill\smallqed\medskip

\figureproof{\Cref{subfig:config13}}
If $L(v_5)\neq L(v_6)$, then by coloring $v_6$ with $x\notin L(v_5)$, one could finish by \Cref{subfig:config2}. Thus $L(v_5)=L(v_6)$ and we restrict the list of colors of $v_4$ to $L(v_4)\setminus L(v_5)$, color vertices $v_1$, $v_2$, $v_3$, $v_4$ by \Cref{subfig:config1} and finish with $v_5$ and $v_6$.
\hfill\smallqed\medskip

\figureproof{\Cref{subfig:config14}}
Observe that $L(v_1)=L(v_2)$ as otherwise by coloring $v_2$ with $x\notin L(v_1)$, one could color $v_3,v_4,v_5,v_6,v_7$ by \Cref{subfig:config2} and finish by coloring $v_1$.
Therefore, color $v_3$ with $y\notin L(v_2) \cup L(v_1)$, color $v_4,v_5,v_6,v_7$ by \Cref{subfig:config1} and finish by coloring $v_2,v_1$ in this order.
\hfill\smallqed\medskip

\figureproof{\Cref{subfig:config15}}
Note that $L(v_6)=L(v_7)$ as otherwise by coloring $v_7$ with $x\notin L(v_6)$ one could finish by \Cref{subfig:config13}. Hence color $v_5$ with $y\notin L(v_7) \cup L(v_6)$, then color $v_1,v_2,v_3,v_4$ by \Cref{subfig:config1} and finish with $v_6$, $v_7$.
\hfill\smallqed\medskip

\figureproof{\Cref{subfig:config16}}
If it is possible to color $v_1$ and $v_5$ with the same color, then after coloring $v_6$, we get \Cref{subfig:config10}. Hence $L(v_1)\cap L(v_5)=\emptyset$. If it is possible to color $v_5$ and $v'_2$ with a common color, then after coloring $v_6$, we get again \Cref{subfig:config10}. Hence $L(v'_2)\cap L(v_5)=\emptyset$. Symmetrically, we have $L(v''_3)\cap L(v_5)=\emptyset$ and $L(v'''_3)\cap L(v_5)=\emptyset$.

Now, since we are considering a 6-coloring, we restrict the list of colors of $v_3$ to $L(v_3)=L(v_5)$ and color vertices $v_3$, $v_4$, $v'_4$, $v_5$, $v_6$ by \Cref{subfig:config5}. We finish by coloring the remaining vertices in the following order: $v_1$, $v_2$, $v'_2$, $v'_3$, $v''_3$, $v'''_3$.
\hfill\smallqed\medskip

\end{proof}

\begin{figure}[!ht]
\hspace*{-1cm}
\subfloat[]
{
\label{subfig:config1}                                                 
\scalebox{0.75}{
\begin{tikzpicture}[join=bevel,vertex/.style={circle,draw, minimum size=0.6cm},inner sep=0mm,scale=0.85]
  \foreach \i / \l in {0/1,2/2,4/3,6/4}{
  	\node[vertex] (\l) at (\i,0) {$v_\l$};
  }
  \node at (0,-0.7) {$2$};
  \node at (2,-0.7) {$3$};
  \node at (4,-0.7) {$2$};
  \node at (6,-0.7) {$2$};
  \draw[-] (1) -- (2) -- (3) -- (4);
\end{tikzpicture}
}
}
\hspace*{0.4cm}
\subfloat[]
{
\label{subfig:config2}
\scalebox{0.75}{
\begin{tikzpicture}[join=bevel,vertex/.style={circle,draw, minimum size=0.6cm},inner sep=0mm,scale=0.85]
  \foreach \i / \l in {0/1,2/2,4/3,6/4,8/5}{
  	\node[vertex] (\l) at (\i,0) {$v_\l$};
  }
  \node at (0,-0.7) {$2$};
  \node at (2,-0.7) {$3$};
  \node at (4,-0.7) {$3$};
  \node at (6,-0.7) {$2$};
  \node at (8,-0.7) {$2$};
  \draw[-] (1) -- (2) -- (3) -- (4) -- (5);
\end{tikzpicture}
}
}
\hspace*{0.4cm}
\subfloat[]
{
\label{subfig:config3}
\scalebox{0.75}{
\begin{tikzpicture}[join=bevel,vertex/.style={circle,draw, minimum size=0.6cm},inner sep=0mm,scale=0.85]
  \foreach \i / \l in {0/1,2/2,4/3,6/4,8/5,10/6}{
  	\node[vertex] (\l) at (\i,0) {$v_\l$};
  }
  \node at (0,-0.7) {$2$};
  \node at (2,-0.7) {$3$};
  \node at (4,-0.7) {$2$};
  \node at (6,-0.7) {$3$};
  \node at (8,-0.7) {$3$};
  \node at (10,-0.7) {$2$};
  \draw[-] (1) -- (2) -- (3) -- (4) -- (5) -- (6);
\end{tikzpicture}
}
}

\hspace*{-1.8cm}
\subfloat[]
{
\label{subfig:config4}
\scalebox{0.7}{
\begin{tikzpicture}[join=bevel,vertex/.style={circle,draw, minimum size=0.6cm},inner sep=0mm,scale=0.85]
  \foreach \i / \l in {0/1,2/2,4/3,6/4}{
  	\node[vertex] (\l) at (\i,0) {$v_\l$};
  }
  \node at (0,-0.7) {$2$};
  \node at (2,-0.7) {$4$};
  \node at (4,-0.7) {$2$};
  \node[vertex] (3bis) at (4, 1.5) {$v_3'$};
  \node at (4, 2.2) {$3$};
  \node at (6,-0.7) {$2$};
  \draw[-] (1) -- (2) -- (3) -- (4)  (3) -- (3bis);
\end{tikzpicture}
}
}
\hspace*{0.3cm}
\subfloat[]
{
\label{subfig:config5}
\scalebox{0.7}{
\begin{tikzpicture}[join=bevel,vertex/.style={circle,draw, minimum size=0.6cm},inner sep=0mm,scale=0.85]
  \foreach \i / \l in {0/1,2/2,4/3,6/4}{
  	\node[vertex] (\l) at (\i,0) {$v_\l$};
  }
  \node at (0,-0.7) {$2$};
  \node at (2,-0.7) {$2$};
  \node at (4,-0.7) {$4$};
  \node[vertex] (3bis) at (4, 1.5) {$v_3'$};
  \node at (4, 2.2) {$3$};
  \node at (6,-0.7) {$2$};
  \draw[-] (1) -- (2) -- (3) -- (4)  (3) -- (3bis);
\end{tikzpicture}
}
}
\hspace*{0.3cm}
\subfloat[]
{
\label{subfig:config6}
\scalebox{0.7}{
\begin{tikzpicture}[join=bevel,vertex/.style={circle,draw, minimum size=0.6cm},inner sep=0mm,scale=0.85]
  \foreach \i / \l in {0/1,2/2,4/3,6/4,8/5}{
  	\node[vertex] (\l) at (\i,0) {$v_\l$};
  }
  \node at (0,-0.7) {$2$};
  \node at (2,-0.7) {$2$};
  \node at (4,-0.7) {$4$};
  \node[vertex] (3bis) at (4, 1.5) {$v_3'$};
  \node at (4, 2.2) {$3$};
  \node at (6,-0.7) {$3$};
  \node at (8,-0.7) {$2$};
  \draw[-] (1) -- (2) -- (3) -- (4) -- (5) (3) -- (3bis);
\end{tikzpicture}
}
}
\hspace*{0.3cm}
\subfloat[]
{
\label{subfig:config7}
\scalebox{0.7}{
\begin{tikzpicture}[join=bevel,vertex/.style={circle,draw, minimum size=0.6cm},inner sep=0mm,scale=0.85]
  \foreach \i / \l in {0/1,2/2,4/3,6/4}{
  	\node[vertex] (\l) at (\i,0) {$v_\l$};
  }
  \node at (0,-0.7) {$2$};
  \node at (2,-0.7) {$4$};
  \node[vertex] (2bis) at (2, 1.5) {$v_2'$};
  \node at (2, 2.2) {$3$};
  \node at (4,-0.7) {$3$};
  \node[vertex] (3bis) at (4, 1.5) {$v_3'$};
  \node at (4, 2.2) {$3$};
  \node at (6,-0.7) {$2$};
  \draw[-] (1) -- (2) -- (3) -- (4) (2) -- (2bis) (3) -- (3bis);
\end{tikzpicture}
}
}

\hspace*{-0.5cm}
\subfloat[]
{
\label{subfig:config8}
\scalebox{0.7}{
\begin{tikzpicture}[join=bevel,vertex/.style={circle,draw, minimum size=0.6cm},inner sep=0mm,scale=0.85]
  \foreach \i / \l in {0/1,2/2,4/3,6/4,8/5}{
  	\node[vertex] (\l) at (\i,0) {$v_\l$};
  }
  \node at (0,-0.7) {$2$};
  \node at (2,-0.7) {$4$};
  \node[vertex] (2bis) at (2, 1.5) {$v_2'$};
  \node at (2, 2.2) {$3$};
  \node at (4,-0.7) {$4$};
  \node[vertex] (3bis) at (4, 1.5) {$v_3'$};
  \node at (4, 2.2) {$3$};
  \node at (6,-0.7) {$2$};
  \node at (8, -0.7) {$2$};
  \draw[-] (1) -- (2) -- (3) -- (4) -- (5) (2) -- (2bis) (3) -- (3bis);
\end{tikzpicture}
}
}
\hspace*{0.3cm}
\subfloat[]
{
\label{subfig:config9}
\scalebox{0.7}{
\begin{tikzpicture}[join=bevel,vertex/.style={circle,draw, minimum size=0.6cm},inner sep=0mm,scale=0.85]
  \foreach \i / \l in {0/1,2/2,4/3,6/4,8/5,10/6}{
  	\node[vertex] (\l) at (\i,0) {$v_\l$};
  }
  \node at (0,-0.7) {$2$};
  \node at (2,-0.7) {$3$};
  \node at (4,-0.7) {$4$};
  \node[vertex] (3bis) at (4, 1.5) {$v_3'$};
  \node at (4, 2.2) {$3$};
  \node at (6,-0.7) {$4$};
  \node[vertex] (4bis) at (6, 1.5) {$v_4'$};
  \node at (6, 2.2) {$3$};
  \node at (8, -0.7) {$2$};
  \node at (10, -0.7) {$2$};
  \draw[-] (1) -- (2) -- (3) -- (4) -- (5) -- (6) (4) -- (4bis) (3) -- (3bis);
\end{tikzpicture}
}
}
\hspace*{0.3cm}
\subfloat[]
{
\label{subfig:config10}
\scalebox{0.7}{
\begin{tikzpicture}[join=bevel,vertex/.style={circle,draw, minimum size=0.6cm},inner sep=0mm,scale=0.85]
  \foreach \i / \l in {0/1,2/2,4/3,6/4,8/5}{
  	\node[vertex] (\l) at (\i,0) {$v_\l$};
  }
  \node at (0,-0.7) {$2$};
  \node at (2,-0.7) {$3$};
  \node at (4,-0.7) {$5$};
  \node[vertex] (3bis) at (4, 1.5) {$v_3'$};
  \node at (3, 1.5) {$2$};
  \node[vertex] (3ter) at (4, 3) {$v_3''$};
  \node at (4, 3.7) {$2$};
  \node at (6,-0.7) {$4$};
  \node[vertex] (4bis) at (6, 1.5) {$v_4'$};
  \node at (6, 2.2) {$3$};
  \node at (8,-0.7) {$2$};
  \draw[-] (1) -- (2) -- (3) -- (4) -- (5) (3) -- (3bis) -- (3ter) (4) -- (4bis);
\end{tikzpicture}
}
}

\hspace*{-0.5cm}
\subfloat[]
{
\label{subfig:config11}
\scalebox{0.75}{
\begin{tikzpicture}[join=bevel,vertex/.style={circle,draw, minimum size=0.6cm},inner sep=0mm,scale=0.85]
  \foreach \i / \l in {0/1,2/2,4/3,6/4,8/5,10/6}{
  	\node[vertex] (\l) at (\i,0) {$v_\l$};
  }
  \node at (0,-0.7) {$2$};
  \node at (2,-0.7) {$2$};
  \node at (4,-0.7) {$5$};
  \node at (6,-0.7) {$4$};
  \node[vertex] (4bis) at (6, 1.5) {$v_4'$};
  \node at (6, 2.2) {$2$};
  \node at (8,-0.7) {$2$};
  \node at (10,-0.7) {$2$};
  \draw[-] (1) -- (2) -- (3) -- (4) -- (5) -- (6) (4) -- (4bis);
\end{tikzpicture}
}
}\hspace*{0.5cm}
\subfloat[]
{
\label{subfig:config12}
\scalebox{0.75}{
\begin{tikzpicture}[join=bevel,vertex/.style={circle,draw, minimum size=0.6cm},inner sep=0mm,scale=0.85]
  \foreach \i / \l in {0/1,2/2,4/3,6/4,8/5,10/6}{
  	\node[vertex] (\l) at (\i,0) {$v_\l$};
  }
  \node at (0,-0.7) {$2$};
  \node at (2,-0.7) {$2$};
  \node at (4,-0.7) {$4$};
  \node at (6,-0.7) {$4$};
  \node[vertex] (4bis) at (6, 1.5) {$v_4'$};
  \node at (6, 2.2) {$3$};
  \node at (8,-0.7) {$2$};
  \node at (10,-0.7) {$2$};
  \draw[-] (1) -- (2) -- (3) -- (4) -- (5) -- (6)  (4) -- (4bis);
\end{tikzpicture}
}
}

\subfloat[]
{
\label{subfig:config13}
\scalebox{0.75}{
\begin{tikzpicture}[join=bevel,vertex/.style={circle,draw, minimum size=0.6cm},inner sep=0mm,scale=0.85]
  \foreach \i / \l in {0/1,2/2,4/3,6/4,8/5,10/6}{
  	\node[vertex] (\l) at (\i,0) {$v_\l$};
  }
  \node at (0,-0.7) {$2$};
  \node at (2,-0.7) {$2$};
  \node at (4,-0.7) {$3$};
  \node at (6,-0.7) {$4$};
  \node at (8,-0.7) {$2$};
  \node at (10,-0.7) {$2$};
  \draw[-] (1) -- (2) -- (3) -- (4) -- (5) -- (6);
\end{tikzpicture}
}
}\hspace*{0.5cm}
\subfloat[]
{
\label{subfig:config14}
\scalebox{0.75}{
\begin{tikzpicture}[join=bevel,vertex/.style={circle,draw, minimum size=0.6cm},inner sep=0mm,scale=0.85]
  \foreach \i / \l in {0/1,2/2,4/3,6/4,8/5,10/6,12/7}{
  	\node[vertex] (\l) at (\i,0) {$v_\l$};
  }
  \node at (0,-0.7) {$2$};
  \node at (2,-0.7) {$2$};
  \node at (4,-0.7) {$3$};
  \node at (6,-0.7) {$3$};
  \node at (8,-0.7) {$3$};
  \node at (10,-0.7) {$3$};
  \node at (12,-0.7) {$2$};
  \draw[-] (1) -- (2) -- (3) -- (4) -- (5) -- (6) -- (7);
\end{tikzpicture}
}
}

\centering
\subfloat[]
{
\label{subfig:config15}
\scalebox{0.75}{
\begin{tikzpicture}[join=bevel,vertex/.style={circle,draw, minimum size=0.6cm},inner sep=0mm,scale=0.85]
  \foreach \i / \l in {0/1,2/2,4/3,6/4,8/5,10/6,12/7}{
  	\node[vertex] (\l) at (\i,0) {$v_\l$};
  }
  \node at (0,-0.7) {$2$};
  \node at (2,-0.7) {$2$};
  \node at (4,-0.7) {$4$};
  \node at (6,-0.7) {$3$};
  \node at (8,-0.7) {$3$};
  \node at (10,-0.7) {$2$};
  \node at (12,-0.7) {$2$};
  \draw[-] (1) -- (2) -- (3) -- (4) -- (5) -- (6) -- (7);
\end{tikzpicture}
}
}

\centering
\subfloat[]
{
\label{subfig:config16}
\scalebox{0.75}{
\begin{tikzpicture}[join=bevel,vertex/.style={circle,draw, minimum size=0.6cm},inner sep=0mm,scale=0.8]
 \foreach[count=\l from 1] \i in {0,2,...,10}{
  	\node[vertex] (\l) at (\i,0) {$v_\l$};
  }
  \node at (0,-0.7) {$2$};
  \node at (2,-0.7) {$4$};
  \node at (4,-0.7) {$6$};
  \node at (6,-0.7) {$4$};
  \node at (8,-0.7) {$2$};
  \node at (10,-0.7) {$2$};
  
  \node[vertex] (v'2) at (2,2) {$v'_2$};
  \node at (1.2,2.4) {$3$};
  \node[vertex] (v'4) at (6,2) {$v'_4$};
  \node at (5.2,2.3) {$3$};
  
  \node[vertex] (v'3) at (4,2) {$v'_3$};
  \node at (3.1,2) {$4$};
  \node[vertex] (v''3) at (3,4) {$v''_3$};
  \node at (2.2,4) {$2$};
  \node[vertex] (v'''3) at (5,4) {$v'''_3$};
  \node at (4.2,4) {$3$};

  \draw[-] (1) -- (2) -- (3) -- (4) -- (5) -- (6);
  \draw[-] (2) -- (v'2);
  \draw[-] (3) -- (v'3) -- (v''3) (v'3) -- (v'''3);
  \draw[-] (4) -- (v'4);
\end{tikzpicture}
}
}
\caption[]{Useful 2-distance colorable configurations (\Cref{lemma:observations})}
\end{figure}
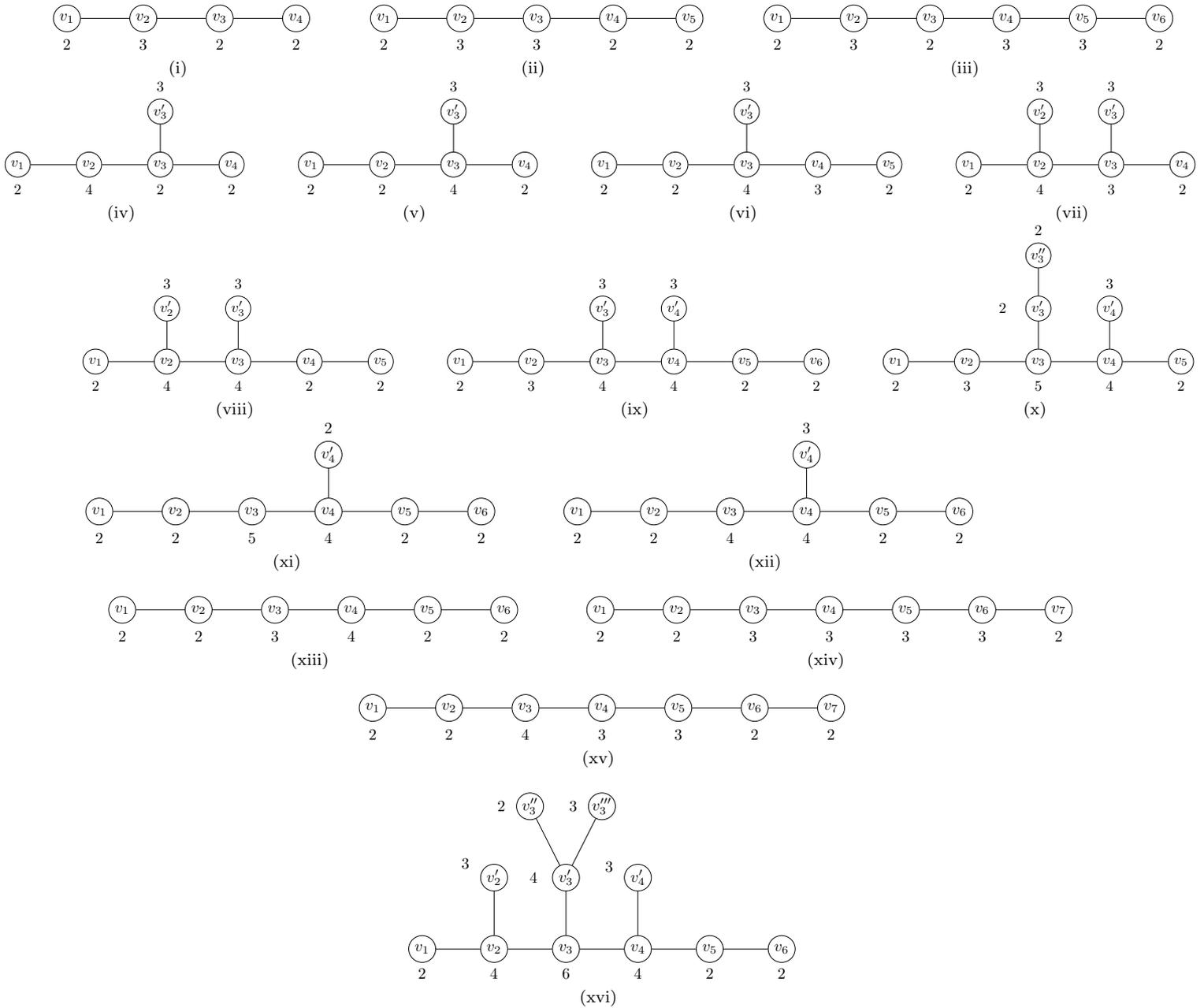

In \Crefrange{fig:forced_non_colorable_3path}{fig:forced_non_colorable_P5} we provide several useful non-colorable configurations. The important fact is that the non-colorable configurations can force the lists of colors on some vertices. 

\begin{lemma}
The graphs depicted in \Crefrange{fig:forced_non_colorable_3path}{fig:forced_non_colorable_P5} are 2-distance colorable unless their lists of available colors are exactly as indicated. 
\end{lemma}

\begin{figure}[!ht]
\centering

\subfloat[Initial configurations]
{
\label{subfig:forced_non_colorable_3path1} 
\scalebox{0.75}{
\begin{tikzpicture}[join=bevel,vertex/.style={circle,draw, minimum size=0.6cm},inner sep=0mm]
  \foreach \i / \l in {0/1,2/2,4/3}{
  	\node[vertex] (\l) at (\i,0) {$v_\l$};
  }
  \node at (0,-0.7) {$1$};
  \node at (2,-0.7) {$2$};
  \node at (4,-0.7) {$2$};
  \draw[-] (1) -- (2) -- (3);
  
  \foreach \i / \l in {0/1,2/2,4/3}{
  	\node[vertex] (\l) at (\i,-2) {$v_\l$};
  }
  \node at (0,-2.7) {$2$};
  \node at (2,-2.7) {$1$};
  \node at (4,-2.7) {$2$};
  \draw[-] (1) -- (2) -- (3);
\end{tikzpicture}
}
}
 \hspace*{1cm}
\subfloat[Forced lists of colors]
{
\label{subfig:forced_non_colorable_3path2}                                                  
\scalebox{0.75}{
\begin{tikzpicture}[join=bevel,vertex/.style={circle,draw, minimum size=0.6cm},inner sep=0mm]
  \foreach \i / \l in {0/1,2/2,4/3}{
  	\node[vertex] (\l) at (\i,0) {$v_\l$};
  }
  \node at (0,-0.7) {$L(v_1)\subseteq \{a,b\}$};
  \node at (2,-0.7) {$\{a,b\}$};
  \node at (4,-0.7) {$\{a,b\}$};
  \draw[-] (1) -- (2) -- (3);

  \foreach \i / \l in {0/1,2/2,4/3}{
  	\node[vertex] (\l) at (\i,-2) {$v_\l$};
  }
  \node at (0,-2.7) {$\{a,b\}$};
  \node at (2,-2.7) {$L(v_2)\subseteq \{a,b\}$};
  \node at (4,-2.7) {$\{a,b\}$};
  \draw[-] (1) -- (2) -- (3);
\end{tikzpicture}
}
}
\caption{A non-colorable path on 3 vertices}
\label{fig:forced_non_colorable_3path}
\end{figure}
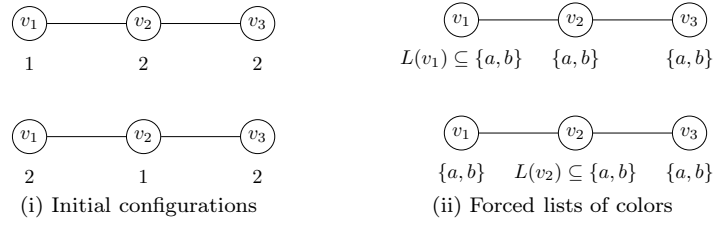

\begin{figure}[!ht]
\centering

\subfloat[Initial configuration]
{
\label{subfig:forced_non_colorable_claw1} 
\scalebox{0.75}{
\begin{tikzpicture}[join=bevel,vertex/.style={circle,draw, minimum size=0.6cm},inner sep=0mm,scale=0.9]
  \foreach \i / \l in {0/2,3/3,6/4}{
  	\node[vertex] (\l) at (\i,0) {$v_\l$};
  }
  \node[vertex] (1) at (3,1.5) {$v_1$};
  \node at (3,2.2) {$2$};
  \node at (0,-0.7) {$3$};
  \node at (3,-0.7) {$2$};
  \node at (6,-0.7) {$3$};
  \draw[-] (1) -- (3) (2) -- (3) -- (4);
\end{tikzpicture}
}
}
 \hspace*{0.3cm}
\subfloat[Forced lists of colors]
{
\label{subfig:forced_non_colorable_claw2}                                                  
\scalebox{0.75}{
\begin{tikzpicture}[join=bevel,vertex/.style={circle,draw, minimum size=0.6cm},inner sep=0mm,scale=0.9]
  \foreach \i / \l in {0/2,3/3,6/4}{
  	\node[vertex] (\l) at (\i,0) {$v_\l$};
  }
  \node[vertex] (1) at (3,1.5) {$v_1$};
  \node at (3,2.2) {$L(v_1)\subseteq \{a,b,c\}$};
  \node at (0,-0.7) {$\{a,b,c\}$};
  \node at (3,-0.7) {$L(v_3)\subseteq \{a,b,c\}$};
  \node at (6,-0.7) {$\{a,b,c\}$};
  \draw[-] (1) -- (3) (2) -- (3) -- (4);
\end{tikzpicture}
}
}
\caption{A non-colorable graph}
\label{fig:forced_non_colorable_claw}
\end{figure}

\bigskip

\begin{figure}[!ht]
\centering

\subfloat[Initial configuration]
{
\label{subfig:forced_non_colorable_clawpath1} 
\scalebox{0.75}{
\begin{tikzpicture}[join=bevel,vertex/.style={circle,draw, minimum size=0.6cm},inner sep=0mm,scale=0.9]
  \foreach \i / \l in {0/2,3/3,6/4,9/5}{
  	\node[vertex] (\l) at (\i,0) {$v_\l$};
  }
  \node[vertex] (1) at (3,1.5) {$v_1$};
  \node at (3,2.2) {$2$};
  \node at (0,-0.7) {$3$};
  \node at (3,-0.7) {$3$};
  \node at (6,-0.7) {$2$};
  \node at (9,-0.7) {$2$};
  \draw[-] (1) -- (3) (2) -- (3) -- (4) -- (5);
\end{tikzpicture}
}
}
 \hspace*{0.3cm}
\subfloat[Forced lists of colors]
{
\label{subfig:forced_non_colorable_clawpath2}                                                  
\scalebox{0.75}{
\begin{tikzpicture}[join=bevel,vertex/.style={circle,draw, minimum size=0.6cm},inner sep=0mm,scale=0.9]
  \foreach \i / \l in {0/2,3/3,6/4,9/5}{
  	\node[vertex] (\l) at (\i,0) {$v_\l$};
  }
  \node[vertex] (1) at (3,1.5) {$v_1$};
  \node at (3,2.2) {$L(v_1)\subseteq\{a,b,c\}$};
  \node at (0,-0.7) {$\{a,b,c\}$};
  \node at (3,-0.7) {$\{a,b,c\}$};
  \node at (6,-0.7) {$L(v_4)\subseteq \{a,b,c\}$};
  \draw[-] (1) -- (3) (2) -- (3) -- (4) -- (5);
\end{tikzpicture}
}
}
\caption{A non-colorable graph}
\label{fig:forced_non_colorable_clawpath}
\end{figure}

\bigskip

\begin{figure}[!ht]
\centering

\subfloat[Initial configuration]
{
\label{subfig:forced_non_colorable_P5_1}
\scalebox{0.7}{
\begin{tikzpicture}[join=bevel,vertex/.style={circle,draw, minimum size=0.6cm},inner sep=0mm,scale=0.85]
  \foreach \i / \l in {0/1,2/2,4/3,6/4,8/5}{
  	\node[vertex] (\l) at (\i,0) {$v_\l$};
  }
  \node at (0,-0.7) {$2$};
  \node at (2,-0.7) {$2$};
  \node at (4,-0.7) {$4$};
  \node at (6,-0.7) {$2$};
  \node at (8, -0.7) {$2$};
  \draw[-] (1) -- (2) -- (3) -- (4) -- (5);
\end{tikzpicture}
}
}
 \hspace*{1cm}
\subfloat[Forced lists of colors]
{
\label{subfig:forced_non_colorable_P5_2}
\scalebox{0.7}{
\begin{tikzpicture}[join=bevel,vertex/.style={circle,draw, minimum size=0.6cm},inner sep=0mm,scale=0.85]
  \foreach \i / \l in {0/1,2/2,4/3,6/4,8/5}{
  	\node[vertex] (\l) at (\i,0) {$v_\l$};
  }
  \node at (0,-0.7) {$\{a,b\}$};

  \node at (2,-0.7) {$\{a,b\}$};
  \node at (4,-0.7) {$\{a,b,c,d\}$};
  \node at (6,-0.7) {$\{c,d\}$};
  \node at (8, -0.7) {$\{c,d\}$};
  \draw[-] (1) -- (2) -- (3) -- (4) -- (5);
\end{tikzpicture}
}
}
\caption{A non-colorable graph}
\label{fig:forced_non_colorable_P5}
\end{figure}
\bigskip

\figureproof{\Cref{fig:forced_non_colorable_3path}}
By Hall's Theorem, if $|L(v_1)\cup L(v_2)\cup L(v_3)|\geq 3$, then the graph is 2-distance colorable. Hence the forced lists in \Cref{subfig:forced_non_colorable_3path2} follow.
\hfill\smallqed\medskip

\figureproof{\Cref{fig:forced_non_colorable_claw}}
By Hall's Theorem, if $|L(v_1)\cup L(v_2)\cup L(v_3)\cup L(v_4)|\geq 4$, then the graph is 2-distance colorable. Hence the forced lists in \Cref{subfig:forced_non_colorable_claw2} follow.
\hfill\smallqed\medskip

\figureproof{\Cref{fig:forced_non_colorable_clawpath}}
First, observe that if $|L(v_1)|\geq 4$ or $|L(v_2)|\geq 4$, we can color the other vertices by \Cref{subfig:config1} and finish with $v_1$ or $v_2$ respectively. If $L(v_4)\geq 4$, then we obtain \Cref{subfig:config4}. Similarly, if $|L(v_3)|\geq 4$, then we obtain \Cref{subfig:config5}.

Also note that if $|L(v_5)|\geq 3$, then either $v_1$, $v_2$, $v_3$, $v_4$ can be colored and we color $v_5$ last. Or they cannot be colored and by \Cref{subfig:forced_non_colorable_claw2}, we have \Cref{subfig:forced_non_colorable_clawpath2}. 

We will show that if $v_1$, $v_2$, $v_3$, $v_4$ are colorable, then the whole configuration is colorable ($v_5$ included). Thus, they cannot be colored and by \Cref{fig:forced_non_colorable_claw} (since all four vertices see each other at distance two), we obtain \Cref{subfig:forced_non_colorable_clawpath2}.

So, let us assume that $v_1$, $v_2$, $v_3$, $v_4$ are colorable, in which case, $|L(v_1)\cup L(v_2)\cup L(v_3)\cup L(v_4)|\geq 4$ and $|L(v_5)|=2$.

If $L(v_5)\subseteq L(v_4)$, then we restrict $L(v_3)$ to $L(v_3)\setminus L(v_5)$ and  observe that $|L(v_1)\cup L(v_2)\cup(L(v_3)\setminus L(v_5))\cup L(v_4)| = |L(v_1)\cup L(v_2)\cup L(v_3)\cup L(v_4)|\geq 4$ since $L(v_5)\subseteq L(v_4)$. So, we can color $v_1$, $v_2$, $v_3$, $v_4$ and finish by coloring $v_5$.

If $L(v_5)\not\subseteq L(v_4)$, then we restrict $L(v_4)$ to $L(v_4)\setminus L(v_5)$. If $|L(v_1)\cup L(v_2)\cup L(v_3)\cup (L(v_4)\setminus L(v_5))|\geq 4$, then we can color $v_1$, $v_2$, $v_3$, $v_4$ and finish with $v_5$. Thus, $|L(v_1)\cup L(v_2)\cup L(v_3)\cup (L(v_4)\setminus L(v_5))| = 3$ and we can assume w.l.o.g. that $L(v_1)\subseteq L(v_2)=L(v_3)=\{a,b,c\}$ and $d\in L(v_4)\cap L(v_5)$. Now, it suffices to color $v_4$ with $d$, then color $v_5$, $v_1$, $v_3$, $v_2$ in this order.
\hfill\smallqed\medskip

\figureproof{\Cref{fig:forced_non_colorable_P5}}
First, observe that if $|L(v_1)|\geq 3$, then we can color the other vertices by \Cref{subfig:config1} and color $v_1$ last. If $|L(v_2)|\geq 3$, then we obtain \Cref{subfig:config2}. Symmetrically, the same holds for $L(v_4)$ and $L(v_5)$. If $|L(v_3)|\geq 5$, we can color $v_1$, $v_2$, $v_4$, $v_5$, $v_3$ in this order. 

Now, let us try to color the configuration. If $L(v_1)\neq L(v_2)$, then color $v_1$ with $a\notin L(v_2)$ and get \Cref{subfig:config1}. Therefore we have $L(v_1)=L(v_2)$ and symmetrically $L(v_4)=L(v_5)$. Finally, if $L(v_1)\cup L(v_5)\neq L(v_3)$, then one could color $v_3$ with $b\notin L(v_1)\cup L(v_5)$ and finish by coloring $v_1$, $v_2$, $v_4$, $v_5$ in this order. Hence the lists in \Cref{subfig:forced_non_colorable_P5_2} follow.
\hfill\smallqed\medskip

\begin{lemma}
\label{lemma:colorable_22422}
If there exists a coloring $\phi$ of the configuration from \Cref{subfig:forced_non_colorable_P5_1} where $\phi(v_1)\neq \phi (v_5)$, then there exists a coloring $\phi'$ such that $\phi(v_1)\neq \phi'(v_1)$ or $\phi(v_5)\neq \phi'(v_5)$.
\end{lemma}

\begin{proof}
Suppose that the configuration from \Cref{subfig:forced_non_colorable_P5_1} is colorable with $\phi$ where $\phi(v_1) = a$, $\phi(v_5) = b$ and $a\neq b$. Suppose by contradiction that for every coloring $\phi'$ of \Cref{subfig:forced_non_colorable_P5_1}, $\phi'(v_1)=a$ and $\phi'(v_5)=b$.

Let $L(v_1)=\{a,x\}$. We color $v_1$ with $x$. Since there exists no valid coloring $\phi'$ where $\phi'(v_1)=x$, the remaining configuration must not be colorable. So $x\in L(v_2)$, otherwise, we can color $v_2$, $v_3$, $v_4$, $v_5$ by \Cref{subfig:config1}. Let $L(v_2)=\{x,y\}$. Moreover, $x,y\in L(v_3)$. Otherwise, we color $v_1$ with $x$, $v_2$ with $y$ and finish by coloring $v_4$, $v_5$, $v_3$ in this order.

Symmetrically, the same holds for $v_5$. Let $L(v_5)=\{b,x'\}$, then we must have $L(v_4)=\{x',y'\}$ and $x',y'\in L(v_3)$.

Observe that when we color $v_1$ with $x$ and $v_2$ with $y$, the remaining configuration is not colorable so by \Cref{fig:forced_non_colorable_3path}, we must have $L(v_3)=\{x,y,b,x'\}$. Symmetrically, if instead we color $v_5$ with $x'$ and $v_4$ with $y'$, then we must have $L(v_3)=\{x',y',a,x\}$. We conclude that $\{x,x',b,y\} = \{x,x',a,y'\}$. In other words, $a=y$ and $b=y'$. Thus, we have $L(v_1)=L(v_2)=\{a,x\}$, $L(v_4)=L(v_5)=\{b,x'\}$ and $L(v_3)=\{a,x,b,x'\}$. By \Cref{fig:forced_non_colorable_3path}, we know that this configuration is not colorable, which is a contradiction as there exists a valid coloring $\phi$.
\end{proof}

\section{Structural properties of a minimal counterexample}
\label{sec:structural_properties}
Let $G$ be a counterexample to \Cref{thm:girth8} with the minimum number of vertices. We show some properties of $G$.

\begin{lemma}
\label{lemma:connected}
Graph $G$ is connected.
\end{lemma}

\begin{proof}
If $G$ is not connected, then we consider one of its connected component that is not 2-distance colorable (which exists since $G$ is a counterexample to \Cref{thm:girth8}). This component is also a planar subcubic graph with girth at least 8 that is a counterexample to \Cref{thm:girth8}, which contradicts $G$'s minimality.
\end{proof}

\begin{lemma}
\label{lemma:1vertex}
Graph $G$ has minimum degree at least 2. 
\end{lemma}

\begin{proof}
If $G$ has a $0$-vertex, since $G$ is connected, it is a single vertex which is colorable. Assume by contradiction that $G$ has a 1-vertex $v$. We remove such vertex and 2-distance color the resulting graph which is possible due to the minimality of $G$. Then, we add the vertex back then choose a color for $v$ different from all of its 2-distance neighbors' as $v$ has at most 3 neighbors at distance 2 and we have 6 colors.
\end{proof}

By \Cref{lemma:connected} and \Cref{lemma:1vertex}, the graph $G$ has only 2-vertices and 3-vertices.

\begin{lemma}
\label{lemma:no 2-path}
Graph $G$ has no $k$-path with $k\geq 2$.
\end{lemma}

\begin{proof}
Assume by contradiction that $G$ has a $k$-path with $k\geq 2$. We remove the $2$-vertices of this path and color the resulting graph. One can easily see that such coloring is greedily extendable to the removed $2$-vertices.
\end{proof}

\medskip

In what follows we show a set of subgraphs of $G$ that are \emph{reducible}, that is none of these subgraphs can appear in $G$ as otherwise it would contradict the choice of $G$. All these configurations are depicted in \Cref{fig:reducible}, \Cref{fig:special_reducible} and \Cref{fig:reducible_cycles}. In order to simplify the reading of the paper, the captions of the corresponding configurations of these figures will be explained later in \Cref{subsec:second_round_discharging} as they are not used in this section. 
In each of the sub-figures, we define $S$ as the set of all vertices labeled $v_i$, $v'_i$, $v''_i$ or $v'''_i$, where $i$ is a positive integer. The degree of these vertices are given by their incident edges. In order to prove the reducibility of $S$ we consider a 2-distance coloring $\phi$ of $G-S$ (by induction hypothesis) and show how to extend $\phi$ to $G$ leading to a contradiction. In each figure, the number drawn next to a vertex of $S$ in the figure corresponds to the number of available colors in the precoloring extension of $G-S$.

Since $G$ has girth $g\geq 8$, one can easily observe that $G[S]^2=G^2[S]$ for each configuration in \Cref{fig:reducible}. In other words, there are no extra conflicts between vertices in $S$ than the conflicts in $G[S]$. Unlike the configurations of \Cref{fig:reducible}, in those of \Cref{fig:special_reducible}, some pair of vertices may see each other in $G$ while they are at distance at least 3 in the subgraph induced by $S$, that is sometimes $G[S]^2\neq G^2[S]$. 

\begin{lemma}
\label{lemma:reducible}
Graph $G$ does not contain the configurations depicted in \Cref{fig:reducible}. 

\end{lemma}
\begin{proof}
We will give the proofs for each figure in order:

\figureproof{\Cref{subfig:1c1a1}}
Color arbitrarily vertex $v'_2$ and then get \Cref{subfig:config2}.
\hfill\smallqed\medskip

\figureproof{\Cref{subfig:1c0c0a1}}
Direct implication of \Cref{subfig:config8}.
\hfill\smallqed\medskip

\figureproof{\Cref{subfig:1c0c1}}
Direct implication of \Cref{subfig:config7}.
\hfill\smallqed\medskip

\figureproof{\Cref{subfig:1a1a0c1}}
To prove this configuration, we redefine the set $S$ to be $\{v_1,v_2,v_3\}$. Consider a 2-distance coloring $\phi$ of $G-S$. If $\phi$ is extendable to $G$, then we are done. Thus the available colors of vertices in $S$ correspond to \Cref{fig:forced_non_colorable_3path}. More precisely, $L(v_2)\subseteq L(v_1)=L(v_3)=\{a,b\}$. Now, uncolor vertices $v_4$, $v_5$, $v_6$ and $v'_5$ and observe that the numbers of available colors of the non-colored vertices of $G$ are the ones depicted in \Cref{subfig:1a1a0c1}. 

Without loss of generality we may assume that $\phi(v_4)=c$ and $\phi(v_5)=d$. Consequently, after the uncoloring of vertices $v_4$, $v_5$, $v_6$ and $v'_5$, we have $L(v_3)=\{a,b,c,d\}$ and $L(v_1)=\{a,b\}$. If we can choose a color $x\notin \{c,d\}$ for $v_4$ and color vertices $v_5$, $v_6$ and $v'_5$, then due to \Cref{fig:forced_non_colorable_3path}, we can finish the coloring of $v_1$, $v_2$ and $v_3$. Thus, $|L(v_5)|=3$ and the available colors for $v_5$, $v_6$ and $v'_5$ are $\{x,y,z\}\in\{a,b,c,d,e,f\}$ (again due to \Cref{fig:forced_non_colorable_3path}). Note that $\phi(v_4)=c\notin\{x,y,z\}$, otherwise $\phi$ would not be a valid coloring of $G-S$. We can assume w.l.o.g that $x\neq d$ and we color $v_4$, $v_5$, $v_6$, $v'_5$ with $c$, $x$, $y$, $z$ respectively. Finally, due to \Cref{fig:forced_non_colorable_3path} we can finish by coloring $v_1$, $v_2$, $v_3$ since the lists of available colors for $v_1$ and $v_3$ are not the same anymore.
\hfill\smallqed\medskip

\figureproof{\Cref{subfig:1a1c0a1}}
Direct implication of \Cref{subfig:config11}.
\hfill\smallqed\medskip

\figureproof{\Cref{subfig:1b0b0a1}}
Color $v'_3$ with a color $a\notin L(v''_3)$, and color $v_4$, $v_5$ in order. Then color vertices $v_1$, $v_2$, $v_3$, $v'_2$, $v''_2$, $v'''_2$ by \Cref{subfig:config7} and finish by coloring $v'''_3$ and $v''_3$ in this order.
\hfill\smallqed\medskip

\figureproof{\Cref{subfig:1c0b0c1} and \Cref{subfig:1c0b0c0c}}
Direct implication of \Cref{subfig:config16} for \Cref{subfig:1c0b0c0c}. As for \Cref{subfig:1c0b0c1}, it suffices to see that by adding an imaginary vertex $v_6$ adjacent to $v_5$ with any list of colors that verifies $|L(v_6)|\geq 2$, \Cref{subfig:config16} gives us a valid coloring for vertices of \Cref{subfig:1c0b0c1}.
\hfill\smallqed\medskip

\end{proof}

\begin{figure}[H]
\centering
\subfloat[\texttt{1c1a1}, \texttt{1c1c}]
{
\label{subfig:1c1a1}
\scalebox{0.75}{%
\begin{tikzpicture}[join=bevel,vertex/.style={circle,draw, minimum size=0.6cm},inner sep=0mm,scale=0.8]
 \foreach[count=\l from 1] \i in {0,2,...,8}{
  	\node[vertex] (\l) at (\i,0) {$v_\l$};
  }
  \node (0) at (-1,0) {};
  \node at (0,-0.7) {$3$};
  \node at (2,-0.7) {$4$};
  \node at (4,-0.7) {$5$};
  \node at (6,-0.7) {$2$};
  \node at (8,-0.7) {$2$};
  \node (6) at (9,0) {};
  \node[vertex] (v'2) at (2,2) {$v'_2$};
  \node at (1.2,2) {$3$};
  \node (v''2) at (2,3) {};
  \node (v'4) at (6,1) {};

  \draw[-] (0) -- (1) -- (2) -- (3) -- (4) -- (5) -- (6);
  \draw[-] (2) -- (v'2) -- (v''2)  (4) -- (v'4);
\end{tikzpicture}
}
}
\hspace*{0.3cm}
\subfloat[\texttt{1c0c0a1}, \texttt{1c0c0c}, \texttt{1a0b1}, \texttt{1b0c}]
{
\label{subfig:1c0c0a1}
\scalebox{0.75}{%
\begin{tikzpicture}[join=bevel,vertex/.style={circle,draw, minimum size=0.6cm},inner sep=0mm,scale=0.8]
 \foreach[count=\l from 1] \i in {0,2,...,8}{
  	\node[vertex] (\l) at (\i,0) {$v_\l$};
  }
  \node (0) at (-1,0) {};
  \node at (0,-0.7) {$3$};
  \node at (2,-0.7) {$4$};
  \node at (4,-0.7) {$4$};
  \node[vertex] (v'3) at (4,2) {$v'_3$};
  \node at (3.2,2) {$3$};
  \node (v''3) at (4,3) {};
  \node at (6,-0.7) {$2$};
  \node at (8,-0.7) {$2$};
  \node (6) at (9,0) {};
  \node[vertex] (v'2) at (2,2) {$v'_2$};
  \node at (1.2,2) {$3$};
  \node (v''2) at (2,3) {};
  \node (v'4) at (6,1) {};

  \draw[-] (0) -- (1) -- (2) -- (3) -- (4) -- (5) -- (6);
  \draw[-] (2) -- (v'2) -- (v''2)  (3) -- (v'3) -- (v''3) (4) -- (v'4);
\end{tikzpicture}
}
}

\subfloat[\texttt{1c0c1}, \texttt{1b1}]
{
\label{subfig:1c0c1}
\scalebox{0.75}{%
\begin{tikzpicture}[join=bevel,vertex/.style={circle,draw, minimum size=0.6cm},inner sep=0mm,scale=0.8]
 \foreach[count=\l from 1] \i in {0,2,...,6}{
  	\node[vertex] (\l) at (\i,0) {$v_\l$};
  }
  \node (0) at (-1,0) {};
  \node at (0,-0.7) {$3$};
  \node at (2,-0.7) {$4$};
  \node at (4,-0.7) {$4$};
  \node[vertex] (v'3) at (4,2) {$v'_3$};
  \node at (3.2,2) {$3$};
  \node (v''3) at (4,3) {};
  \node at (6,-0.7) {$3$};
  \node (5) at (7, 0) {};
  \node[vertex] (v'2) at (2,2) {$v'_2$};
  \node at (1.2,2) {$3$};
  \node (v''2) at (2,3) {};

  \draw[-] (0) -- (1) -- (2) -- (3) -- (4) -- (5);
  \draw[-] (2) -- (v'2) -- (v''2)  (3) -- (v'3) -- (v''3);
\end{tikzpicture}
}
}
\hspace*{0.3cm}
\subfloat[\texttt{1a1a0c1}, \texttt{c1a0c1}, \texttt{1a1b}, \texttt{c1b}]
{
\label{subfig:1a1a0c1}
\scalebox{0.75}{%
\begin{tikzpicture}[join=bevel,vertex/.style={circle,draw, minimum size=0.6cm},inner sep=0mm,scale=0.8]
 \foreach[count=\l from 1] \i in {0,2,...,10}{
  	\node[vertex] (\l) at (\i,0) {$v_\l$};
  }
  \node (0) at (-1,0) {};
  \node at (0,-0.7) {$2$};
  \node at (2,-0.7) {$2$};
  \node at (4,-0.7) {$4$};
  \node at (6,-0.7) {$3$};
  \node at (8,-0.7) {$3$};
  \node at (10,-0.7) {$3$};
  \node (7) at (11,0) {};
  \node (v'2) at (2,1) {};
  \node (v'4) at (6,1) {};
  \node[vertex] (v'5) at (8,2) {$v'_5$};
  \node at (7.2,2) {$3$};
  \node (v''5) at (8,3) {};

  \draw[-] (0) -- (1) -- (2) -- (3) -- (4) -- (5) -- (6) -- (7);
  \draw[-] (2) -- (v'2) (4) -- (v'4) (5) -- (v'5) -- (v''5);
\end{tikzpicture}
}
}

\subfloat[\texttt{1a1c0a1}, \texttt{c1c0a1}, \texttt{1a1c0c}, \texttt{c1c0c}]
{
\label{subfig:1a1c0a1}
\scalebox{0.75}{%
\begin{tikzpicture}[join=bevel,vertex/.style={circle,draw, minimum size=0.6cm},inner sep=0mm,scale=0.8]
 \foreach[count=\l from 1] \i in {0,2,...,10}{
  	\node[vertex] (\l) at (\i,0) {$v_\l$};
  }
  \node (0) at (-1,0) {};
  \node at (0,-0.7) {$2$};
  \node at (2,-0.7) {$2$};
  \node at (4,-0.7) {$5$};
  \node at (6,-0.7) {$4$};
  \node at (8,-0.7) {$2$};
  \node at (10,-0.7) {$2$};
  \node (7) at (11,0) {};
  \node (v'2) at (2,1) {};
  \node[vertex] (v'4) at (6,2) {$v'_4$};
  \node at (5.2,2) {$3$};
  \node (v''4) at (6,3) {};
  \node (v'5) at (8,1) {};

  \draw[-] (0) -- (1) -- (2) -- (3) -- (4) -- (5) -- (6) -- (7);
  \draw[-] (2) -- (v'2) (4) -- (v'4) -- (v''4) (5) -- (v'5);
\end{tikzpicture}
}
}
\hspace*{0.3cm}
\subfloat[\texttt{1b0b0a1}, \texttt{1b0b0c}, \texttt{1c0c0b0a1}, \texttt{1c0c0b0c}]
{
\label{subfig:1b0b0a1}
\scalebox{0.75}{%
\begin{tikzpicture}[join=bevel,vertex/.style={circle,draw, minimum size=0.6cm},inner sep=0mm,scale=0.8]
 \foreach[count=\l from 1] \i in {0,2,...,8}{
  	\node[vertex] (\l) at (\i,0) {$v_\l$};
  }
  \node (0) at (-1,0) {};
  \node at (0,-0.7) {$3$};
  \node at (2,-0.7) {$5$};
  \node at (4,-0.7) {$5$};
  \node at (6,-0.7) {$2$};
  \node at (8,-0.7) {$2$};
  \node (6) at (9,0) {};
  
  \node[vertex] (v'2) at (2,2) {$v'_2$};
  \node at (1.2,2.4) {$4$};
  \node[vertex] (v''2) at (2,4) {$v''_2$};
  \node at (1.2,4) {$3$};
  \node (v''21) at (2,5) {};
  \node[vertex] (v'''2) at (0,2) {$v'''_2$};
  \node at (-0.8,2.3) {$3$};
  \node (v'''21) at (-1,2) {};

  \node[vertex] (v'3) at (4,2) {$v'_3$};
  \node at (3.1,2) {$4$};
  \node[vertex] (v''3) at (4,4) {$v''_3$};
  \node at (3.2,4) {$3$};
  \node (v''31) at (4,5) {};
  \node[vertex] (v'''3) at (6,2) {$v'''_3$};
  \node at (5.2,2.3) {$3$};
  \node (v'''31) at (7,2) {};
  
  \node (v'4) at (6,1) {};

  \draw[-] (0) -- (1) -- (2) -- (3) -- (4) -- (5) -- (6);
  \draw[-] (2) -- (v'2) -- (v''2) -- (v''21) (v'2) -- (v'''2) -- (v'''21);
  \draw[-] (3) -- (v'3) -- (v''3) -- (v''31) (v'3) -- (v'''3) -- (v'''31);
  \draw[-] (4) -- (v'4);
\end{tikzpicture}
}
}

\subfloat[\texttt{1c0b0c1}]
{
\label{subfig:1c0b0c1}
\scalebox{0.75}{%
\begin{tikzpicture}[join=bevel,vertex/.style={circle,draw, minimum size=0.6cm},inner sep=0mm,scale=0.8]
 \foreach[count=\l from 1] \i in {0,2,...,8}{
  	\node[vertex] (\l) at (\i,0) {$v_\l$};
  }
  \node (0) at (-1,0) {};
  \node at (0,-0.7) {$3$};
  \node at (2,-0.7) {$4$};
  \node at (4,-0.7) {$6$};
  \node at (6,-0.7) {$4$};
  \node at (8,-0.7) {$3$};
  \node (6) at (9,0) {};
  
  \node[vertex] (v'2) at (2,2) {$v'_2$};
  \node at (1.2,2.4) {$3$};
  \node (v''2) at (2,3) {};
  \node[vertex] (v'4) at (6,2) {$v'_4$};
  \node at (5.2,2.3) {$3$};
  \node (v''4) at (6,3) {};
  
  \node[vertex] (v'3) at (4,2) {$v'_3$};
  \node at (3.1,2) {$4$};
  \node[vertex] (v''3) at (3,4) {$v''_3$};
  \node at (2.2,4) {$3$};
  \node (v''31) at (2.5,5) {};
  \node[vertex] (v'''3) at (5,4) {$v'''_3$};
  \node at (4.2,4) {$3$};
  \node (v'''31) at (5.5,5) {};

  \draw[-] (0) -- (1) -- (2) -- (3) -- (4) -- (5) -- (6);
  \draw[-] (2) -- (v'2) -- (v''2);
  \draw[-] (3) -- (v'3) -- (v''3) -- (v''31) (v'3) -- (v'''3) -- (v'''31);
  \draw[-] (4) -- (v'4) -- (v''4);
\end{tikzpicture}
}
}
\hspace*{0.3cm}
\subfloat[\texttt{1c0b0c0a1}, \texttt{1c0b0c0c}]
{
\label{subfig:1c0b0c0c}
\scalebox{0.75}{%
\begin{tikzpicture}[join=bevel,vertex/.style={circle,draw, minimum size=0.6cm},inner sep=0mm,scale=0.8]
 \foreach[count=\l from 1] \i in {0,2,...,10}{
  	\node[vertex] (\l) at (\i,0) {$v_\l$};
  }
  \node (0) at (-1,0) {};
  \node at (0,-0.7) {$3$};
  \node at (2,-0.7) {$4$};
  \node at (4,-0.7) {$6$};
  \node at (6,-0.7) {$4$};
  \node at (8,-0.7) {$2$};
  \node at (10,-0.7) {$2$};
  \node (7) at (11,0) {};
  
  \node[vertex] (v'2) at (2,2) {$v'_2$};
  \node at (1.2,2.4) {$3$};
  \node (v''2) at (2,3) {};
  \node[vertex] (v'4) at (6,2) {$v'_4$};
  \node at (5.2,2.3) {$3$};
  \node (v''4) at (6,3) {};
  \node (v'5) at (8,1) {};
  
  \node[vertex] (v'3) at (4,2) {$v'_3$};
  \node at (3.1,2) {$4$};
  \node[vertex] (v''3) at (3,4) {$v''_3$};
  \node at (2.2,4) {$3$};
  \node (v''31) at (2.5,5) {};
  \node[vertex] (v'''3) at (5,4) {$v'''_3$};
  \node at (4.2,4) {$3$};
  \node (v'''31) at (5.5,5) {};

  \draw[-] (0) -- (1) -- (2) -- (3) -- (4) -- (5) -- (6) -- (7);
  \draw[-] (2) -- (v'2) -- (v''2);
  \draw[-] (3) -- (v'3) -- (v''3) -- (v''31) (v'3) -- (v'''3) -- (v'''31);
  \draw[-] (4) -- (v'4) -- (v''4) (5) -- (v'5);
\end{tikzpicture}
}
}
\caption{Reducible configurations (\Cref{lemma:reducible}).}
\label{fig:reducible}
\end{figure}

\begin{lemma}
\label{lemma:first_reducible_cycles}
Graph $G$ does not contain the configurations depicted in \Cref{fig:first_reducible_cycles}. 
\end{lemma}

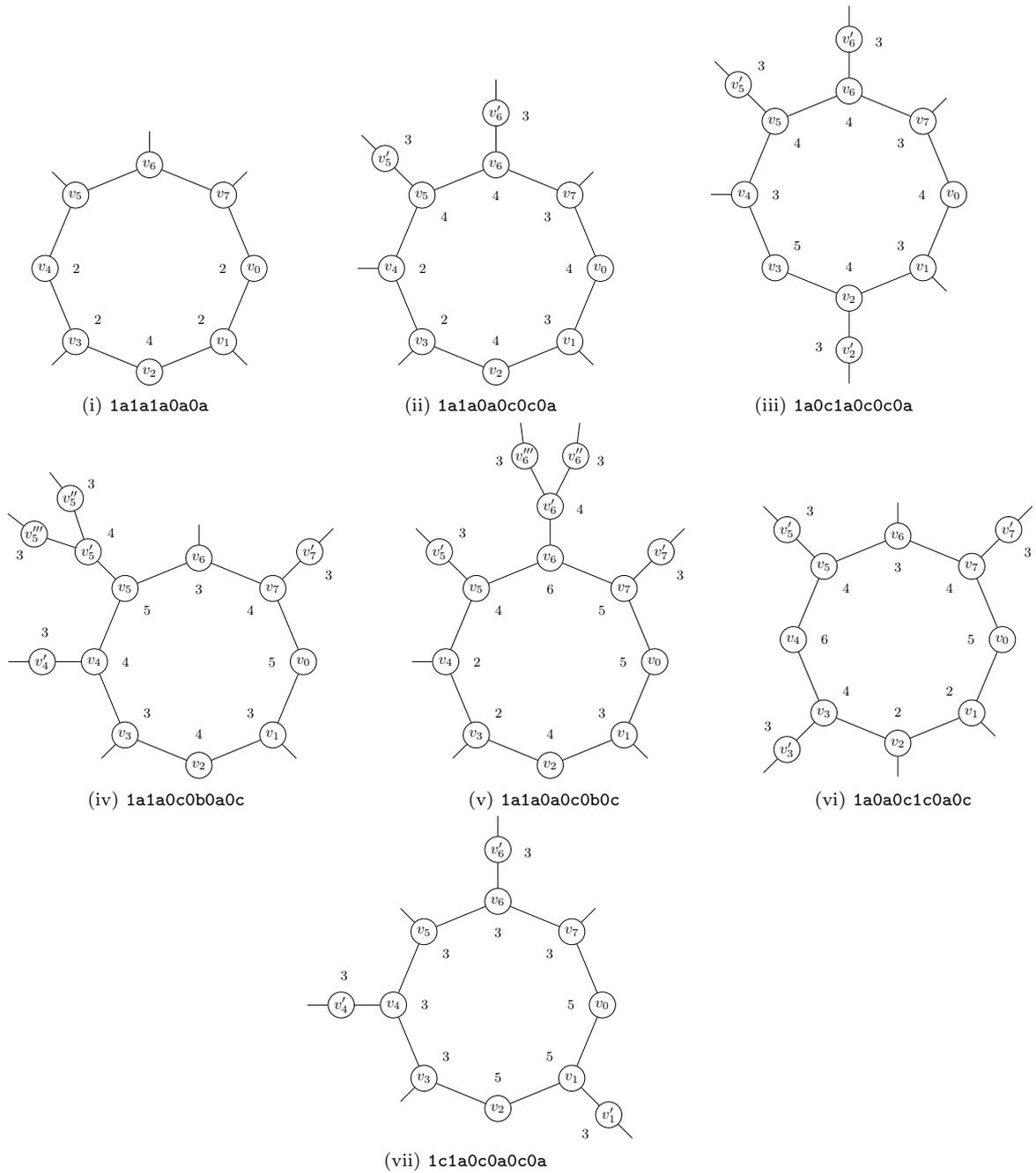
\begin{figure}
\centering
\subfloat[\texttt{1a1a1a0a0a}]{
\label{subfig:1a1a1a0a0a}
\scalebox{0.65}{
	\begin{tikzpicture}[join=bevel,vertex/.style={circle,draw, minimum size=0.6cm},inner sep=0mm,scale=0.8]
\foreach \i in {0,...,7}  \node[vertex] (\i) at (-45.0*\i:3) {$v_\i$};
\foreach \i in {0,...,7}  \draw let \n1={int(mod(\i+1,8))} in (\i) -- (\n1);
\node at (-45.0*0:2.1) {\small $2$};
\node at (-45.0*2:2.1) {\small $4$};
\node at (-45.0*4:2.1) {\small $2$};
\node (u1) at (-45.0*1:4) {};
\draw (1) -- (u1);
\node at (-45.0*1:2.1) {\small $2$};
\node (u3) at (-45.0*3:4) {};
\draw (3) -- (u3);
\node at (-45.0*3:2.1) {\small $2$};
\node (u5) at (-45.0*5:4) {};
\draw (5) -- (u5);
\node at (-45.0*5:2.1) {};
\node (u6) at (-45.0*6:4) {};
\draw (6) -- (u6);
\node at (-45.0*6:2.1) {};
\node (u7) at (-45.0*7:4) {};
\draw (7) -- (u7);
\node at (-45.0*7:2.1) {};
\end{tikzpicture}
}
}
\hspace*{0.7cm}
\subfloat[\texttt{1a1a0a0c0c0a}]{
\label{subfig:1a1a0a0c0c0a}
\scalebox{0.65}{
	\begin{tikzpicture}[join=bevel,vertex/.style={circle,draw, minimum size=0.6cm},inner sep=0mm,scale=0.8]
\foreach \i in {0,...,7}  \node[vertex] (\i) at (-45.0*\i:3) {$v_\i$};
\foreach \i in {0,...,7}  \draw let \n1={int(mod(\i+1,8))} in (\i) -- (\n1);
\node at (-45.0*0:2.1) {\small $4$};
\node at (-45.0*2:2.1) {\small $4$};
\node (u1) at (-45.0*1:4) {};
\draw (1) -- (u1);
\node at (-45.0*1:2.1) {\small $3$};
\node (u3) at (-45.0*3:4) {};
\draw (3) -- (u3);
\node at (-45.0*3:2.1) {\small $2$};
\node (u4) at (-45.0*4:4) {};
\draw (4) -- (u4);
\node at (-45.0*4:2.1) {\small $2$};
\node[vertex] (v'5) at (-45.0*5:4.5) {$v'_5$};
\node (u'5) at (-45.0*5:5.5) {};
\node at (-45.0*5-11:4.5) {\small $3$};
\draw (5) -- (v'5);
\draw (v'5) -- (u'5);
\node at (-45.0*5:2.1) {\small $4$};
\node[vertex] (v'6) at (-45.0*6:4.5) {$v'_6$};
\node (u'6) at (-45.0*6:5.5) {};
\node at (-45.0*6-11:4.5) {\small $3$};
\draw (6) -- (v'6);
\draw (v'6) -- (u'6);
\node at (-45.0*6:2.1) {\small $4$};
\node (u7) at (-45.0*7:4) {};
\draw (7) -- (u7);
\node at (-45.0*7:2.1) {\small $3$};
\end{tikzpicture}
}
}
\hspace*{0.8cm}
\subfloat[\texttt{1a0c1a0c0c0a}]{
\label{subfig:1a0c1a0c0c0a}
\scalebox{0.65}{
	\begin{tikzpicture}[join=bevel,vertex/.style={circle,draw, minimum size=0.6cm},inner sep=0mm,scale=0.8]
\foreach \i in {0,...,7}  \node[vertex] (\i) at (-45.0*\i:3) {$v_\i$};
\foreach \i in {0,...,7}  \draw let \n1={int(mod(\i+1,8))} in (\i) -- (\n1);
\node at (-45.0*0:2.1) {\small $4$};
\node at (-45.0*3:2.1) {\small $5$};
\node (u1) at (-45.0*1:4) {};
\draw (1) -- (u1);
\node at (-45.0*1:2.1) {\small $3$};
\node[vertex] (v'2) at (-45.0*2:4.5) {$v'_2$};
\node (u'2) at (-45.0*2:5.5) {};
\node at (-45.0*2-11:4.5) {\small $3$};
\draw (2) -- (v'2);
\draw (v'2) -- (u'2);
\node at (-45.0*2:2.1) {\small $4$};
\node (u4) at (-45.0*4:4) {};
\draw (4) -- (u4);
\node at (-45.0*4:2.1) {\small $3$};
\node[vertex] (v'5) at (-45.0*5:4.5) {$v'_5$};
\node (u'5) at (-45.0*5:5.5) {};
\node at (-45.0*5-11:4.5) {\small $3$};
\draw (5) -- (v'5);
\draw (v'5) -- (u'5);
\node at (-45.0*5:2.1) {\small $4$};
\node[vertex] (v'6) at (-45.0*6:4.5) {$v'_6$};
\node (u'6) at (-45.0*6:5.5) {};
\node at (-45.0*6-11:4.5) {\small $3$};
\draw (6) -- (v'6);
\draw (v'6) -- (u'6);
\node at (-45.0*6:2.1) {\small $4$};
\node (u7) at (-45.0*7:4) {};
\draw (7) -- (u7);
\node at (-45.0*7:2.1) {\small $3$};
\end{tikzpicture}
}
}

\hspace*{0.5cm}
\subfloat[\texttt{1a1a0c0b0a0c}]{
\label{subfig:1a1a0c0b0a0c}
\scalebox{0.65}{
	\begin{tikzpicture}[join=bevel,vertex/.style={circle,draw, minimum size=0.6cm},inner sep=0mm,scale=0.8]
\foreach \i in {0,...,7}  \node[vertex] (\i) at (-45.0*\i:3) {$v_\i$};
\foreach \i in {0,...,7}  \draw let \n1={int(mod(\i+1,8))} in (\i) -- (\n1);
\node at (-45.0*0:2.1) {\small $5$};
\node at (-45.0*2:2.1) {\small $4$};
\node (u1) at (-45.0*1:4) {};
\draw (1) -- (u1);
\node at (-45.0*1:2.1) {\small $3$};
\node (u3) at (-45.0*3:4) {};
\draw (3) -- (u3);
\node at (-45.0*3:2.1) {\small $3$};
\node[vertex] (v'4) at (-45.0*4:4.5) {$v'_4$};
\node (u'4) at (-45.0*4:5.5) {};
\node at (-45.0*4-11:4.5) {\small $3$};
\draw (4) -- (v'4);
\draw (v'4) -- (u'4);
\node at (-45.0*4:2.1) {\small $4$};
\node[vertex] (v'5) at (-45.0*5:4.5) {$v'_5$};
\node[vertex] (v''5) at (-45.0*5-7:6) {$v''_5$};
\node[vertex] (v'''5) at (-45.0*5+7:6) {$v'''_5$};
\node (u''5) at (-45.0*5-7:7) {};
\node (u'''5) at (-45.0*5+7:7) {};
\node at (-45.0*5-11:4.5) {\small $4$};
\node at (-45.0*5-14:6) {\small $3$};
\node at (-45.0*5+14:6) {\small $3$};
\draw (5) -- (v'5);
\draw (v'5) -- (v''5);
\draw (v'5) -- (v'''5);
\draw (v''5) -- (u''5);
\draw (v'''5) -- (u'''5);
\node at (-45.0*5:2.1) {\small $5$};
\node (u6) at (-45.0*6:4) {};
\draw (6) -- (u6);
\node at (-45.0*6:2.1) {\small $3$};
\node[vertex] (v'7) at (-45.0*7:4.5) {$v'_7$};
\node (u'7) at (-45.0*7:5.5) {};
\node at (-45.0*7-11:4.5) {\small $3$};
\draw (7) -- (v'7);
\draw (v'7) -- (u'7);
\node at (-45.0*7:2.1) {\small $4$};
\end{tikzpicture}
}
}
\hspace*{0.5cm}
\subfloat[\texttt{1a1a0a0c0b0c}]{
\label{subfig:1a1a0a0c0b0c}
\scalebox{0.65}{
	\begin{tikzpicture}[join=bevel,vertex/.style={circle,draw, minimum size=0.6cm},inner sep=0mm,scale=0.8]
\foreach \i in {0,...,7}  \node[vertex] (\i) at (-45.0*\i:3) {$v_\i$};
\foreach \i in {0,...,7}  \draw let \n1={int(mod(\i+1,8))} in (\i) -- (\n1);
\node at (-45.0*0:2.1) {\small $5$};
\node at (-45.0*2:2.1) {\small $4$};
\node (u1) at (-45.0*1:4) {};
\draw (1) -- (u1);
\node at (-45.0*1:2.1) {\small $3$};
\node (u3) at (-45.0*3:4) {};
\draw (3) -- (u3);
\node at (-45.0*3:2.1) {\small $2$};
\node (u4) at (-45.0*4:4) {};
\draw (4) -- (u4);
\node at (-45.0*4:2.1) {\small $2$};
\node[vertex] (v'5) at (-45.0*5:4.5) {$v'_5$};
\node (u'5) at (-45.0*5:5.5) {};
\node at (-45.0*5-11:4.5) {\small $3$};
\draw (5) -- (v'5);
\draw (v'5) -- (u'5);
\node at (-45.0*5:2.1) {\small $4$};
\node[vertex] (v'6) at (-45.0*6:4.5) {$v'_6$};
\node[vertex] (v''6) at (-45.0*6-7:6) {$v''_6$};
\node[vertex] (v'''6) at (-45.0*6+7:6) {$v'''_6$};
\node (u''6) at (-45.0*6-7:7) {};
\node (u'''6) at (-45.0*6+7:7) {};
\node at (-45.0*6-11:4.5) {\small $4$};
\node at (-45.0*6-14:6) {\small $3$};
\node at (-45.0*6+14:6) {\small $3$};
\draw (6) -- (v'6);
\draw (v'6) -- (v''6);
\draw (v'6) -- (v'''6);
\draw (v''6) -- (u''6);
\draw (v'''6) -- (u'''6);
\node at (-45.0*6:2.1) {\small $6$};
\node[vertex] (v'7) at (-45.0*7:4.5) {$v'_7$};
\node (u'7) at (-45.0*7:5.5) {};
\node at (-45.0*7-11:4.5) {\small $3$};
\draw (7) -- (v'7);
\draw (v'7) -- (u'7);
\node at (-45.0*7:2.1) {\small $5$};
\end{tikzpicture}
}
}
\hspace*{0.5cm}
\subfloat[\texttt{1a0a0c1c0a0c}]{
\label{subfig:1a0a0c1c0a0c}
\scalebox{0.65}{
	\begin{tikzpicture}[join=bevel,vertex/.style={circle,draw, minimum size=0.6cm},inner sep=0mm,scale=0.8]
\foreach \i in {0,...,7}  \node[vertex] (\i) at (-45.0*\i:3) {$v_\i$};
\foreach \i in {0,...,7}  \draw let \n1={int(mod(\i+1,8))} in (\i) -- (\n1);
\node at (-45.0*0:2.1) {\small $5$};
\node at (-45.0*4:2.1) {\small $6$};
\node (u1) at (-45.0*1:4) {};
\draw (1) -- (u1);
\node at (-45.0*1:2.1) {\small $2$};
\node (u2) at (-45.0*2:4) {};
\draw (2) -- (u2);
\node at (-45.0*2:2.1) {\small $2$};
\node[vertex] (v'3) at (-45.0*3:4.5) {$v'_3$};
\node (u'3) at (-45.0*3:5.5) {};
\node at (-45.0*3-11:4.5) {\small $3$};
\draw (3) -- (v'3);
\draw (v'3) -- (u'3);
\node at (-45.0*3:2.1) {\small $4$};
\node[vertex] (v'5) at (-45.0*5:4.5) {$v'_5$};
\node (u'5) at (-45.0*5:5.5) {};
\node at (-45.0*5-11:4.5) {\small $3$};
\draw (5) -- (v'5);
\draw (v'5) -- (u'5);
\node at (-45.0*5:2.1) {\small $4$};
\node (u6) at (-45.0*6:4) {};
\draw (6) -- (u6);
\node at (-45.0*6:2.1) {\small $3$};
\node[vertex] (v'7) at (-45.0*7:4.5) {$v'_7$};
\node (u'7) at (-45.0*7:5.5) {};
\node at (-45.0*7-11:4.5) {\small $3$};
\draw (7) -- (v'7);
\draw (v'7) -- (u'7);
\node at (-45.0*7:2.1) {\small $4$};
\end{tikzpicture}
}
}

\hspace*{-1cm}
\subfloat[\texttt{1c1a0c0a0c0a}]{
\label{subfig:1c1a0c0a0c0a}
\scalebox{0.65}{
	\begin{tikzpicture}[join=bevel,vertex/.style={circle,draw, minimum size=0.6cm},inner sep=0mm,scale=0.8]
\foreach \i in {0,...,7}  \node[vertex] (\i) at (-45.0*\i:3) {$v_\i$};
\foreach \i in {0,...,7}  \draw let \n1={int(mod(\i+1,8))} in (\i) -- (\n1);
\node at (-45.0*0:2.1) {\small $5$};
\node at (-45.0*2:2.1) {\small $5$};
\node[vertex] (v'1) at (-45.0*1:4.5) {$v'_1$};
\node (u'1) at (-45.0*1:5.5) {};
\node at (-45.0*1-11:4.5) {\small $3$};
\draw (1) -- (v'1);
\draw (v'1) -- (u'1);
\node at (-45.0*1:2.1) {\small $5$};
\node (u3) at (-45.0*3:4) {};
\draw (3) -- (u3);
\node at (-45.0*3:2.1) {\small $3$};
\node[vertex] (v'4) at (-45.0*4:4.5) {$v'_4$};
\node (u'4) at (-45.0*4:5.5) {};
\node at (-45.0*4-11:4.5) {\small $3$};
\draw (4) -- (v'4);
\draw (v'4) -- (u'4);
\node at (-45.0*4:2.1) {\small $3$};
\node (u5) at (-45.0*5:4) {};
\draw (5) -- (u5);
\node at (-45.0*5:2.1) {\small $3$};
\node[vertex] (v'6) at (-45.0*6:4.5) {$v'_6$};
\node (u'6) at (-45.0*6:5.5) {};
\node at (-45.0*6-11:4.5) {\small $3$};
\draw (6) -- (v'6);
\draw (v'6) -- (u'6);
\node at (-45.0*6:2.1) {\small $3$};
\node (u7) at (-45.0*7:4) {};
\draw (7) -- (u7);
\node at (-45.0*7:2.1) {\small $3$};
\end{tikzpicture}
}
}
\caption{Reducible configuration in \Cref{lemma:first_reducible_cycles}}
\label{fig:first_reducible_cycles}
\end{figure}

\begin{proof}

\figureproof{\Cref{subfig:1a1a1a0a0a}}
Here, we redefine $S=\{v_0,v_1,v_2,v_3,v_4\}$. By \Cref{fig:forced_non_colorable_P5}, $L(v_0)=L(v_1)=\{a,b\}$, $L(v_3)=L(v_4)=\{c,d\}$ and $L(v_2)=\{a,b,c,d\}$. Therefore, we can assume w.l.o.g that $v_6$ is colored $e$. Since $|L(v_0)|=2$, all of the colored vertices that $v_0$ sees must be colored differently. The same holds for $v_4$. However, it means that $v_2$ does not see the color $e$, which is impossible since $L(v_2)=\{a,b,c,d\}$.
\hfill\smallqed\medskip

\figureproof{\Cref{subfig:1a1a0a0c0c0a}}
Note that \gs. We first prove three important observations.
\begin{itemize}
\item $L(v_7)\neq L(v'_6)$. Suppose the contrary and color $v_7$, $v_6$, $v'_6$, $v_5$, $v'_5$, $v_4$, $v_3$ by \Cref{subfig:config8}. Now if $v_0$, $v_1$ and $v_2$ are colorable, then we are done. Thus according to \Cref{fig:forced_non_colorable_3path}, we can assume that $L(v_1)\subset L(v_0)=L(v_2)$. But then since by our assumption $L(v_7)= L(v'_6)$, we permute the colors of $v'_6$ and $v_7$ so that $L(v_0)\neq L(v_2)$ and we are done.
\item $L(v_3)\subset L(v_2)\supset L(v_4)$. If not, color $v_3$ and $v_4$ such that $|L(v_2)|\geq 3$. Recall that $L(v_7)\neq L(v'_6)$. Hence we color $v'_5$, $v_5$, $v_6$, $v'_6$, $v_7$ by \Cref{fig:forced_non_colorable_clawpath}. We finish by coloring $v_1$, $v_0$, $v_2$ in this order.
\item $L(v_1)\cap L(v_4)=\emptyset$. By contradiction, suppose $a\in L(v_1)\cap L(v_4)$. We will show the following observations.
\begin{itemize}
\item $a\notin L(v'_6)$. If $a\in L(v'_6)$, we color $v_1$, $v_4$ and $v'_6$ with $a$. Then, we color $v_3$. After that, we color $v'_5$, $v_5$, $v_6$, $v_7$ by \Cref{subfig:config1} and we finish by coloring $v_0$ and $v_2$ in this order.
\item $a\in L(v_7)$. If $a\notin L(v_7)$, we color $v_1$ and $v_4$ with $a$. Then, we color $v_3$. After that, we color $v'_5$, $v_5$, $v_6$, $v'_6$, $v_7$ by \Cref{fig:forced_non_colorable_clawpath} (recall that $L(v'_6)\neq L(v_7)$) and finish by coloring $v_0$ and $v_2$ in this order.
\item $a\in L(v'_5)$. If $a\notin L(v'5)$, we color $v_4$ and $v_7$ with $a$. Then, we color $v_3$. Finally, we finish by coloring $v_1$, $v_2$, $v_0$, $v_6$, $v_5$, $v'_6$, $v'_5$ in this order.
\item $|L(v_3)\setminus\{a\}|=1$. Otherwise, we color $v_4$ and $v_7$ with $a$. Then, we color $v_5$ in such a way that $v_3$ has at least 2 colors left. After that, we color $v'_5$, $v_5$, $v_6$, $v_7$ in this order. Finally, we finish by coloring $v_3$, $v_2$, $v_1$, $v_0$ by \Cref{subfig:config1}. 
\end{itemize}

Thus, we color $v'_5$, $v_3$ and $v_7$ with $a$, then we color the remaining vertices in the following order: $v_4$, $v_2$, $v_1$, $v_0$, $v_6$, $v_5$, $v'_6$.
\end{itemize}

Since $L(v_1)\cap L(v_4)=\emptyset$, we assume w.l.o.g. that $L(v_4)\subseteq \{a,b,c\}$ and $L(v_1)=\{d,e,f\}$. As $L(v_3)\subset L(v_2)\supset L(v_4)$, there exists a color, say $d$, in $L(v_1)$ such that after coloring $v_1$ with $d$, we have $|L(v_2)|\geq 4$ and $|L(v_3)|,|L(v_3)|\geq 2$. In conclusion, we color $v_1$ with $d$, $v_7$, $v_6$, $v'_6$, $v_5$, $v'5$, $v_4$, $v_3$ by \Cref{subfig:config8} and finish by coloring $v_0$ and $v_2$ in this order.

\hfill\smallqed\medskip

\figureproof{\Cref{subfig:1a0c1a0c0c0a}}
If $v'_2$ sees $v'_6$, then the are at distance exactly 2 and share a common neighbor, say $v_8$. Then vertices $v'_6$, $v_8$, $v'_2$, $v_2$, $v_3$, $v_4$, $v_5$, $v_6$ correspond to the reducible configuration of \Cref{subfig:1a1a1a0a0a}.

Therefore, we can assume that \gs. Color $v_2$ with $x\notin L(v'_2)$ and color greedily $v_1$. Then color vertices $v_4$, $v_5$, $v'_5$, $v_6$, $v'_6$, $v_7$, $v_0$ by \Cref{subfig:config7} and finish by coloring $v_3$ and $v'_2$ in this order.
\hfill\smallqed\medskip

\figureproof{\Cref{subfig:1a1a0c0b0a0c}}
If $v''_5$ sees $v_1$ by sharing a common neighbor, say $v_8$, then vertices $v'''_5$, $v'_5$, $v''_5$, $v_8$, $v_1$, $v_2$, $v_0$ form the reducible configuration of \Cref{subfig:1a1a0c1}. The case when $v'''_5$ sees $v_1$ is symmetric. 

Therefore, we can suppose that \gs. First we show that $L(v_1)\cap L(v_7)=\emptyset$. Suppose the contrary and color $v_1$ and $v_7$ with a same color. Then restrict $L(v_5)$ to $L(v_5)\setminus L(v''_5)$ and color vertices $v_6$, $v_5$, $v_4$, $v'_4$, $v_3$ by \Cref{subfig:config5}. Finish by coloring vertices $v'_5$, $v'''_5$, $v''_5$, $v_7$, $v_2$, $v_0$ in this order.

Observe that $L(v_1)\subset L(v_0)$. Therefore, since $L(v_1)\cap L(v_7)=\emptyset$ and since we are doing a 6-coloring, we conclude that $L(v'_7)\not\subset L(v_0)$.

We color $v'_5$ with $x\notin L(v''_5)$ and $v_6$, $v_5$, $v_4$, $v'_4$, $v_3$ by \Cref{subfig:config4}. Then we color $v'''_5$ and $v''_5$ in this order. Observe the remaining uncolored vertices are $v'_7$, $v_7$, $v_0$, $v_1$ and $v_2$. If the lists of available colors of these vertices, do not correspond to \Cref{fig:forced_non_colorable_P5}, then we are done. And it is indeed the case, since the only colored vertex seen by both $v_0$ and $v'_7$ is $v_6$, and since initially $L(v'_7)\not\subset L(v_0)$.
\hfill\smallqed\medskip

\figureproof{\Cref{subfig:1a1a0a0c0b0c}} We have \gs. Color vertices $v_0$ and $v_4$ with the same color by pigeonhole principle and then $v_3$, $v_1$ and $v_2$ in this order. The remaining vertices can be colored by \Cref{subfig:config10}.

\hfill\smallqed\medskip

\figureproof{\Cref{subfig:1a0a0c1c0a0c}}
If $v'_3$ sees $v'_7$, then they must be at distance exactly 2 since $G$ has girth $8$. Say $v_8$ is their common neighbor, then $v_0$, $v_7$, $v'_7$, $v_8$, $v'_3$, $v_3$, $v_2$ and $v_1$ form the reducible configuration from \Cref{subfig:1a1a1a0a0a}.

Thus, we have \gs. First, observe that $|L(v'_7)|=|L(v'_5)|=|L(v'_6)|=3$ and we will prove the following:
\begin{itemize}
\item $L(v_6)=L(v'_7)$. Otherwise, color $v_6$ differently from $L(v'_7)$, then color $v_1$ and $v_2$ in this order. Color $v'_5$, $v_5$, $v_4$, $v_3$, and $v'_3$ by \Cref{subfig:config13}. Finish by coloring $v_7$, $v_0$, and $v'_7$ in this order. 
\item $L(v_6)=L(v'_5)$. Otherwise, color $v_6$ differently from $L(v'_5)$, then color $v'_7$, $v_7$, $v_0$, $v_1$, and $v_2$ by \Cref{subfig:config2}. Finish by coloring $v'_3$, $v_3$, $v_5$, $v_4$, and $v'_5$ in this order.
\item $L(v_1)\cap L(v'_7)=\emptyset$. Otherwise, color $v_1$ and $v'_7$ with $x\in L(v_1)\cap L(v'_7)$. Then, color $v_2$ and $v_6$. Color $v'_5$, $v_5$, $v_4$, $v_3$, and $v'_3$ by \Cref{subfig:config13}. Finish by coloring $v_7$ and $v_0$ in this order.  
\end{itemize}
Using the equalities above, we have the following. Color $v_7$ differently from $L(v_6)$ and $L(v'_7)$. Now, color $v_1$ and $v_4$ with the same color, which is possible since $v_4$ has all six colors available. Observe that, since $L(v_1)\cap L(v'_7)=\emptyset$ and $L(v'_7)=L(v_6)=L(v'_5)$, $v_6$ and $v_5$ still have the same amount of available colors remaining. Finish by coloring $v_2$, $v'_3$, $v_3$, $v_5$, $v_6$, $v'_5$, $v_0$, and $v'_7$ in this order.

\hfill\smallqed\medskip

\figureproof{\Cref{subfig:1c1a0c0a0c0a}}
Note that \gs. Here, we redefine $S=\{v_0,v_1,v'_1,v_2\}$. Consider $\phi$ a coloring of $G-S$. Note that if $\phi$ is extendable to $G$, then we have a contradiction. Thus, $L(v_0)=L(v_1)=L(v'_1)=L(v_2)=\{a,b,c\}$ by \Cref{fig:forced_non_colorable_claw}. Now, we uncolor $v_3$, $v_4$, $v'_4$, $v_5$, $v_6$, $v'_6$, $v_7$ and note that the number of available colors correspond to what is depicted in \Cref{subfig:1c1a0c0a0c0a}. We assume w.l.o.g. that $L(v_0)=\{a,b,c,d,e\}$ where $d=\phi(v_7)$ and $e=\phi(v_6)$. Observe that $L(v'_6)\neq L(v_7)$, otherwise, we can permute the colors of $v'_6$ and $v_7$ in $\phi$ and extend $\phi$ to $G$ as $L(v_0)$ would no longer be $\{a,b,c\}$. Symmetrically, $L(v_3)\neq L(v'_4)$. 

If $d\notin L(v'_6)$, then we can color $v_7$ with $d$, $v_6$ with $x\neq e$, $v_5$, then $v_3$, $v_4$, $v'_4$ by \Cref{fig:forced_non_colorable_3path} since $L(v_3)\neq L(v'4)$, and finish by coloring $v'_6$. As $L(v_0)\neq \{a,b,c\}$, $\phi$ is extendable to $G$. 

Now, $d\in L(v'_6)$. In which case, there exists $ y\in L(v_7)\setminus L(v'6)$ so we color $v_7$ with $y$, $v_6$ with $z\neq d$, $v_5$, then $v_3$, $v_4$, $v'_4$ and finish by coloring $v'_6$. Finally, $\phi$ is extendable to $G$ because $L(v_0)\neq \{a,b,c\}$.

\hfill\smallqed\medskip

\end{proof}

\begin{lemma}
\label{lemma:special_reducible}
Graph $G$ does not contain the configurations depicted in \Cref{fig:special_reducible}. 
\end{lemma}
\begin{proof}

\figureproof{\Cref{subfig:c1a1c}}
If $v_1$ does not see $v_7$. Then the proof is a direct implication of \Cref{subfig:config15}.
If $v_1$ sees $v_7$, then they must be at distance exactly 2 since $G$ has girth at least 8 and therefore $|L(v_1)|\geq 3$ and $|L(v_7)|\geq 3$. We color $v_1$ such that $v_2$ has at least 2 colors left. We then obtain \Cref{subfig:config13}.
\hfill\smallqed\medskip

\figureproof{\Cref{subfig:c1a0c0c0c}}
If $v_1$ sees $v'_6$, then they must be at distance exactly 2 since $G$ has girth at least 8. Say $v_0$ is their common neighbor, then $v'_6,v_0,v_1,\dots,v_6$ form the reducible configuration from \Cref{subfig:1a1a1a0a0a}. If $v_1$ sees $v_7$, then they share a common neighbor $v_0$ and $v_1$, $v_2$, $v_3$, $v_4$, $v_5$, $v'_5$, $v_6$, $v'_6$, $v_7$, $v_8$, $v_0$ form the reducible configuration from \Cref{subfig:1a1a0a0c0c0a}. If $v_2$ sees $v_8$, then they share a common neighbor $v'_2$ and $v_8$, $v'_2$, $v_2$, $v_1$, $v_3$, $v_4$, $v_5$, $v'_5$, $v_6$, $v'_6$, $v_7$ form the reducible configuration from \Cref{subfig:1a0c1a0c0c0a}.

If $v_1$ sees $v_8$, they must be at distance exactly 2 since both are 2-vertices and there are no 2-paths due to \Cref{lemma:no 2-path}. Thus, $3\leq |L(v_1)|, |L(v_8)|\leq 4$. If we can color $v_2$ such that $v_1$ has at least 3 colors left, then we can color $v_4$, $v_5$, $v'_5$, $v_6$, $v'_6$, $v_7$, $v_8$ by \Cref{subfig:config8} and finish by coloring $v_3$ and $v_1$ in this order. Therefore, $|L(v_1)| = 3$ and $L(v_2)\subseteq L(v_1)$. We color $v_3$ with $x\notin L(v_1)$. Then, we color $v_4$, $v_5$, $v'_5$, $v_6$, $v'_6$, $v_7$ by \Cref{subfig:config7} and finish by coloring $v_8$, $v_2$ and $v_1$ in this order.

Now, \gs. If we can color $v_2$ such that $v_1$ has at least 2 colors left, then we can color $v_4$, $v_5$, $v'_5$, $v_6$, $v'_6$, $v_7$, $v_8$ by \Cref{subfig:config8}, and finish by coloring $v_3$ and $v_1$ in this order. Therefore, $L(v_1)=L(v_2)$ and $|L(v_1)|=2$. We restrict $L(v_3)$ to $L(v_3)\setminus L(v_1)$. Then, we color $v_3$, $v_4$, $v_5$, $v'_5$, $v_6$, $v'_6$, $v_7$, $v_8$ by \Cref{subfig:config9} and finish by coloring $v_2$ and $v_1$ in this order.

\figureproof{\Cref{subfig:1a0b0b0c0c1}}
If $v''_3$ sees $v_7$, then they must be at distance exactly 2 since $G$ has girth $8$. Say $v_8$ is their common neighbor, then $v''_3$, $v'_3$, $v'''_3$, $v_8$, $v_7$, $v_6$, $v'_6$ form the reducible configuration from \Cref{subfig:c1a1c}. Note that the cases when $v'''_3$ sees $v_7$, or $v''_3$ sees $v'_6$, or $v''_3$ sees $v_7$ are symmetric.

Observe that since $v_1$ cannot see both $v'_6$ and $v_7$, we can assume that $v_1$ does not see $v'_6$. Note that in this case $|L(v'_6)|=3$. Thus we restrict $L(v_5)$ to $L(v_5)\setminus L(v'_6)$ and $L(v_4)$ to $L(v_4)\setminus L(v''_4)$. We color vertices $v_5$, $v_4$, $v_3$, $v_2$, $v_1$, $v'_3$, $v''_3$, $v'''_3$ by \Cref{subfig:config10}. Then finish by coloring $v'_5$, $v'_4$, $v'''_4$, $v''_4$, $v_6$, $v_7$, $v'_6$ in this order.
\hfill\smallqed\medskip

\end{proof}

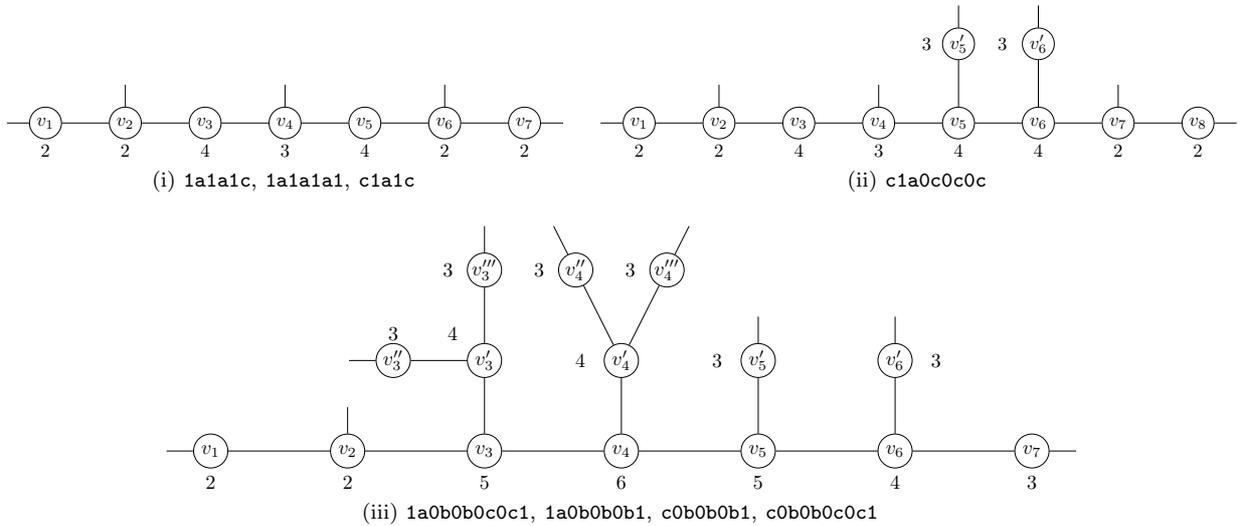
\begin{figure}[H]
\centering
\subfloat[\texttt{1a1a1c}, \texttt{1a1a1a1}, \texttt{c1a1c}]
{
\label{subfig:c1a1c}
\scalebox{0.7}{%
\begin{tikzpicture}[join=bevel,vertex/.style={circle,draw, minimum size=0.6cm},inner sep=0mm,scale=0.75]
 \foreach[count=\l from 1] \i in {0,2,...,12}{
  	\node[vertex] (\l) at (\i,0) {$v_\l$};
  }
  \node (0) at (-1,0) {};
  \node at (0,-0.7) {$2$};
  \node at (2,-0.7) {$2$};
  \node at (4,-0.7) {$4$};
  \node at (6,-0.7) {$3$};
  \node at (8,-0.7) {$4$};
  \node at (10,-0.7) {$2$};
  \node at (12,-0.7) {$2$};
  \node (8) at (13,0) {};
  \node (v'2) at (2,1) {};
  \node (v'4) at (6,1) {};
  \node (v'6) at (10,1) {};

  \draw[-] (0) -- (1) -- (2) -- (3) -- (4) -- (5) -- (6) -- (7) -- (8);
  \draw[-] (2) -- (v'2) (4) -- (v'4) (6) -- (v'6);
\end{tikzpicture}
}
}
\subfloat[\texttt{c1a0c0c0c}]
{
\label{subfig:c1a0c0c0c}
\scalebox{0.7}{%
\begin{tikzpicture}[join=bevel,vertex/.style={circle,draw, minimum size=0.6cm},inner sep=0mm,scale=0.75]
 \foreach[count=\l from 1] \i in {0,2,...,14}{
  	\node[vertex] (\l) at (\i,0) {$v_\l$};
  }
  \node (0) at (-1,0) {};
  \node at (0,-0.7) {$2$};
  \node at (2,-0.7) {$2$};
  \node at (4,-0.7) {$4$};
  \node at (6,-0.7) {$3$};
  \node at (8,-0.7) {$4$};
  \node at (10,-0.7) {$4$};
  \node at (12,-0.7) {$2$};
  \node at (14,-0.7) {$2$};
  \node (9) at (15,0) {};
  \node (v'2) at (2,1) {};
  \node (v'4) at (6,1) {};
  \node[vertex] (v'5) at (8,2) {$v'_5$};
  \node at (7.2,2) {$3$};  
  \node (v''5) at (8,3) {};
  \node[vertex] (v'6) at (10,2) {$v'_6$};
  \node at (9.1,2) {$3$};
  \node (v''6) at (10,3) {};
  \node (v'7) at (12,1) {};

  \draw[-] (0) -- (1) -- (2) -- (3) -- (4) -- (5) -- (6) -- (7) -- (8) -- (9);
  \draw[-] (2) -- (v'2) (4) -- (v'4) (5) -- (v'5) -- (v''5) (6) -- (v'6) -- (v''6) (7) -- (v'7);
\end{tikzpicture}
}
}

\subfloat[\texttt{1a0b0b0c0c1}, \texttt{1a0b0b0b1}, \texttt{c0b0b0b1}, \texttt{c0b0b0c0c1}]
{
\label{subfig:1a0b0b0c0c1}
\scalebox{0.75}{%
\begin{tikzpicture}[join=bevel,vertex/.style={circle,draw, minimum size=0.6cm},inner sep=0mm,scale=0.8]
 \foreach[count=\l from 1] \i in {0,3,...,18}{
  	\node[vertex] (\l) at (\i,0) {$v_\l$};
  }
  \node (0) at (-1,0) {};
  \node at (0,-0.7) {$2$};
  \node at (3,-0.7) {$2$};
  \node at (6,-0.7) {$5$};
  \node at (9,-0.7) {$6$};
  \node at (12,-0.7) {$5$};
  \node at (15,-0.7) {$4$};
  \node at (18,-0.7) {$3$};
  \node (8) at (19,0) {};
  \node (v'2) at (3,1) {};
  
  \node[vertex] (v'3) at (6,2) {$v'_3$};
  \node at (5.3,2.6) {$4$};
  \node[vertex] (v''3) at (4,2) {$v''_3$};
  \node at (4,2.6) {$3$};
  \node (v''31) at (3,2) {};
  \node[vertex] (v'''3) at (6,4) {$v'''_3$};
  \node at (5.2,4) {$3$};
  \node (v'''31) at (6,5) {};
  
  \node[vertex] (v'4) at (9,2) {$v'_4$};
  \node at (8.1,2) {$4$};
  \node[vertex] (v''4) at (8,4) {$v''_4$};
  \node at (7.2,4) {$3$};
  \node (v''41) at (7.5,5) {};
  \node[vertex] (v'''4) at (10,4) {$v'''_4$};
  \node at (9.2,4) {$3$};
  \node (v'''41) at (10.5,5) {};
  
  \node (v''5) at (12,3) {};
  \node[vertex] (v'5) at (12,2) {$v'_5$};
  \node at (11.1,2) {$3$};
  \node[vertex] (v'6) at (15,2) {$v'_6$};
  \node at (15.9,2) {$3$};
  \node (v''6) at (15,3) {};

  \draw[-] (0) -- (1) -- (2) -- (3) -- (4) -- (5) -- (6) -- (7) -- (8);
  \draw[-] (2) -- (v'2) (3) -- (v'3) -- (v''3) -- (v''31) (v'3) -- (v'''3) -- (v'''31) (4) -- (v'4) -- (v''4) -- (v''41) (v'4) -- (v'''4) -- (v'''41) (5) -- (v'5) -- (v''5) (6) -- (v'6) -- (v''6);
\end{tikzpicture}
}
}
\caption{Reducible configurations in \Cref{lemma:special_reducible}.}
\label{fig:special_reducible}
\end{figure}

\begin{lemma}
\label{lemma:reducible_cycles}
Graph $G$ does not contain the $8$-faces depicted in \Cref{fig:reducible_cycles} in the Appendix. 
\end{lemma}

The proof of \Cref{lemma:reducible_cycles} is also in the Appendix. It follows the same scheme as \Cref{lemma:first_reducible_cycles} and uses \Cref{sec:useful} as well as the previous lemma. There are a lot of configurations and their proofs are quite tedious and do not contribute extra value to what we already know, even though they are necessary.

\medskip

We have started out by coloring these configurations by computer (by testing all precoloring of the set of vertices separating our configuration from the rest of the graph) but this proves to be very time consuming. Moreover, there are tricks that can be done manually (restricting the considered set of vertices in the configurations, uncoloring then recoloring part of the configuration) that can hardly be replicated by computer. Concretely, it means that not all precoloring is a possible precoloring of a proper subgraph of $G$ and we cannot know which precoloring to test, which not to with our naive approach.

\begin{lemma} \label{lemma:1c1a0a1a0a}
Consider the configuration in \Cref{subfig:1c1a0a1a0a}. If $v_3$, $v_4$, $v_5$, $v_6$, and $v_7$ are colorable, but the configuration as a whole is not, then $L(v_3)=L(v_4)=L(v_6)=L(v_7)=L(v_1)\setminus L(v'_1)$ and $|L(v_3)|=2$.
\end{lemma}

\begin{figure}[H]
\centering
\scalebox{0.75}{
\begin{tikzpicture}[join=bevel,vertex/.style={circle,draw, minimum size=0.6cm},inner sep=0mm,scale=0.8]
\foreach \i in {0,...,7}  \node[vertex] (\i) at (-45.0*\i:3) {$v_\i$};
\foreach \i in {0,...,7}  \draw let \n1={int(mod(\i+1,8))} in (\i) -- (\n1);
\node at (-45.0*0:2.1) {\small $5$};
\node at (-45.0*2:2.1) {\small $5$};
\node at (-45.0*5:2.1) {\small $4$};
\node[vertex] (v'1) at (-45.0*1:4.5) {$v'_1$};
\node (u'1) at (-45.0*1:5.5) {};
\node at (-45.0*1-11:4.5) {\small $3$};
\draw (1) -- (v'1);
\draw (v'1) -- (u'1);
\node at (-45.0*1:2.1) {\small $5$};
\node (u3) at (-45.0*3:4) {};
\draw (3) -- (u3);
\node at (-45.0*3:2.1) {\small $2$};
\node (u4) at (-45.0*4:4) {};
\draw (4) -- (u4);
\node at (-45.0*4:2.1) {\small $2$};
\node (u6) at (-45.0*6:4) {};
\draw (6) -- (u6);
\node at (-45.0*6:2.1) {\small $2$};
\node (u7) at (-45.0*7:4) {};
\draw (7) -- (u7);
\node at (-45.0*7:2.1) {\small $2$};
\end{tikzpicture}
}
\caption{\label{subfig:1c1a0a1a0a} \texttt{1c1a0a1a0a}}
\end{figure}

\begin{proof}
First, observe that we have \gs. We color $v_3$, $v_4$, $v_5$, $v_6$, and $v_7$. Observe that $|L(v_0)|=|L(v_2)|=|L(v'_1)|=3$ and $|L(v_1)|\geq 3$. So, the remaining vertices are not colorable if and only if $L(v_0)=L(v_1)=L(v'_1)=L(v_2)=\{a,b,c\}$ w.l.o.g. due to \Cref{fig:forced_non_colorable_claw}.
 
Now, let $\{d,e\}=L(v_1)\setminus L(v'_1)$ and uncolor $v_3$, $v_4$, $v_5$, $v_6$, and $v_7$. Due to our previous observations, we can assume w.l.o.g. that $v_3$ and $v_7$ must have been colored $d$ and $e$ respectively. Moreover, due to \Cref{lemma:colorable_22422}, since we know that $v_3$, $v_4$, $v_5$, $v_6$, and $v_7$ are colorable, there exists another coloring of these vertices where $v_3$ is not colored $d$ or $v_7$ is not colored $e$. As $v_0$, $v_1$, $v'_1$, and $v_2$ must remain uncolorable, we know that $v_3$ must have been colored $e$ and $v_7$ colored $d$. So, we know that $\{d,e\}\subseteq L(v_3)$ and $\{d,e\}\subseteq L(v_7)$.
In addition, when $v_3$ was colored $d$ ($e$), $d$ ($e$) must be in $L(v_2)$ or we would have had $|L(v_2)|\geq 4$ after the coloring of $v_3$, $v_4$, $v_5$, $v_6$, and $v_7$. In other words, $L(v_2)=\{a,b,c,d,e\}$. Symmetrically, the same holds for $L(v_0)$.
Knowing that $L(v_2)=\{a,b,c,d,e\}$, when $v_3$ was colored $d$ ($e$), $v_4$ must have been colored $e$ ($d$). So we get $\{d,e\}\subseteq L(v_4)$. Similarly, the same holds for $L(v_6)$.
Finally, if any of $v_3$, $v_4$, $v_6$, or $v_7$ has another available color than $d$ and $e$, we could have colored them with one vertex not colored $d$, nor $e$, and finish coloring the rest of the configuration due to \Cref{fig:forced_non_colorable_P5} and \Cref{fig:forced_non_colorable_claw}, which is impossible. Consequently, we have $L(v_3)=L(v_4)=L(v_6)=L(v_7)=L(v_1)\setminus L(v'_1)=\{d,e\}$. 
\end{proof}

\begin{lemma} \label{lemma:colorable_1c1a0a1a0a}
The configurations in \Cref{fig:colorable_1c1a0a1a0a} are colorable.
\end{lemma}

\begin{figure}[H]
\centering
\subfloat[]{
\label{subfig:colorable_1c1a0a1a0a3}
\scalebox{0.75}{
\begin{tikzpicture}[join=bevel,vertex/.style={circle,draw, minimum size=0.6cm},inner sep=0mm,scale=0.8]
\foreach \i in {0,...,7}  \node[vertex] (\i) at (-45.0*\i:3) {$v_\i$};
\foreach \i in {0,...,7}  \draw let \n1={int(mod(\i+1,8))} in (\i) -- (\n1);
\node at (-45.0*0:2.1) {\small $5$};
\node at (-45.0*2:2.1) {\small $5$};
\node at (-45.0*5:2.1) {\small $6$};
\node[vertex] (v'1) at (-45.0*1:4.5) {$v'_1$};
\node (u'1) at (-45.0*1:5.5) {};
\node at (-45.0*1-11:4.5) {\small $3$};
\draw (1) -- (v'1);
\draw (v'1) -- (u'1);
\node at (-45.0*1:2.1) {\small $5$};
\node (u3) at (-45.0*3:4) {};
\draw (3) -- (u3);
\node at (-45.0*3:2.1) {\small $3$};
\node at (-45.0*4:2.1) {\small $4$};
\node at (-45.0*6:2.1) {\small $4$};
\node (u7) at (-45.0*7:4) {};
\draw (7) -- (u7);
\node at (-45.0*7:2.1) {\small $3$};

\node[vertex,xshift=-2.3cm,yshift=2.3cm] (v'4) at (-45.0*3:3) {$v'_4$};
\node[xshift=-2.3cm,yshift=2.3cm] (u'4) at (-45.0*3:4) {};
\node[xshift=-2.3cm,yshift=2.3cm] at (-45.0*3:2.1) {\small $3$};
\node[vertex,xshift=-2.3cm,yshift=2.3cm] (v''4) at (-45.0*4:3) {$v''_4$};
\node[vertex,xshift=-2.3cm,yshift=2.3cm] (u''4) at (-45.0*4:4.5) {$u''_4$};
\node[xshift=-2.3cm,yshift=2.3cm] at (-45.0*4-11:4.5) {\small $3$};
\node[xshift=-2.3cm,yshift=2.3cm] (u'''4) at (-45.0*4:5.5) {};
\node[xshift=-2.3cm,yshift=2.3cm] at (-45.0*4:2.1) {\small $4$};
\node[vertex,xshift=-2.3cm,yshift=2.3cm] (v8) at (-45.0*5:3) {$v_8$};
\node[vertex,xshift=-2.3cm,yshift=2.3cm] (u8) at (-45.0*5:4.5) {$u_8$};
\node[xshift=-2.3cm,yshift=2.3cm] at (-45.0*5-11:4.5) {\small $3$};
\node[xshift=-2.3cm,yshift=2.3cm] (u'8) at (-45.0*5:5.5) {};
\node[xshift=-2.3cm,yshift=2.3cm] at (-45.0*5:2.1) {\small $5$};
\node[vertex,xshift=-2.3cm,yshift=2.3cm] (v''6) at (-45.0*6:3) {$v''_6$};
\node[vertex,xshift=-2.3cm,yshift=2.3cm] (u''6) at (-45.0*6:4.5) {$u''_6$};
\node[xshift=-2.3cm,yshift=2.3cm] at (-45.0*6-11:4.5) {\small $3$};
\node[xshift=-2.3cm,yshift=2.3cm] (u'''6) at (-45.0*6:5.5) {};
\node[xshift=-2.3cm,yshift=2.3cm] at (-45.0*6:2.1) {\small $4$};
\node[vertex,xshift=-2.3cm,yshift=2.3cm] (v'6) at (-45.0*7:3) {$v'_6$};
\node[xshift=-2.3cm,yshift=2.3cm] (u'6) at (-45.0*7:4) {};
\node[xshift=-2.3cm,yshift=2.3cm] at (-45.0*7:2.1) {\small $3$};

\draw (4) -- (v'4);
\draw (6) -- (v'6);
\draw (v'6) -- (v''6);
\draw (v'4) -- (v''4);
\draw (v8) -- (v''6);
\draw (v8) -- (v''4);
\draw (v8) -- (u8);
\draw (u8) -- (u'8);
\draw (v'4) -- (u'4);
\draw (v''4) -- (u''4);
\draw (u''4) -- (u'''4);
\draw (v''6) -- (u''6);
\draw (u''6) -- (u'''6);
\draw (v'6) -- (u'6);
\end{tikzpicture}
}
}
\subfloat[]{
\label{subfig:colorable_1c1a0a1a0a2}
\scalebox{0.75}{
\begin{tikzpicture}[join=bevel,vertex/.style={circle,draw, minimum size=0.6cm},inner sep=0mm,scale=0.8]
\foreach \i in {0,...,7}  \node[vertex] (\i) at (-45.0*\i:3) {$v_\i$};
\foreach \i in {0,...,7}  \draw let \n1={int(mod(\i+1,8))} in (\i) -- (\n1);
\node at (-45.0*0:2.1) {\small $5$};
\node at (-45.0*2:2.1) {\small $5$};
\node at (-45.0*5:2.1) {\small $6$};
\node[vertex] (v'1) at (-45.0*1:4.5) {$v'_1$};
\node (u'1) at (-45.0*1:5.5) {};
\node at (-45.0*1-11:4.5) {\small $3$};
\draw (1) -- (v'1);
\draw (v'1) -- (u'1);
\node at (-45.0*1:2.1) {\small $5$};
\node (u3) at (-45.0*3:4) {};
\draw (3) -- (u3);
\node at (-45.0*3:2.1) {\small $3$};
\node at (-45.0*4:2.1) {\small $4$};
\node at (-45.0*6:2.1) {\small $4$};
\node (u7) at (-45.0*7:4) {};
\draw (7) -- (u7);
\node at (-45.0*7:2.1) {\small $3$};

\node[vertex,xshift=-2.3cm,yshift=2.3cm] (v'4) at (-45.0*3:3) {$v'_4$};
\node[xshift=-2.3cm,yshift=2.3cm] (u'4) at (-45.0*3:4) {};
\node[xshift=-2.3cm,yshift=2.3cm] at (-45.0*3:2.1) {\small $2$};
\node[vertex,xshift=-2.3cm,yshift=2.3cm] (v''4) at (-45.0*4:3) {$v''_4$};
\node[xshift=-2.3cm,yshift=2.3cm] (u''4) at (-45.0*4:4) {};
\node[xshift=-2.3cm,yshift=2.3cm] at (-45.0*4:2.1) {\small $2$};
\node[vertex,xshift=-2.3cm,yshift=2.3cm] (v8) at (-45.0*5:3) {$v_8$};
\node[xshift=-2.3cm,yshift=2.3cm] (u8) at (-45.0*4:4) {};
\node[xshift=-2.3cm,yshift=2.3cm] at (-45.0*5:2.1) {\small $4$};
\node[vertex,xshift=-2.3cm,yshift=2.3cm] (v''6) at (-45.0*6:3) {$v''_6$};
\node[xshift=-2.3cm,yshift=2.3cm] (u''6) at (-45.0*6:4) {};
\node[xshift=-2.3cm,yshift=2.3cm] at (-45.0*6:2.1) {\small $2$};
\node[vertex,xshift=-2.3cm,yshift=2.3cm] (v'6) at (-45.0*7:3) {$v'_6$};
\node[xshift=-2.3cm,yshift=2.3cm] (u'6) at (-45.0*7:4) {};
\node[xshift=-2.3cm,yshift=2.3cm] at (-45.0*7:2.1) {\small $2$};

\draw (4) -- (v'4);
\draw (6) -- (v'6);
\draw (v'6) -- (v''6);
\draw (v'4) -- (v''4);
\draw (v8) -- (v''6);
\draw (v8) -- (v''4);
\draw (v'4) -- (u'4);
\draw (v''4) -- (u''4);
\draw (v''6) -- (u''6);
\draw (v'6) -- (u'6);
\end{tikzpicture}
}
}

\subfloat[]{
\label{subfig:colorable_1c1a0a1a0a1}
\scalebox{0.75}{
\begin{tikzpicture}[join=bevel,vertex/.style={circle,draw, minimum size=0.6cm},inner sep=0mm,scale=0.8]
\foreach \i in {0,...,7}  \node[vertex] (\i) at (-45.0*\i:3) {$v_\i$};
\foreach \i in {0,...,7}  \draw let \n1={int(mod(\i+1,8))} in (\i) -- (\n1);
\node at (-45.0*0:2.1) {\small $5$};
\node at (-45.0*2:2.1) {\small $5$};
\node at (-45.0*5:2.1) {\small $5$};
\node[vertex] (v'1) at (-45.0*1:4.5) {$v'_1$};
\node (u'1) at (-45.0*1:5.5) {};
\node at (-45.0*1-11:4.5) {\small $3$};
\draw (1) -- (v'1);
\draw (v'1) -- (u'1);
\node at (-45.0*1:2.1) {\small $5$};
\node (u3) at (-45.0*3:4) {};
\draw (3) -- (u3);
\node at (-45.0*3:2.1) {\small $3$};
\node at (-45.0*4:2.1) {\small $4$};
\node (u6) at (-45.0*6:4) {};
\draw (6) -- (u6);
\node at (-45.0*6:2.1) {\small $2$};
\node (u7) at (-45.0*7:4) {};
\draw (7) -- (u7);
\node at (-45.0*7:2.1) {\small $2$};
\node[vertex] (v'4) at (-45.0*4:4.5) {$v'_4$};
\node[vertex] (v''4) at (-45.0*4-7:6) {$v''_4$};
\node (v'''4) at (-45.0*4+7:6) {};
\node (u''4) at (-45.0*4-7:7) {};
\draw (4) -- (v'4);
\draw (v'4) -- (v''4);
\draw (v'4) -- (v'''4);
\draw (v''4) -- (u''4);
\node at (-45.0*4-11:4.5) {\small $2$};
\node at (-45.0*4-14:6) {\small $2$};
\end{tikzpicture}
}
}
\caption{\label{fig:colorable_1c1a0a1a0a} Reducible configurations in \Cref{lemma:colorable_1c1a0a1a0a}.}
\end{figure}
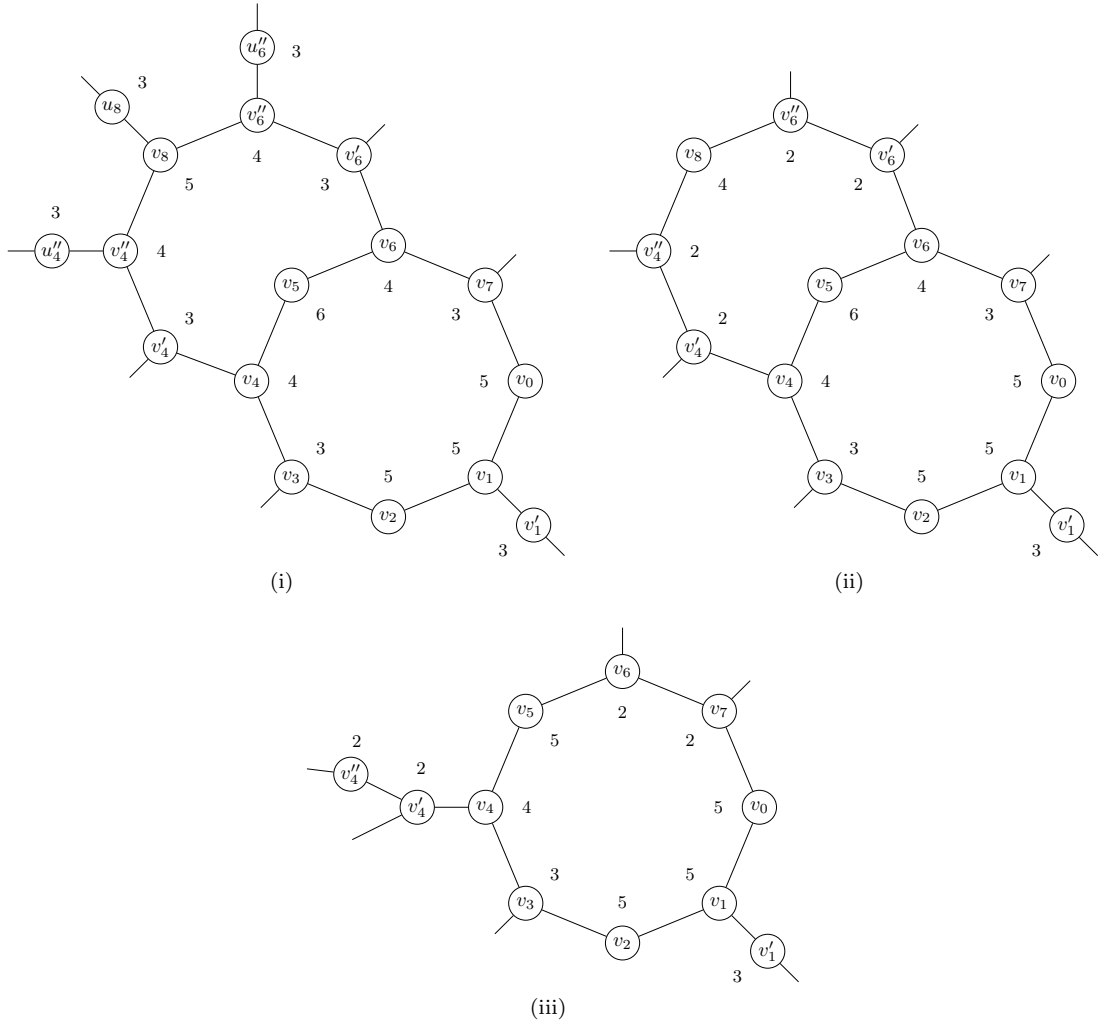

\begin{proof} The outline of each proof uses the same conventions as before.

\figureproof{\Cref{subfig:colorable_1c1a0a1a0a3}}
If $v'_1=u_8$, then $|L(v'_1)|=|L(v_8)|=|L(v_1)|=6$. Now, consider the two following cases:
\begin{itemize}
\item If there exists $x\in L(v_3)\cap L(v_7)$, then color $v_3$ and $v_7$ with $x$. Color $v''_6$ such that $u''_6$ still has 3 colors remaining, then $v'_6$ and $v_6$ in this order. Color $v_4$, $v'_4$, $v''_4$, and $u''_4$ by \Cref{subfig:config1}. Finish by coloring $v_8$, $u''_6$, $v'_1$ ($=u_8$), $v_0$, $v_2$, and $v_1$ in this order.
\item If $L(v_3)\cap L(v_7)=\emptyset$, then it suffices to show that we can color $v_3$, $v_4$, $v_5$, $v_6$, $v_7$, $v'_4$, $v''_4$, $u''_4$, $v_8$, $v'_6$, $v''_6$, and $u''_6$. 

Indeed, say they are colorable with $\phi$, then after coloring $v'_4$, $v''_4$, $u''_4$, $v_8$, $v'_6$, $v''_6$, and $u''_6$ with $\phi$, we obtain the configuration from \Cref{subfig:1c1a0c0a0c0a} where $v_3$, $v_4$, $v_5$, $v_6$, and $v_7$ are colorable (with $\phi$) but $L(v_3)\cap L(v_7)=\emptyset$, so the whole configuration can be colored.

It remains to show that there exists such a coloring $\phi$. Start by coloring $v''_4$ such that $u''_4$ still has 3 colors remaining. Similarly, color $v_6$ such that $v_7$ still has 3 colors remaining. Finish by coloring $v'_4$, $v_4$, $v_3$, $v_5$, $v'_6$, $v_7$, $v''_6$, $u''_6$, $v_8$, and $u''_4$ in this order.
\end{itemize}

Now, observe that $v'_1$ might see $u''_4$ and if it does, then they must be at distance exactly 2 since $G$ has no $2^+$-paths due to \Cref{lemma:no 2-path}. Symmetrically, the same holds if $v'_1$ sees $u''_6$. The following colorings will still work when $v'_1$ sees $u''_4$ or $u''_6$.

Consider the two following cases:
\begin{itemize}
\item If $|L(v_3)\cap L(v_7)|\geq 2$, say $\{d,e\}\subset L(v_3)\cap L(v_7)$, then let $x\in L(v_3)\setminus\{d,e\}$. We restrict $L(v'_4)$ to $L(v'_4)\setminus\{x\}$ and we color $v'_6$ differently from $\{d,e\}$. Color $v'_4$, $v''_4$, $u''_4$, $v_8$, $u_8$, $v''_6$, and $u''_6$ by \Cref{subfig:config8}.

Observe that we obtain the configuration from \Cref{subfig:1c1a0c0a0c0a} where  
$v_3$, $v_4$, $v_5$, $v_6$, and $v_7$ are colorable by \Cref{fig:forced_non_colorable_P5} since $L(v_3)$ and $L(v_7)$ will have at least one color in common. Moreover, we will have either $L(v_7)=\{d,e\}$ and $x\in L(v_3)\setminus\{d,e\}$, or $|L(v_7)|\geq 3$, both of which means that the remaining configuration is colorable by \Cref{lemma:1c1a0a1a0a}.

\item If $|L(v_3)\cap L(v_7)|\leq 1$, then it suffices to show that we can color $v_3$, $v_4$, $v_5$, $v_6$, $v_7$, $v'_4$, $v''_4$, $u''_4$, $v_8$, $u_8$, $v'_6$, $v''_6$, and $u''_6$. 

Indeed, say they are colorable with $\phi$, then after coloring $v'_4$, $v''_4$, $u''_4$, $v_8$, $u_8$, $v'_6$, $v''_6$, and $u''_6$ with $\phi$, we obtain the configuration from \Cref{subfig:1c1a0c0a0c0a} where $v_3$, $v_4$, $v_5$, $v_6$, and $v_7$ are colorable (with $\phi$) but $|L(v_3)\cap L(v_7)|\leq 1$, so the whole configuration can be colored.

It remains to show that there exists such a coloring $\phi$. Start by coloring $v''_4$ such that $u''_4$ still has 3 colors remaining. Similarly, color $v_6$ such that $v_7$ still has 3 colors remaining. Then, color $v'_6$. Color $u''_6$, $v''_6$, $v_8$, and $u_8$ by \Cref{subfig:config1}. Finish by coloring $v'_4$, $u''_4$, $v_4$, $v_3$, $v_5$, and $v_7$ in this order.
\end{itemize} 
 
\hfill\smallqed\medskip

\figureproof{\Cref{subfig:colorable_1c1a0a1a0a2}}
If $v'_1$ sees $v''_4$, then they must be at distance exactly 2 since $G$ has girth 8. Say $v_8$ is their common neighbor, then $v'_1$, $v_1$, $v_0$, $v_2$, $v_3$, $v_4$, $v_5$, $v'_4$, $v''_4$, and $v_8$ form the reducible configuration from \Cref{subfig:1c1a0c0a0c0a}. Symmetrically, the same holds if $v'_1$ sees $v''_6$.

So we have \gs.

We redefine $S=\{v_0,v_1,v'_1,v_2\}$ and let $\phi$ be the coloring of the rest of the graph. Now we uncolor the rest of the configuration and we have the corresponding list of colors as in \Cref{subfig:colorable_1c1a0a1a0a2}.

After coloring $v'_4$, $v''_4$, $v_8$, $v''_6$, and $v'_6$ with $\phi$, the remaining colors for $v_3$, $v_4$, $v_6$, $v_7$ must be the same two colors, say $\{d,e\}$ (determined by $L(v_1)\setminus L(v'_1)$), or the whole configuration would be colorable by \Cref{lemma:1c1a0a1a0a}. We can also deduce that $L(v_3)=\{d,e,\phi(v'_4)\}$. Similarly, $L(v_7)=\{d,e,\phi(v'_6)\}$. Now, thanks to \Cref{lemma:colorable_22422}, we know there exists another coloring $\phi'$ of $v'_4$, $v''_4$, $v_8$, $v''_6$, and $v'_6$ such that $\phi'(v'_4)\neq \phi(v'_4)$ or $\phi'(v'_6)\neq \phi(v'_6)$. Say w.l.o.g. that $\phi'(v'_4)\neq \phi(v'_4)$. As a result, $v_3$, $v_4$, $v_5$, $v_6$, and $v_7$ is colorable by \Cref{fig:forced_non_colorable_P5} and $L(v_3)\neq \{d,e\}$ so the configuration is colorable by \Cref{lemma:1c1a0a1a0a}.
\hfill\smallqed\medskip

\figureproof{\Cref{subfig:colorable_1c1a0a1a0a1}}
If $v'_1$ sees $v''_4$, then they must be at distance exactly 2 since $G$ has girth 8. Say $v_8$ is their common neighbor, then $v''_4$, $v_8$, $v'_1$, $v_1$, $v_2$, $v_3$, $v_4$, and $v'_4$ form the reducible configuration from \Cref{subfig:1a1a1a0a0a}.

Now, we have \gs.

We redefine $S=\{v_0,v_1,v'_1,v_2\}$ and let $\phi$ be the coloring of the rest of the graph. Now we uncolor $v_3$, $v_4$, $v'_4$, $v''_4$, $v_5$, $v_6$, and $v_7$ and we have the corresponding list of colors as in \Cref{subfig:colorable_1c1a0a1a0a1}.

Let $\{d,e\}\subseteq L(v_6)$. 

If $\{d,e\}\subseteq L(v_3)$, then we color $v'_4$ differently from $L(v_3)\setminus\{d,e\}$ and color $v''_4$. As a result, $v_3$, $v_4$, $v_5$, $v_6$, and $v_7$ are colorable by \Cref{fig:forced_non_colorable_P5} and $L(v_3)\neq\{d,e\}\subseteq L(v_6)$ so the configuration is colorable by \Cref{lemma:1c1a0a1a0a}.

If $\{d,e\}\not\subseteq L(v_3)$, then since $v_3$, $v_4$, $v'_4$, $v''_4$, $v_5$, $v_6$, and $v_7$ was colorable with $\phi$, we recolor $v'_4$ and $v''_4$ with $\phi(v'_4)$ and $\phi(v''_4)$ respectively. Now, observe that $v_3$, $v_4$, $v_5$, $v_6$, and $v_7$ are colorable but $L(v_3)\neq L(v_6)$ so the configuration is colorable by \Cref{lemma:1c1a0a1a0a}.
\hfill\smallqed\medskip

\end{proof}

\section{Discharging procedure}
\label{sec:discharging}

\paragraph{Charge distribution:} For a plane graph $G=(V,E,F)$, Euler formula $|V|-|E|+|F|=2$ can be rewritten as

\begin{equation}\label{euler}
\sum_{v\in V(G)}\left(\frac{7}{2}d(v)-9\right)+\sum_{f\in F(G)}(d(f)-9)=-18.
\end{equation}

We assign to each vertex $v$ the charge $\mu(v)=\frac72 d(v)-9$ and to each face $f$ the charge $\mu(f)=d(f)-9$. To prove the non-existence of $G$, we will redistribute the charges preserving their sum and obtaining a non-negative total charge, which will contradict \Cref{euler}.

To do so, we will divide the discharging procedure into two rounds. In the first round, we will redistribute the charges only between the vertices of $G$, resulting in a non-negative amount of charge on each vertex using the properties proved in \Cref{lemma:special_reducible,lemma:reducible}. For the second round, first observe that $\mu(f)=d(f)-9\geq 0$ for every face of size at least 9. Therefore, since $g(G)\geq 8$ and $\mu(f)=-1$ for every $8$-face, we will redistribute the remaining charges on each vertex over the non-reducible $8$-faces (every reducible cycle is shown in\Cref{lemma:reducible_cycles}) to obtain a non-negative amount of charge on faces. Thus, we will get a non-negative total of charge, which is a contradiction to \Cref{euler}. In our proof, we have to consider a large number of non-reducible $8$-faces. To handle this, we will provide a computer procedure that checks the remaining charge on each non-reducible $8$-face. In order to define this procedure, we will present an encoding of the $8$-faces, the reducible configurations, and the discharging rules.

\subsection{First round: vertices to vertices}
\label{subsec:discharging_vertices}

We define the following discharging rules on the vertices of $G$ :
\begin{itemize}
\item[\ru0] A 3-vertex gives 1 to a 2-neighbor.
\item[\ru1] A 3-vertex gives $\frac12$ to a (1,1,0)-neighbor.
\item[\ru2] A 3-vertex gives $\frac12$ to a (1,1,1)-vertex at distance 2.
\end{itemize}

We will now calculate the exact amount $\mu^*(v)$ of charges that $v$ ends up with after applying \ru0, \ru1 and \ru2.

\paragraph{If $d(v)=2$:}
Recall that the initial charge for $v$ is $\mu(v)=\frac72 d(v) - 9 = -2$. By \Cref{lemma:no 2-path}, $v$ can only have 3-neighbors. According to the discharging rules, $v$ receives 1 from each of its neighbor by \ru0 and does not give any charge away. Thus, $v$ ends up with $$\mu^*(v)=-2+2\cdot 1 = 0.$$

\paragraph{If $d(v)=3$:} Recall that the initial charge is $\mu(v) = \frac72 d(v) - 9 = \frac32$.
\begin{itemize}
\item If $v$ is a $(1,1,1)$-vertex.\\
Every neighbor of $v$ is a 2-vertex so only \ru0 and \ru2 may apply. However, due to \Cref{subfig:1c1a1}, there is no $(1,1,0^+)$-vertex at distance 2 from $v$. So, $v$ does not give away any charge to 3-vertices but only receive instead. Thus, by \ru0 and \ru2, we have 
$$ \mu^*(v)=\frac32 - 3\cdot 1 + 3\cdot \frac12 = 0.$$

\item If $v$ is a $(1,1,0)$-vertex.\\
Due to \Cref{subfig:1c1a1}, there is no $(1,1,1)$-vertex at distance 2 from $v$ so \ru2 does not apply. Due to \Cref{subfig:1c0c1}, $v$ cannot have a $(1,1,0)$-neighbor. So, $v$ does not give away any charge to 3-vertices but only receive by \ru1 instead. Thus, by \ru0 and \ru1, we have 
$$ \mu^*(v) = \frac32 - 2\cdot 1 + \frac12 = 0.$$

\item If $v$ is a $(1,0,0)$-vertex.\\
\begin{itemize}
\item If $v$ has a $(1,1,0)$-neighbor, $v$ cannot have another $(1,0^+,0)$-neighbor due to \Cref{subfig:1c0c0a1}. By \Cref{subfig:1a1a0c1}, $v$ cannot share a common 2-neighbor with a $(1,1,0^+)$-vertex at distance 2 so \ru2 does not apply. Hence, by \ru0 and \ru1, we have
$$ \mu^*(v) = \frac32 - 1 - \frac12 = 0.$$

\item If $v$ see a $(1,1,1)$-vertex at distance 2, $v$ can only see exactly one such vertex. By \Cref{subfig:1a1a0c1}, $v$ cannot have $(1,1,0)$-neighbor so \ru1 does not apply. Thus, by \ru0 and \ru2, we have
$$ \mu^*(v) = \frac32 - 1 - \frac12 = 0.$$

\item If $v$ does not have a $(1,1,0)$-neighbor and does not see a $(1,1,1)$-vertex at distance 2, then only \ru0 applies and we have
$$ \mu^*(v) = \frac32 - 1 = \frac12.$$
\end{itemize}

\item If $v$ is a $(0,0,0)$-vertex.\\
Observe that \ru0 and \ru2 cannot apply since $v$ does not have any 2-neighbor and cannot see a $(1,1,1)$-vertex at distance 2. So, only \ru1 can apply and by \Cref{subfig:1c0b0c1}, $v$ cannot have three $(1,1,0)$-neighbors. Consequently, 
\begin{itemize}
\item if $v$ has exactly two $(1,1,0)$-neighbors, then we have $$ \mu^*(v)=\frac32 - 2\cdot \frac12 = \frac12.$$
\item if $v$ has exactly one $(1,1,0)$-neighbor, then we have $$ \mu^*(v)=\frac32 - \frac12 = 1.$$
\item if $v$ has no $(1,1,0)$-neighbor, then we have $$ \mu^*(v)=\frac32.$$
\end{itemize}

\end{itemize}

Below, we recapitulate the remaining charges of each type of $3$-vertex $v$ (as 2-vertices are at 0) after applying \ru0, \ru1, and \ru2. In \Cref{1110,1100,10012,1000,00032,0001,00012}, the 2-vertices will be filled while the 3-vertices will not be. 

\begin{figure}[H]
\centering
\begin{minipage}[t]{.32\textwidth}
\centering
\begin{tikzpicture}
\begin{scope}[every node/.style={circle,thick,draw,minimum size=1pt,inner sep=2}]
    \node (c) at (-5.5,0) {$w_1$};
    \node (t) at (-4,0) {$v$};
    
    \node[fill] (t0) at (-4.75,0) {};
    
    \node[fill] (t1) at (-3.5,0.5) {};
    \node (t12) at (-2.75,0.5) {$w_2$};
    
    \node[fill] (t2) at (-3.5,-0.5) {};
    \node (t22) at (-2.75,-0.5) {$w_3$};
\end{scope}

\begin{scope}[every edge/.style={draw=black,thick}]
    \path (c) edge (t);
    
    \path (t) edge (t1);
    \path (t1) edge (t12);
    
    \path (t) edge (t2);
    \path (t2) edge (t22);
\end{scope}
\end{tikzpicture}
\caption{(1,1,1).\\ $w_1,w_2,w_3\neq (1,1,0^+)$.\\ $\mu^*(v)=0$.}
\label{1110}
\end{minipage}
\hfill
\begin{minipage}[t]{.32\textwidth}
\centering
\begin{tikzpicture}
\begin{scope}[every node/.style={circle,thick,draw,minimum size=1pt,inner sep=2}]
    \node (c) at (-4.75,0) {$u_1$};
    \node (t) at (-4,0) {$v$};
    
    \node[fill] (t1) at (-3.5,0.5) {};
    \node (t12) at (-2.75,0.5) {$w_2$};
    
    \node[fill] (t2) at (-3.5,-0.5) {};
    \node (t22) at (-2.75,-0.5) {$w_3$};
\end{scope}

\begin{scope}[every edge/.style={draw=black,thick}]
    \path (c) edge (t);
    
    \path (t) edge (t1);
    \path (t1) edge (t12);
    
    \path (t) edge (t2);
    \path (t2) edge (t22);
\end{scope}
\end{tikzpicture}
\caption{(1,1,0).\\ $u_1\neq (1,1,0)$ and $w_2,w_3\neq (1,1,1)$.\\ $\mu^*(v)=0$.}
\label{1100}
\end{minipage}
\hfill
\begin{minipage}[t]{.32\textwidth}
\centering
\begin{tikzpicture}
\begin{scope}[every node/.style={circle,thick,draw,minimum size=1pt,inner sep=2}]
    \node (c) at (-5.5,0) {$w_1$};
    \node (t) at (-4,0) {$v$};
    
    \node[fill] (t0) at (-4.75,0) {};
    
    \node (t1) at (-3.5,0.5) {$u_2$};
    
    \node (t2) at (-3.5,-0.5) {$u_3$};
\end{scope}

\begin{scope}[every edge/.style={draw=black,thick}]
    \path (c) edge (t);
    
    \path (t) edge (t1);
    
    \path (t) edge (t2);
\end{scope}
\end{tikzpicture}
\caption{(1,0,0).\\ $w_1\neq (1,1,1)$ and $u_2,u_3\neq (1,1,0)$.\\ $\mu^*(v)=\frac12$.}
\label{10012}
\end{minipage}
\end{figure}

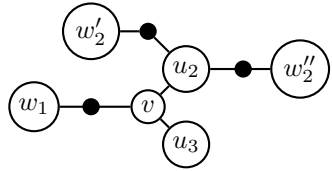
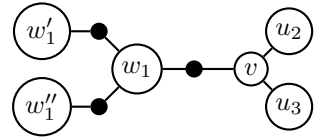
\begin{figure}[H]
\centering
\subfloat[$w_1\neq (1,1,0^+)$ and $u_3\neq (1,0^+,0)$.]{
\label{0002a}
\begin{tikzpicture}
\begin{scope}[every node/.style={circle,thick,draw,minimum size=1pt,inner sep=2}]
    \node[fill] (c) at (-4.75,0) {};
    \node (c1) at (-5.5,0) {$w_1$};
    \node (t) at (-4,0) {$v$};
    
    \node (t1) at (-3.5,0.5) {$u_2$};
    \node[fill] (t11) at (-2.75,0.5) {};
    \node (t12) at (-2,0.5) {$w''_2$};
    
    \node (t2) at (-3.5,-0.5) {$u_3$};
    
    \node[fill] (t3) at (-4,1) {};
    \node (t31) at (-4.75,1) {$w'_2$};
\end{scope}

\begin{scope}[every edge/.style={draw=black,thick}]
    \path (c1) edge (t);
    
    \path (t) edge (t1);
    \path (t1) edge (t12);
    
    \path (t) edge (t2);

    \path (t1) edge (t3);
    \path (t3) edge (t31);
\end{scope}
\end{tikzpicture}
}
\hspace*{0.4cm}
\subfloat[$u_2,u_3 \neq (1,1,0)$.]{
\label{0002b}
\begin{tikzpicture}
\begin{scope}[every node/.style={circle,thick,draw,minimum size=1pt,inner sep=2}]
    \node[fill] (c) at (-4.75,0) {};
    \node (c1) at (-5.5,0) {$w_1$};
    \node (t) at (-4,0) {$v$};
    
    \node (t1) at (-3.5,0.5) {$u_2$};
    \node (t2) at (-3.5,-0.5) {$u_3$};
    
    \node[fill] (c21) at (-6,0.5) {};
    \node (c22) at (-6.75,0.5) {$w'_1$};
    
    \node[fill] (c31) at (-6,-0.5) {};
    \node (c32) at (-6.75,-0.5) {$w''_1$};
\end{scope}

\begin{scope}[every edge/.style={draw=black,thick}]
    \path (c1) edge (t);    
    \path (t) edge (t1);
    \path (t) edge (t2);
	\path (c1) edge (c21);
	\path (c21) edge (c22);
	\path (c1) edge (c31);
	\path (c31) edge (c32);    
\end{scope}
\end{tikzpicture}
}
\caption{(1,0,0).\\ $\mu^*(v)=0$.}
\label{1000}
\end{figure}

\begin{figure}[H]
\begin{minipage}[b]{0.32\textwidth}
\centering
\begin{tikzpicture}
\begin{scope}[every node/.style={circle,thick,draw,minimum size=1pt,inner sep=2}]
    \node (c) at (-4.75,0) {$u_1$};
    \node (t) at (-4,0) {$v$};
    
    \node (t1) at (-3.5,0.5) {$u_2$};
    \node (t2) at (-3.5,-0.5) {$u_3$};
\end{scope}

\begin{scope}[every edge/.style={draw=black,thick}]
    \path (c) edge (t);
    \path (t) edge (t1);
    \path (t) edge (t2);
\end{scope}
\end{tikzpicture}
\caption{(0,0,0).\\ $u_1,u_2,u_3\neq (1,1,0)$.\\ $\mu^*(v)=\frac32$.}
\label{00032}
\end{minipage}
\begin{minipage}[b]{0.32\textwidth}
\centering
\begin{tikzpicture}
\begin{scope}[every node/.style={circle,thick,draw,minimum size=1pt,inner sep=2}]
    \node(c) at (-4.75,0) {$u_1$};
    \node (t) at (-4,0) {$v$};
    
    \node (t1) at (-3.5,0.5) {$u_2$};
    \node[fill] (t11) at (-2.75,0.5) {};
    \node (t12) at (-2,0.5) {$w''_2$};
    
    \node (t2) at (-3.5,-0.5) {$u_3$};
    
    \node[fill] (t3) at (-4,1) {};
    \node (t31) at (-4.75,1) {$w'_2$};
\end{scope}

\begin{scope}[every edge/.style={draw=black,thick}]
    \path (c) edge (t);
    
    \path (t) edge (t1);
    \path (t1) edge (t12);
    
    \path (t) edge (t2);

    \path (t1) edge (t3);
    \path (t3) edge (t31);
\end{scope}
\end{tikzpicture}

\caption{(0,0,0).\\ $u_1,u_3\neq (1,1,0)$.\\$\mu^*(v)=1$.}
\label{0001}
\end{minipage}
\begin{minipage}[b]{0.32\textwidth}
\centering
\begin{tikzpicture}
\begin{scope}[every node/.style={circle,thick,draw,minimum size=1pt,inner sep=2}]
    \node(c) at (-4.75,0) {$u_1$};
    \node (t) at (-4,0) {$v$};
    
    \node (t1) at (-3.5,0.5) {$u_2$};
    \node[fill] (t11) at (-2.75,0.5) {};
    \node (t12) at (-2,0.5) {$w''_2$};
    
    \node (t2) at (-3.5,-0.5) {$u_3$};
    
    \node[fill] (t3) at (-4,1) {};
    \node (t31) at (-4.75,1) {$w'_2$};
    
    \node[fill] (t41) at (-2.75,-0.5) {};
    \node (t42) at (-2,-0.5) {$w''_3$};
    
    \node[fill] (t51) at (-4,-1) {};
    \node (t52) at (-4.75,-1) {$w'_2$};
\end{scope}

\begin{scope}[every edge/.style={draw=black,thick}]
    \path (c) edge (t);
    
    \path (t) edge (t1);
    \path (t1) edge (t12);
    
    \path (t) edge (t2);
   
    \path (t1) edge (t3);
    \path (t3) edge (t31);
    
    \path (t2) edge (t42);
    \path (t2) edge (t51);
    \path (t51) edge (t52);
\end{scope}
\end{tikzpicture}
\caption{(0,0,0).\\ $u_1\neq (1,1,0)$.\\$\mu^*(v)=\frac12$.}
\label{00012}
\end{minipage}
\end{figure}

\begin{table}[!h]
\centering
\begin{tabular}{|c||c|c|c|c|c|}
 \hline
 \backslashbox{$v$ \kern-0.5em}{\kern-0.5em $\mu^*(v)$}& \Large $\frac32$ & \Large 1 & \Large $\frac12$ & \Large 0 \\
 \hline \hline
 (1,1,1) & \slashbox{}{} & \slashbox{}{} & \slashbox{}{} & \cref{1110} \\
 \hline
 (1,1,0) & \slashbox{}{} & \slashbox{}{} & \slashbox{}{} & \cref{1100} \\
 \hline
 (1,0,0) & \slashbox{}{} & \slashbox{}{} & \cref{10012} & \cref{1000} \\
 \hline
 (0,0,0) & \cref{00032} & \cref{0001} & \cref{00012} & \slashbox{}{}  \\
  \hline
\end{tabular}
\caption{Available amount of charges for each type of 3-vertex after applying \ru0-\ru2}
\label{tab:recap}
\end{table}

\subsection{Second round: vertices to faces}
\label{subsec:second_round_discharging}

Recall that $\mu^*(v)$ is the remaining charges of $v$ after applying rules \ru0-\ru2. We define the following discharging rules between the vertices and $8$-faces of $G$:
\begin{itemize}
\item[\ru3] If a 3-vertex $v$ is not a $(1,0,0)$-vertex, then it gives $\frac{\mu^*(v)}{n_1}$ to each incident $8$-face, where $n_1$ is the number of incident $8$-faces.
\item[\ru4] For a $(1,0,0)$-vertex $v$, let $n_2$ be the number of $8$-faces incident to $v$ and to its $2$-neighbor. Vertex $v$ gives $\frac{\mu^*(v)}{n_2}$ to each of these $n_2$ $8$-faces. 
\end{itemize}

Recall that, given a face $f$, the initial amount of charge $\mu(f) = d(f)-9$ so all $k$-faces with $k\geq 9$ have a positive charge. Moreover, after applying \ru3-\ru4, every $3$-vertex $v$ will have a remaining charge of at least $\mu^*(v)-n_i\cdot\frac{\mu^*(v)}{n_i}=0$ for $1\leq i\leq 2$.

As a result, it remains to verify that every $8$-face $f$ will receive at least charge 1 so that its final charge will be $\mu^*(f)\geq \mu(f)-9+1 = 8-8=0$.

To generate every possible 8-face efficiently, we introduce the following encoding of a configuration around an 8-face. 

\textbf{Encoding a face $f$:}
\begin{itemize}
\item For every pair of consecutive $3$-vertices in clockwise order, count the number of $2$-vertices in between. We obtain a circular sequence of integers in clockwise order of length equal to the number of $3$-vertices of $f$. Since $G$ has no $2^+$-paths by \Cref{lemma:no 2-path}, each integer is in $\{0,1\}$. Observe that there are at most as many ways to write this sequence of integers as the number of $3$-vertices of $f$. Indeed, we can choose any $3$-vertex $v$ as a starting point and start counting the number of $2$-vertices between $v$ and the next $3$-vertex in clockwise order. We choose as representative the first one in the lexicographic order where \texttt{1} precedes \texttt{0} and call it the \textbf{\textit{number-word}} of $f$.

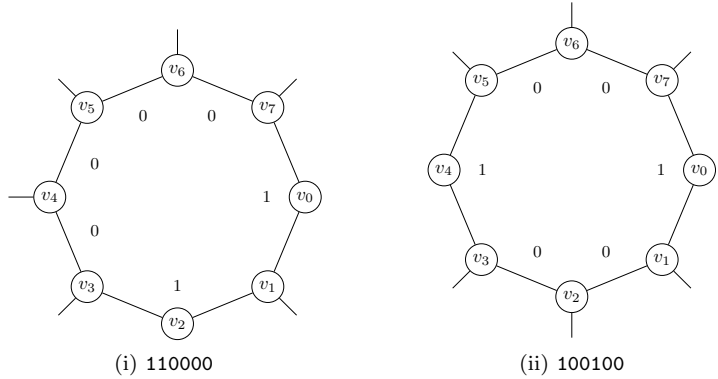
\begin{figure}[H]
\centering
\subfloat[\texttt{110000}]{\label{subfig:number_word1}
\scalebox{0.7}{
\begin{tikzpicture}[join=bevel,vertex/.style={circle,draw, minimum size=0.6cm},inner sep=0mm,scale=0.8]
\foreach \i in {0,...,7}  \node[vertex] (\i) at (-45.0*\i:3) {$v_\i$};
\foreach \i in {0,...,7}  \draw let \n1={int(mod(\i+1,8))} in (\i) -- (\n1);
\node (u1) at (-45.0*1:4) {};
\draw (1) -- (u1);
\node (u3) at (-45.0*3:4) {};
\draw (3) -- (u3);
\node (u4) at (-45.0*4:4) {};
\draw (4) -- (u4);
\node (u5) at (-45.0*5:4) {};
\draw (5) -- (u5);
\node (u6) at (-45.0*6:4) {};
\draw (6) -- (u6);
\node (u7) at (-45.0*7:4) {};
\draw (7) -- (u7);

\node at (-45.0*0:2.1) {\small $1$};
\node at (-45.0*2:2.1) {\small $1$};
\node at (-45.0*3.5:2.1) {\small $0$};
\node at (-45.0*4.5:2.1) {\small $0$};
\node at (-45.0*5.5:2.1) {\small $0$};
\node at (-45.0*6.5:2.1) {\small $0$};

\end{tikzpicture}
}
}
\qquad
\subfloat[\texttt{100100}]{\label{subfig:number_word2}
\scalebox{0.7}{
\begin{tikzpicture}[join=bevel,vertex/.style={circle,draw, minimum size=0.6cm},inner sep=0mm,scale=0.8]
\foreach \i in {0,...,7}  \node[vertex] (\i) at (-45.0*\i:3) {$v_\i$};
\foreach \i in {0,...,7}  \draw let \n1={int(mod(\i+1,8))} in (\i) -- (\n1);
\node (u1) at (-45.0*1:4) {};
\draw (1) -- (u1);
\node (u2) at (-45.0*2:4) {};
\draw (2) -- (u2);
\node (u3) at (-45.0*3:4) {};
\draw (3) -- (u3);
\node (u'5) at (-45.0*5:5.5) {};
\draw (5) -- (u5);
\node (u6) at (-45.0*6:4) {};
\draw (6) -- (u6);
\node (u'7) at (-45.0*7:5.5) {};
\draw (7) -- (u7);
\node at (-45.0*0:2.1) {\small $1$};
\node at (-45.0*1.5:2.1) {\small $0$};
\node at (-45.0*2.5:2.1) {\small $0$};
\node at (-45.0*4:2.1) {\small $1$};
\node at (-45.0*5.5:2.1) {\small $0$};
\node at (-45.0*6.5:2.1) {\small $0$};
\end{tikzpicture}
}
}
\caption{Examples of number-words on $8$-faces.}
\end{figure}

Examples:
\begin{itemize}
\item Take the $8$-face in \Cref{subfig:number_word1} as an example. We consider the $3$-vertices in clockwise order starting at any $3$-vertex, say $v_1$. We get $v_1$, $v_3$, $v_4$, $v_5$, $v_6$, $v_7$. Now, we count the number of $2$-vertices between two consecutives vertices in that sequence. More precisely, there is one $2$-vertex ($v_2$) in between $v_1$ and $v_3$, then none between $v_3$ and $v_4$, and so on. This gives us the sequence of numbers \texttt{100001}. Had we chosen another starting $3$-vertex (say $v_3$) we would have obtained another sequence (\texttt{000011}). Among all of these different sequences, we choose the one that comes first in the lexicographic order where 1 comes before 0. And that sequence is \texttt{110000}, the number-word of $f$, which corresponds to the starting $3$-vertex $v_7$.
\item We can do the same with the $8$-face in \Cref{subfig:number_word2}. The number-word for $f$ is \texttt{100100}. Observe that this sequence can be obtained by taking, in clockwise order, either $v_7$ or $v_3$ as a starting point.
\end{itemize}

\item Due to our discharging rules, we are interested in configurations around $3$-vertices. So, given a $3$-vertex $v$ on $f$, we choose the following letters to encode the neighborhood outside $f$ of $v$:
\begin{itemize}
\item \texttt{c} means that $v$ has a $2$-neighbor outside $f$.
\item \texttt{b} means that $v$ has a $(1,1,0)$-neighbor outside $f$.
\item \texttt{a} represents the rest of the possible neighbors of $v$. In other words, the neighbor of $v$ outside $f$ is a $3$-vertex that is not a $(1,1,0)$-vertex. 
\end{itemize}

Observe that there may be multiple starting $3$-vertices that give the same number-word for $f$. Given one possible starting $3$-vertex of the number-word $nw$, we insert between each pair of consecutive integers of $nw$ the letter encoding of the neighborhood outside $f$ of the corresponding $3$-vertex. We obtain an alternating sequence $fw$ of integers and letters for each starting $3$-vertex.

Among the possible alternating sequences $fw$s, we choose the one where the subsequence of letters is the smallest in alphabetical order. We call this alternating sequence the \textbf{\textit{full-word}} of $f$ and the corresponding subsequence of letters the \textbf{\textit{letter-word}} of $f$.

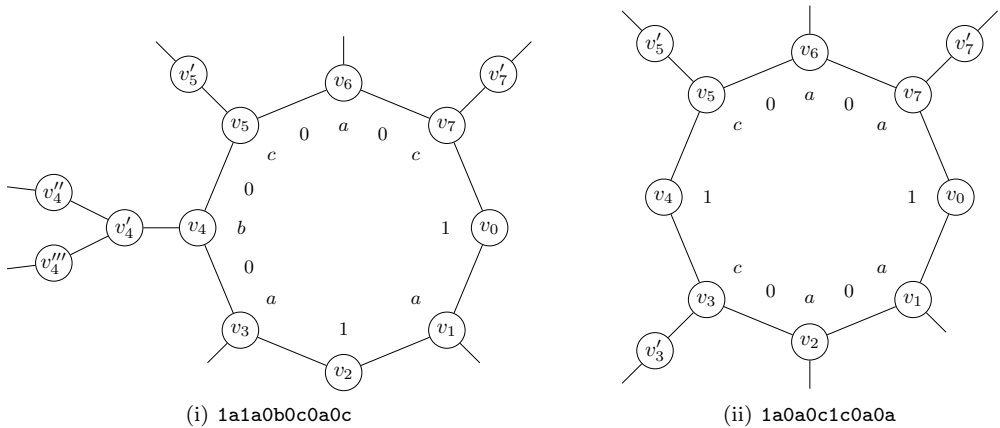
\begin{figure}[H]
\centering
\subfloat[\texttt{1a1a0b0c0a0c}]{\label{subfig:encoding1}
\scalebox{0.8}{
\begin{tikzpicture}[join=bevel,vertex/.style={circle,draw, minimum size=0.6cm},inner sep=0mm,scale=0.8]
\foreach \i in {0,...,7}  \node[vertex] (\i) at (-45.0*\i:3) {$v_\i$};
\foreach \i in {0,...,7}  \draw let \n1={int(mod(\i+1,8))} in (\i) -- (\n1);
\node (u1) at (-45.0*1:4) {};
\draw (1) -- (u1);
\node (u3) at (-45.0*3:4) {};
\draw (3) -- (u3);
\node[vertex] (v'4) at (-45.0*4:4.5) {$v'_4$};
\node[vertex] (v''4) at (-45.0*4-7:6) {$v''_4$};
\node[vertex] (v'''4) at (-45.0*4+7:6) {$v'''_4$};
\node (u''4) at (-45.0*4-7:7) {};
\node (u'''4) at (-45.0*4+7:7) {};
\node at (-45.0*4-11:4.5) {};
\node at (-45.0*4-14:6) {};
\node at (-45.0*4+14:6) {};
\draw (4) -- (v'4);
\draw (v'4) -- (v''4);
\draw (v'4) -- (v'''4);
\draw (v''4) -- (u''4);
\draw (v'''4) -- (u'''4);
\node[vertex] (v'5) at (-45.0*5:4.5) {$v'_5$};
\node (u'5) at (-45.0*5:5.5) {};
\node at (-45.0*5-11:4.5) {};
\draw (5) -- (v'5);
\draw (v'5) -- (u'5);
\node (u6) at (-45.0*6:4) {};
\draw (6) -- (u6);
\node[vertex] (v'7) at (-45.0*7:4.5) {$v'_7$};
\node (u'7) at (-45.0*7:5.5) {};
\node at (-45.0*7-11:4.5) {};
\draw (7) -- (v'7);
\draw (v'7) -- (u'7);

\node at (-45.0*0:2.1) {\small $1$};
\node at (-45.0*2:2.1) {\small $1$};
\node at (-45.0*3.5:2.1) {\small $0$};
\node at (-45.0*4.5:2.1) {\small $0$};
\node at (-45.0*5.5:2.1) {\small $0$};
\node at (-45.0*6.5:2.1) {\small $0$};
\node at (-45.0*1:2.1) {\small $a$};
\node at (-45.0*3:2.1) {\small $a$};
\node at (-45.0*4:2.1) {\small $b$};
\node at (-45.0*5:2.1) {\small $c$};
\node at (-45.0*6:2.1) {\small $a$};
\node at (-45.0*7:2.1) {\small $c$};

\end{tikzpicture}
}
}
\qquad
\subfloat[\texttt{1a0a0c1c0a0a}]{\label{subfig:encoding2}
\scalebox{0.8}{
\begin{tikzpicture}[join=bevel,vertex/.style={circle,draw, minimum size=0.6cm},inner sep=0mm,scale=0.8]
\foreach \i in {0,...,7}  \node[vertex] (\i) at (-45.0*\i:3) {$v_\i$};
\foreach \i in {0,...,7}  \draw let \n1={int(mod(\i+1,8))} in (\i) -- (\n1);
\node (u1) at (-45.0*1:4) {};
\draw (1) -- (u1);
\node (u2) at (-45.0*2:4) {};
\draw (2) -- (u2);
\node[vertex] (v'3) at (-45.0*3:4.5) {$v'_3$};
\node (u'3) at (-45.0*3:5.5) {};
\draw (3) -- (v'3);
\draw (v'3) -- (u'3);
\node[vertex] (v'5) at (-45.0*5:4.5) {$v'_5$};
\node (u'5) at (-45.0*5:5.5) {};
\draw (5) -- (v'5);
\draw (v'5) -- (u'5);
\node (u6) at (-45.0*6:4) {};
\draw (6) -- (u6);
\node[vertex] (v'7) at (-45.0*7:4.5) {$v'_7$};
\node (u'7) at (-45.0*7:5.5) {};
\draw (7) -- (v'7);
\draw (v'7) -- (u'7);
\node at (-45.0*0:2.1) {\small $1$};
\node at (-45.0*1:2.1) {\small $a$};
\node at (-45.0*1.5:2.1) {\small $0$};
\node at (-45.0*2:2.1) {\small $a$};
\node at (-45.0*2.5:2.1) {\small $0$};
\node at (-45.0*3:2.1) {\small $c$};
\node at (-45.0*4:2.1) {\small $1$};
\node at (-45.0*5:2.1) {\small $c$};
\node at (-45.0*5.5:2.1) {\small $0$};
\node at (-45.0*6:2.1) {\small $a$};
\node at (-45.0*6.5:2.1) {\small $0$};
\node at (-45.0*7:2.1) {\small $a$};
\end{tikzpicture}
}
}
\caption{Examples of full-words on $8$-faces.}
\end{figure}

Examples:
\begin{itemize}
\item Take the $8$-face $f$ in \Cref{subfig:encoding1} as an example. It is the same face as in \Cref{subfig:number_word1}, this time with more information about the neighborhood of the $3$-vertices outside of $f$. Observe that when we do not have extra information about the neighborhood of a $3$-vertex outside of $f$ (it could be \texttt{a}, \texttt{b}, or \texttt{c}), we will denote it \texttt{a} for now and explain it later on. We consider the neighborhood of each $3$-vertex, starting with the one that comes right after the first number, which is the $3$-vertex $v_1$. In order, they corresponds to the letters \texttt{a}, \texttt{a}, \texttt{b}, \texttt{c}, \texttt{a}, \texttt{c}, which give us the letter-word \texttt{aabcac}. Finally, we combine these the number-word and the letter-word into the full-word \texttt{1a1a0b0c0a0c}. 
\item We can do the same with the $8$-face in \Cref{subfig:encoding2}, which is the face in \Cref{subfig:number_word2} with extra information. When we choose the letter-word for $f$, we need to consider two encodings, one that starts with the $3$-vertex that comes right after $v_7$ in clockwise order, namely $v_1$, or the one after $v_3$, namely $v_5$. These give us two sequence of letters \texttt{aaccaa} and \texttt{caaaac} respectively. For our letter-word, we choose the first one in alphabetical order, which is \texttt{aaccaa}. Finally, we get the full-word \texttt{1a0a0c1c0a0a}.
\end{itemize}
\end{itemize}

\begin{observation}
Each face has a unique encoding full-word and each full-word uniquely defines a face.
\end{observation}

Under each $8$-cycle of \Cref{fig:first_reducible_cycles,fig:reducible_cycles,subfig:1c1a0a1a0a}, you have the corresponding encoding of the reducible configuration if it were an $8$-face.

\bigskip

In what follows we explain the generation of all possible $8$-faces, how to check which ones are reducible and which ones will obtain enough charge from its incident $3$-vertices by \ru3 and \ru4. The corresponding pseucode is summarized in \Cref{algorithm}.

\begin{algorithm}[H]
\SetAlgoLined
\KwData{\texttt{forbidden\_subwords}, \texttt{dictionary\_of\_charges}, \texttt{number\_words}, \texttt{alphabet}, \texttt{target\_charge}.}
\KwResult{The list of full-words that are not forbidden nor dischargeable.}
 \ForEach{\texttt{number\_word} $\in$ \texttt{number\_words}}{
  \texttt{n} = length of \texttt{number\_word}\;
  \texttt{letter\_words} = set of words of size \texttt{n} in \texttt{alphabet}\;
  \ForEach{\texttt{letter\_word} $\in$ \texttt{letter\_words}}{
	  build \texttt{full\_word} from \texttt{number\_word} and \texttt{letter\_word}\;
	  \If{\texttt{full\_word} does not contain a subword in \texttt{forbidden\_subwords}}{
		Compute the \texttt{charge} of \texttt{full\_word} using \texttt{dictionary\_of\_charges}\; 	  
	  \If{\texttt{charge} $<$ \texttt{target\_charge}}{
   		Write \texttt{full\_word} to output\;
   	  }
   	  }
  }
 }
\caption{\label{algorithm}Filtering forbidden and dischargeable full-words corresponding to faces with a given size.}
\end{algorithm}

Since $G$ has no $2^+$-paths and $f$ has length 8, there can be at most four $2$-vertices on $f$. On the other hand, given a number-word $nw$ of $f$, the number of $2$-vertices of $f$ is given by the number of \texttt{1}s in $nw$. Therefore, one can easily check the following observation: 

\begin{observation}
The only possible number-words for $8$-faces in $G$ are \texttt{1111}, \texttt{11100}, \texttt{11010}, \texttt{110000}, \texttt{101000}, \texttt{100100}, \texttt{1000000}, and \texttt{00000000}. 
\end{observation}

Since the process of generating these number-words is done naively and it is not the main focus of the algorithm, we will not go into technical details. However, the script is available at \url{https://gite.lirmm.fr/discharging/planar-graphs}. For this case, the set of number-words is small enough that it can even be checked manually.

Now, for each number-word $nw$, we can generate all possible sequences of letters in $\{a,b,c\}$ with the same length as $nw$ that we will then interlace with $nw$ to create an alternating sequence corresponding to a full-word (line 5 of \Cref{algorithm}). Observe that during this process of generation, we may obtain several words representing the same face and only one of them is the unique full-word encoding $f$. This has no influence on the correctness of our algorithm, only the time complexity, as some faces might be checked multiple times. Here, it is possible to identify the symmetries in the generated words in order to keep the unique full-words. However, in practice, at least for our case, this subroutine adds complications with minimal time gain.

\bigskip

The list of full-words described above corresponds to all possible neighborhoods at distance at most $2$ of an $8$-face. We filter out every neighborhood that either contains a reducible configuration of \Cref{lemma:first_reducible_cycles,lemma:reducible,lemma:reducible_cycles,lemma:special_reducible} (line 6 of \Cref{algorithm}), or has enough charge available for its $8$-face by \ru3 and \ru4 (line 8 of \Cref{algorithm}).

\medskip

In order to check that the corresponding subgraph of a full-word contains a reducible configuration, we encode the latter using similar conventions as for the neighborhood of the $8$-faces. Indeed, the considered configuration is encoded as seen from an incident face. Thus, one configuration may have multiple different encodings (depending on the incident faces) and we call these encodings \textbf{\textit{forbidden subwords}}. A full-word that contains a forbidden subword is \textbf{\textit{forbidden}}. 

Since we always consider the worst case scenario, if a forbidden subword contains a letter \texttt{a}, then one can always build two other (``weaker'') forbidden subwords by replacing this \texttt{a} by \texttt{b} or \texttt{c}. Therefore, whenever we consider a forbidden subword containing \texttt{a}, we also implicitly consider the other ``weaker'' subwords. See \Cref{fig:first_reducible_cycles,fig:reducible_cycles,fig:reducible,fig:special_reducible} where the captions contain all possible forbidden (``strong'') subwords of each reducible configuration. In a general case, one can define a different symbol (another letter, say $d$ for example) that can be rewritten as multiple different letters (here \texttt{a}, \texttt{b}, and \texttt{c}). Our choice was \texttt{a} for simplicity.

In the code implementation of \Cref{algorithm}, we define a forbidden subword as a regular expression and rewriting rule (formal grammar) in which \texttt{a} can be rewritten as \texttt{b} or \texttt{c}.

\begin{observation}
In a forbidden subword, \texttt{a} can mean \texttt{a}, \texttt{b}, or \texttt{c} in a real encoding.
\end{observation} 

Now, recall that a full-word is actually circular and is read in clockwise order. Thus, in order to check whether it is forbidden, one has to check if it contains a forbidden subword or its mirror. Once we removed the forbidden subword, we are ready to move to the next step of the algorithm.

\bigskip

The next step (lines 7-8 of \Cref{algorithm}) is to check, for every full-word $fw$, whether the $3$-vertices of the corresponding subgraph give enough charge to $f$ according to \ru3 and \ru4 (at least a total charge 1). If it is the case, we say that $fw$ is \textbf{\textit{dischargeable}}. Similar to the encoding of the reducible configurations, we can also encode into a dictionary the configurations from \Cref{1110,1100,10012,1000,00032,0001,00012}. The encoding of each entry of the dictionary corresponds to a possible neighborhood of a $3$-vertex, along with $\frac{\mu^*(v)}{3}$ for the worst case scenario in \ru3 (\Cref{1110,1100,00032,0001,00012}) and $\frac{\mu^*(v)}{2}$ for \ru4 (\Cref{10012,1000}). To work with integers, we multiply by 12 the charge of each vertex and each face of $G$. In \Cref{table:dictionary}, we detail the dictionary entries for each configuration.

\begin{table}[H]
\centering
\begin{tabular}{|c|c|c|c|c|c|c|c|}
\hline
\cref{1110} & \cref{1100} & \cref{10012} & \cref{1000}i & \cref{1000}ii & \cref{00032} & \cref{0001} & \cref{00012} \\
\hline
\texttt{1c1} : 0 & \texttt{1a1} : 0 & \texttt{1a0} : 3 & \texttt{1b0} : 0 & \texttt{0a1c1} : 0 & \texttt{0a0} : 6 & \texttt{0b0} : 4 & \texttt{0b0c1} : 2 \\
& \texttt{1c0} : 0 & \texttt{0c0} : 0 & \texttt{1a0c1} : 0 & & & \texttt{0a0c1} : 4 & \texttt{1c0a0c1} : 2\\
\hline
\end{tabular}
\caption{\label{table:dictionary}The dictionary of charges. Each entry is written as ``\texttt{<encoding>} : <charge>''. \\Every value was multiplied by 12 to get an integer.}
\end{table}

Observe that, in our case, every encoding in a dictionary entry starts and ends with a number. Thus, we have the following observation.

\begin{observation}
The encoding in a dictionary entry always has odd length.
\end{observation}

As a consequence, the $3$-vertex $v$ that holds the charge in the encoding of a dictionary entry corresponds to either 
\begin{itemize}
\item the letter in the middle when it has length $3$ or $7$,
\item or the letter in second position when it has length $5$. 
\end{itemize}

Once again, each encoding can be read from left to right or right to left. Note that one has to be mindful of the position of $v$ when reading an encoding of length $5$ from right to left.
 
In order to count the total amount of charge that an $8$-face will receive from its $3$-vertices, the algorithm consists of sliding a window of odd length across the circular full-word. We start with the window of the largest possible length ($7$ according to our dictionary) in order to have the most information about the neighborhood of $v$. At each step, it searches for the corresponding encoding (or its mirror) in the dictionary and if it exists, it marks the position as ``discharged'' and adds the corresponding amount of charge to its total amount. For a given window size, if the corresponding subword is not in the dictionary, then it means that the dictionary entry corresponding to $v$ must have an encoding of smaller length (recall that the dictionary entries are exhaustive). Then, it suffices to verify that the total amount is at least 12 (\texttt{target\_charge}) since we multiplied every charge by 12. In such a case, we know that our $8$-face will end up with a non-negative amount of charge. 

\subsection{Third round: faces to faces}

We ran \Cref{algorithm} to compute the outcome of the second round of discharging. The only remaining type of face which was output by the algorithm (full-word: \texttt{1c1a0a1a0a}) corresponds to the face $f$ in \Cref{fig:ru5}. We define another discharging rule \ru5 to take care of this last case.

\begin{itemize}
\item[\ru5] Let $f$ and $f'$ be as depicted in \Cref{fig:ru5}. If $f'$ is an $8$-face, then $f'$ gives $\frac12$ to $f$.
\end{itemize}

\begin{figure}[H]
\centering
\scalebox{0.65}{%
\begin{tikzpicture}[join=bevel,vertex/.style={circle,draw, minimum size=0.6cm},inner sep=0mm,scale=0.8]
\foreach \i in {0,...,7}  \node[vertex] (\i) at (-45.0*\i:3) {$v_\i$};
\foreach \i in {0,...,7}  \draw let \n1={int(mod(\i+1,8))} in (\i) -- (\n1);
\node[vertex] (v'1) at (-45.0*1:4.5) {$v'_1$};
\node (u'1) at (-45.0*1:5.5) {};
\draw (1) -- (v'1);
\draw (v'1) -- (u'1);
\node[vertex] (v'3) at (-45.0*3:4.5) {$v'_3$};
\draw (3) -- (v'3);
\node (u'3) at (-45.0*3-7:5.5) {};
\node (u''3) at (-45.0*3+7:5.5) {};
\draw (v'3) -- (u'3);
\draw (v'3) -- (u''3);
\node[vertex] (v'4) at (-45.0*4:4.5) {$v'_4$};
\draw (4) -- (v'4);
\node (u'4) at (-45.0*4-7:5.5) {};
\node (u''4) at (-45.0*4+7:5.5) {};
\draw (v'4) -- (u'4);
\draw (v'4) -- (u''4);
\node[vertex] (v'6) at (-45.0*6:4.5) {$v'_6$};
\draw (6) -- (v'6);
\node (u'6) at (-45.0*6-7:5.5) {};
\node (u''6) at (-45.0*6+7:5.5) {};
\draw (v'6) -- (u'6);
\draw (v'6) -- (u''6);
\node[vertex] (v'7) at (-45.0*7:4.5) {$v'_7$};
\draw (7) -- (v'7);
\node (u'7) at (-45.0*7-7:5.5) {};
\node (u''7) at (-45.0*7+7:5.5) {};
\draw (v'7) -- (u'7);
\draw (v'7) -- (u''7);
\node[draw=none] (f) at (0,0) {\Large $f$};
\node[draw=none] (f') at (-3,3) {\Large $f'$};
\end{tikzpicture}
}
\caption{\label{fig:ru5}$v'_3,v'_4,v'_6,v'_7\neq (1,1,0)$.}
\end{figure}
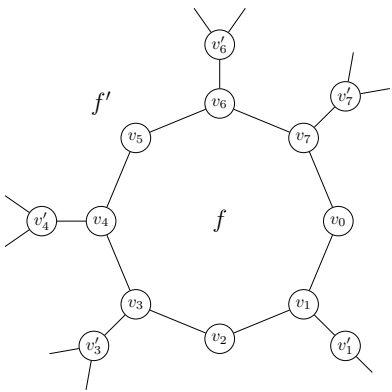

We show that after applying \ru5, we get $\mu^*(f)\geq 0$ and $\mu^*(f')\geq 0$. Recall that $8$-faces have starting charge $-1$.

First of all, by \Cref{10012} and \ru4, if $f'$ is not an $8$-face, then $v_4$ and $v_6$ each give $\frac12$ to $f$. So,
$$ \mu^*(f)\geq -1 + 2\cdot \frac12 = 0.$$ 
If $f'$ is an $8$-face, then $v_4$ and $v_6$ each give $\frac14$ to $f$ by \Cref{10012} and \ru4, and $f'$ gives $f$ $\frac12$ by \ru5. Thus, 
$$ \mu^*(f)\geq -1 + 2\cdot \frac14 + \frac12 = 0.$$ 

Now, let us show that $\mu^*(f')\geq 0$. We know that $f'$ is an $8$-face so $\mu(f')=-1$ and $f'$ gives $\frac12$ to $f$ by \ru5. 

Let $f'=v'_4v_4v_5v_6v'_6v''_6v_8v''_4$. By \Cref{subfig:colorable_1c1a0a1a0a1}, $v''_4$ cannot be a $2$-vertex so it must be a $3$-vertex. Symmetrically, $v''_6$ must also be a $3$-vertex. By \Cref{subfig:colorable_1c1a0a1a0a2}, $v_8$ must also be a $3$-vertex. Observe that \ru5 can thus only apply once to $f'$. Let $u''_4$, $u''_6$, and $u_8$ be the neighbors that do not lie on $f'$ of $v''_4$, $v''_6$, and $v_8$ respectively.

Observe that $v_4$ and $v_6$ each give $\frac14$ to $f'$ by \Cref{10012} and \ru4. Moreover, since $v'_4$ cannot have a $2$-neighbor by \Cref{subfig:colorable_1c1a0a1a0a1}, $v'_4$ gives at least $\frac13$ to $f'$ by \Cref{0001,00032} and \ru3. Symmetrically, the same holds for $v'_6$. We conclude with the following cases:
\begin{itemize}
\item If $u''_4$ (or $u''_6$) is a $3$-vertex, then $v''_4$ (or $v''_6$) gives at least $\frac13$ to $f'$ by \Cref{0001,00032} and \ru3. To sum up,
$$ \mu^*(f') \geq -1 -\frac12 + 2\cdot \frac14 + 3\cdot \frac13  =0. $$
\item If $u''_4$ and $u''_6$ are $2$-vertices, then $u_8$ must be a $3$-vertex by \Cref{subfig:colorable_1c1a0a1a0a3}. In that case, $v_8$ gives at least $\frac13$ to $f'$ by \Cref{0001,00032} and \ru3. To sum up,
$$ \mu^*(f') \geq -1 -\frac12+ 2\cdot \frac14 + 3\cdot \frac13  =0. $$
\end{itemize}

To conclude, we started with a negative total amount of charge on the vertices and faces of $G$ by \Cref{euler} and after our discharging procedures, which preserve the total amount of charge, we ended up with a non-negative amount of charge on each vertex and face of $G$. This is a contradiction, so $G$ does not exists and this ends the proof of \Cref{thm:girth8}.

\section{Discussion on \Cref{thm:girth8}}
\label{sec4}

The discharging method is commonly used on planar graphs, very often because of their sparseness. Sometimes, the planarity of the graph is not needed, in which case the proofs hold for more general classes of sparse graphs (for example, graphs with bounded \emph{maximum average degree} $\mad(G)=\max_{H\subseteq G}\frac{2|E(H)|}{|V(H)|}$). This was the case for the proof of \Cref{thm:girth9} for example. However, when we increase the density of the graph by decreasing the girth of the planar graph or by increasing the maximum average degree, the result might hold for one case but not the other. In particular, the smallest class of graphs of bounded $\mad$ that contains planar graphs with girth at least $g$ has $\mad<\frac{2g}{g-2}$. The clearest example showing that planarity is needed is the Petersen graph with one edge removed: it has $\mad=\frac{14}{5}$ and it needs 8 colors, while all subcubic planar graphs are $7$-colorable~\cite{tho18,har16}. Observe that the class of graphs with $\mad\leq \frac{14}{5}$ does not even contain all planar graphs with girth 6.

For girth 8, the corresponding graphs with bounded $\mad$ verify $\mad<\frac{8}{3}$. If \Cref{thm:girth8} is generalizable to graphs with $\mad<\frac{8}{3}$, then it would be optimal in terms of $\mad$ as the Petersen graph with one vertex removed has $\mad=\frac{8}{3}$ and it needs $7$ colors. On the other hand, it is unclear whether there exists a planar graph with girth 7 needing 7 colors. 

A $5$-cycle with a subdivided chord, which has $\mad=\frac{7}{3}$, shows that a generalization of \Cref{thm:girth8} to graphs with $\mad<\frac{8}{3}$ would also be optimal in terms of number of colors. Once again, it does not mean that there exists a planar graph with girth 8 that needs 6 colors.

In what follows, we provide a planar subcubic construction, with relatively high girth that needs 6 colors. Precisely, we provide a construction of a planar subcubic graph having girth 6 and $
\chi^2\geq 6$.

We call our 5 colors $a$, $b$, $c$, $d$, and $e$.

\begin{lemma} \label{G'}
The graph $G'(u_1,u_2,v_1,v_2)$ in \Cref{fig:G'} has the following properties:
\begin{itemize}
\item $G'(u_1,u_2,v_1,v_2)$ is planar and subcubic.
\item $G'(u_1,u_2,v_1,v_2)$ has girth 6.
\item The distance in $G'(u_1,u_2,v_1,v_2)$ between $u_1$ and $v_1$ is 5.
\item For every $5$-coloring $\phi$ of $G'(u_1,u_2,v_1,v_2)$, if $\phi(u_1)=\phi(v_1)$, then $\phi(u_2)=\phi(v_2)$.
\end{itemize}
\end{lemma} 

\begin{figure}[H]
\centering
\subfloat[\label{fig:G'}The gadget $G'(u_1,u_2,v_1,v_2)$ in \Cref{G'}.]{
\begin{tikzpicture}[scale=0.7]{thick}
\begin{scope}[every node/.style={circle,draw,inner sep=1}]
	\node (u1) at (0,2) {$u_1$};
	\node (u2) at (1,2) {$u_2$};
	\node (u3) at (4,4) {$u_3$};
	\node (u4) at (2,0) {$u_4$};
	\node (w1) at (5,3) {$w_1$};
	\node (w2) at (7,1) {$w_2$};
	\node (x1) at (4,2) {$x_1$};
	\node (x2) at (3,1) {$x_2$};
	\node (y1) at (8,2) {$y_1$};
	\node (y2) at (9,3) {$y_2$};
	\node (v1) at (12,2) {$v_1$};
	\node (v2) at (11,2) {$v_2$};
	\node (v3) at (10,4) {$v_3$};
	\node (v4) at (8,0) {$v_4$};
\end{scope}

\begin{scope}[every edge/.style={draw=black}]
    \path (u1) edge (u2);
    \path (u2) edge (u3);
    \path (u2) edge (u4);
    \path (v4) edge (u4);
    \path (u3) edge (v3);
    \path (v2) edge (v4);
    \path (v2) edge (v3);
    \path (v2) edge (v1);
    \path (u3) edge (w1);
    \path (w1) edge (w2);
    \path (v4) edge (w2);
    \path (x1) edge (w1);
    \path (x1) edge (x2);
    \path (x2) edge (u4);
    \path (w2) edge (y1);
    \path (y2) edge (y1);
    \path (y2) edge (v3);
\end{scope}
\end{tikzpicture}
}
\subfloat[Simplified drawing of $G'(u_1,u_2,v_1,v_2)$.]{
\begin{tikzpicture}[scale=0.8]{thick}
\begin{scope}[every node/.style={circle,draw,inner sep=1}]
	\node (u1) at (0,0) {$u_1$};
	\node (u2) at (1,0) {$u_2$};
	\node (g') at (2.5,0) {\Large $G'$};
	\node (v2) at (4,0) {$v_2$};
	\node (v1) at (5,0) {$v_1$};
\end{scope}

\begin{scope}[every edge/.style={draw=black}]
    \path (u1) edge (u2);
    \path (v2) edge (v1);
    \path (u2) edge[bend left] (g');
    \path (u2) edge[bend right] (g');
    \path (v2) edge[bend left] (g');
    \path (v2) edge[bend right] (g');
\end{scope}
\end{tikzpicture}
}
\caption{}
\end{figure}

\begin{proof}
One can verify that $G'(u_1,u_2,v_1,v_2)$ is planar, subcubic, has girth 6, and that the distance between $u_1$ and $v_1$ is 5 thanks to \Cref{fig:G'}. It remains to prove that if $\phi(u_1)=\phi(v_1)$, then $\phi(u_2)=\phi(v_2)$ for every $5$-coloring $\phi$ of $G'(u_1,u_2,v_1,v_2)$. 

Suppose by contradiction that there exists a $5$-coloring $\phi$ of $G'(u_1,u_2,v_1,v_2)$ such that $\phi(u_1)=\phi(v_1)=a$, but $b=\phi(u_2)\neq \phi(v_2)=c$. We can assume w.l.o.g. that $\phi(u_3)=d$ and $\phi(u_4)=e$. As a result, we have $\phi(v_3)=e$ and $\phi(v_4)=d$. Since $w_1$ sees $u_2$, $u_3$, and $v_3$, $\phi(w_1)\in \{a,c\}$. Since $x_2$ sees $u_2$, $u_4$, and $v_4$, $\phi(x_2)\in \{a,c\}$. Since $y_2$ sees $v_2$, $v_3$, and $u_3$, $\phi(y_2)\in \{a,b\}$. Since $w_2$ sees $v_2$, $v_4$, and $u_4$, $\phi(w_2)\in \{a,b\}$.

\begin{itemize}
\item If $\phi(x_2)=c$, then $\phi(w_1)=a$, $\phi(w_2)=b$, and $\phi(y_2)=a$. However, $x_1$ sees $u_3$, $w_1$, $w_2$, $x_2$, and $u_4$ which are colored $d$, $a$, $b$, $c$, and $e$ respectively. So, $x_1$ is not colorable.
\item If $\phi(x_2)=a$, then $\phi(w_1)=c$.
\begin{itemize}
\item If $\phi(w_2)=b$, then $x_1$ is not colorable since it sees $u_3$, $w_1$, $w_2$, $x_2$, and $u_4$ which are colored $d$, $c$, $b$, $a$, and $e$ respectively.
\item If $\phi(w_2)=a$, then $\phi(y_2)=b$ and $y_1$ is not colorable since it sees $w_1$, $w_2$, $v_4$, $y_2$, and $v_3$ which are colored $c$, $a$, $d$, $b$, and $e$ respectively.
\end{itemize}
\end{itemize}
\end{proof}

\begin{lemma} \label{Gneq}
The graph $G_{\neq}(u_1,u_2,v_1,v_2)$ in \Cref{fig:Gneq} has the following properties:
\begin{itemize}
\item $G_{\neq}(u_1,u_2,v_1,v_2)$ is planar and subcubic.
\item $G_{\neq}(u_1,u_2,v_1,v_2)$ has girth 6.
\item The distance in $G_{\neq}(u_1,u_2,v_1,v_2)$ between $u_1$ and $v_1$ is 5.
\item Every $5$-coloring $\phi$ of $G_{\neq}(u_1,u_2,v_1,v_2)$ satisfies $\phi(u_1)\neq \phi(v_1)$ and $\phi(u_2)=\phi(v_2)$.
\end{itemize}
\end{lemma}

\begin{figure}[H]
\centering
\subfloat[\label{fig:Gneq}The gadget $G_{\neq}(u_1,u_2,v_1,v_2)$ in \Cref{Gneq}.]{
\begin{tikzpicture}[scale=0.7]{thick}
\begin{scope}[every node/.style={circle,draw,inner sep=1}]
	\node (u1) at (0,4) {$u_1$};
	\node (u2) at (1,4) {$u_2$};
	\node (u3) at (3,8) {$u_3$};
	\node (u4) at (3,0) {$u_4$};
	\node (w1) at (3,6) {$w_1$};
	\node (w2) at (4,6) {$w_2$};
	\node (x1) at (8,6) {$x_1$};
	\node (x2) at (7,6) {$x_2$};
	\node (y1) at (3,2) {$y_1$};
	\node (y2) at (4,2) {$y_2$};
	\node (z1) at (8,2) {$z_1$};
	\node (z2) at (7,2) {$z_2$};
	\node (s1) at (3,4) {$s_1$};
	\node (t1) at (8,4) {$t_1$};
	\node (v1) at (11,4) {$v_1$};
	\node (v2) at (10,4) {$v_2$};
	\node (v3) at (8,8) {$v_3$};
	\node (v4) at (8,0) {$v_4$};
	\node (g'1) at (5.5,6) {\Large $G'$};
	\node (g'2) at (5.5,2) {\Large $G'$};
\end{scope}

\begin{scope}[every edge/.style={draw=black}]
    \path (u1) edge (u2);
    \path (u2) edge (u3);
    \path (u2) edge (u4);
    \path (v4) edge (u4);
    \path (u3) edge (v3);
    \path (v2) edge (v4);
    \path (v2) edge (v3);
    \path (v2) edge (v1);
    \path (u3) edge (w1);
    \path (w1) edge (w2);
    \path (w1) edge (s1);
    \path (s1) edge (y1);
    \path (y1) edge (u4);
    \path (y1) edge (y2);
    \path (v3) edge (x1);
    \path (x1) edge (x2);
    \path (x1) edge (t1);
    \path (t1) edge (z1);
    \path (z1) edge (v4);
    \path (z1) edge (z2);
    \path (s1) edge (t1);
    \path (w2) edge[bend left] (g'1);
    \path (w2) edge[bend right] (g'1);
    \path (x2) edge[bend left] (g'1);
    \path (x2) edge[bend right] (g'1);
    \path (y2) edge[bend left] (g'2);
    \path (y2) edge[bend right] (g'2);
    \path (z2) edge[bend left] (g'2);
    \path (z2) edge[bend right] (g'2);
\end{scope}
\end{tikzpicture}
}
\subfloat[Simplified drawing of $G_{\neq}(u_1,u_2,v_1,v_2)$.]{
\begin{tikzpicture}[scale=0.8]{thick}
\begin{scope}[every node/.style={circle,draw,inner sep=1}]
	\node (u1) at (0,0) {$u_1$};
	\node (u2) at (1,0) {$u_2$};
	\node (g') at (2.5,0) {\Large $G_{\neq}$};
	\node (v2) at (4,0) {$v_2$};
	\node (v1) at (5,0) {$v_1$};
\end{scope}

\begin{scope}[every edge/.style={draw=black}]
    \path (u1) edge (u2);
    \path (v2) edge (v1);
    \path (u2) edge[bend left] (g');
    \path (u2) edge[bend right] (g');
    \path (v2) edge[bend left] (g');
    \path (v2) edge[bend right] (g');
\end{scope}
\end{tikzpicture}
}
\caption{}
\end{figure}

\begin{proof}
One can verify that $G_{\neq}(u_1,u_2,v_1,v_2)$ is planar, subcubic, has girth 6, and that the distance between $u_1$ and $v_1$ is 5 thanks to \Cref{fig:G'}. Now, let $\phi$ be a $2$-distance $5$-coloring of $G_{\neq}(u_1,u_2,v_1,v_2)$.

First, observe the following:
\begin{claim} \label{claim}
We have $\{\phi(w_1),\phi(s_1),\phi(t_1),\phi(u_3),\phi(v_3)\}\neq\{a,b,c,d,e\}$ and $\{\phi(y_1),\phi(s_1),\phi(t_1),\phi(u_4),\phi(v_4)\}\neq\{a,b,c,d,e\}$. 
\end{claim}

\begin{proof}
By symmetry, we can suppose by contradiction that $\{\phi(w_1),\phi(s_1),\phi(t_1),\phi(u_3),\phi(v_3)\}=\{a,b,c,d,e\}$. Since $\phi(x_1)\notin\{\phi(s_1),\phi(t_1),\phi(u_3),\phi(v_3)\}$, we get $\phi(w_1)=\phi(x_1)$, in which case $\phi(w_2)=\phi(x_2)$ by \Cref{G'} due to $G'(w_1,w_2,x_1,x_2)$. However, $\phi(w_2)\notin\{\phi(w_1),\phi(s_1),\phi(u_3)\}$ and $\phi(x_2)\notin\{\phi(x_1),\phi(t_1),\phi(v_3)\}$, which is impossible since $\{\phi(w_1),\phi(s_1),\phi(t_1),\phi(u_3),\phi(v_3)\}=\{a,b,c,d,e\}$.
\end{proof}

We can assume w.l.o.g. that $\phi(u_1)=a$, $\phi(u_2)=b$, $\phi(u_3)=c$, and $\phi(u_4)=d$. We claim the following.

\begin{claim} \label{claim2}
We must have $\{\phi(v_3),\phi(v_4)\}\neq\{c,d\}$.
\end{claim}

\begin{proof}
If $\{\phi(v_3),\phi(v_4)\}=\{c,d\}$, then $\phi(v_3)=d$ and $\phi(v_4)=c$. Observe that $\phi(w_1)$, $\phi(s_1)$, and $\phi(t_1)$ must all be distinct and they are also different from $\{c,d\}$. As a result, we get $\{\phi(w_1),\phi(s_1),\phi(t_1),\phi(u_3),\phi(v_3)\}=\{a,b,c,d,e\}$, which is impossible by \Cref{claim}.
\end{proof}

\begin{claim}\label{claim3}
If $\phi(v_4)=c$, then $d\notin \{\phi(w_1),\phi(x_1),\phi(x_2)\}$. Symmetrically, if $\phi(v_3)=d$, then $c\notin\{\phi(y_1),\phi(z_1),\phi(z_2)\}$.
\end{claim}

\begin{proof} 
If $\phi(v_4)=c$, then suppose by contradiction that $d \in \{\phi(w_1),\phi(x_1),\phi(x_2)\}$. Observe that $\phi(y_1)$, $\phi(s_1)$, and $\phi(t_1)$ must all be distinct and they are also different from $\{c,d\}$. As a result, we get $\{\phi(y_1),\phi(s_1),\phi(t_1),\phi(u_4),\phi(v_4)\}=\{a,b,c,d,e\}$, which is impossible by \Cref{claim}.

By symmetry, the same arguments hold for $c\notin\{\phi(y_1),\phi(z_1),\phi(z_2)\}$ when $\phi(v_3)=d$.
\end{proof}

Now, suppose by contradiction that we have the following cases.

\textbf{Case 1:} $\phi(u_1)=\phi(v_1)$.\\
In this case, $\phi(v_1)=\phi(u_1)=a$. Note that $\phi(v_2)\notin \{\phi(v_1),\phi(u_3),\phi(u_4)\}=\{a,c,d\}$. Moreover, if $\phi(v_2)=e$, then we necessarily have $\phi(v_3)=d$ and $\phi(v_4)=c$ which is impossible due to \Cref{claim2}. As a result, $\phi(v_2)=b$. 

By \Cref{claim2} and by symmetry, we can assume that $\phi(v_3)=e$ and as a consequence, $\phi(v_4)=c$. By \Cref{claim3}, $d\notin \{\phi(w_1),\phi(x_1),\phi(x_2)\}$. Consequently, $\phi(w_1)=a$ and $\phi(x_1)=a$, which in turn implies that $\phi(w_2)=\phi(x_2)=b$ by \Cref{G'} and $G'(w_1,w_2,x_1,x_2)$. Hence, $\phi(s_1)=e$ and we get a contradiction since $\phi(y_1)\notin\{\phi(w_1),\phi(s_1),\phi(u_2),\phi(u_4),\phi(v_4)\}=\{a,e,b,d,c\}$.

\textbf{Case 2:} $\phi(u_2)\neq \phi(v_2)$.\\
Since $\phi(v_2)\notin\{\phi(u_2),\phi(u_3),\phi(u_4)\}=\{b,c,d\}$, we have $\phi(v_2)\in\{a,e\}$. By \Cref{claim2} and by symmetry, we can assume that $\phi(v_3)\in \{a,e\}$. As a consequence, $\{\phi(v_2),\phi(v_3)\}=\{a,e\}$ and $\phi(v_4)=c$. By \Cref{claim3}, $d\notin \{\phi(w_1),\phi(x_1),\phi(x_2)\}$. Consequently, $\phi(w_1)=\phi(v_2)$ and $\phi(x_1)=b$. Hence, $\phi(s_1)=\phi(v_3)$ and we get a contradiction since $\phi(y_1)\notin\{\phi(w_1),\phi(s_1),\phi(u_2),\phi(u_4),\phi(v_4)\}=\{a,e,b,d,c\}$.
\end{proof}

\begin{lemma} \label{lemma:non-5-colorable}
The graph in \Cref{fig:non-5-colorable} is a planar subcubic graph of girth 6 with 2-distance chromatic number at least 6.
\end{lemma} 

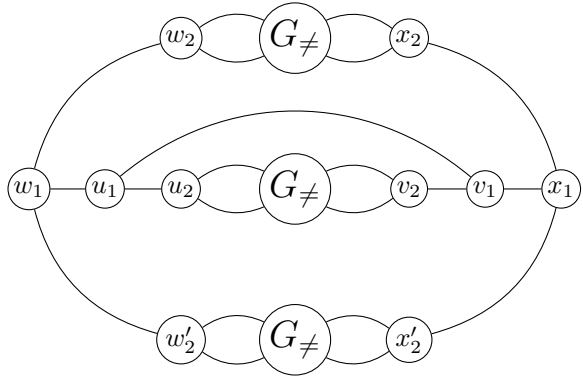
\begin{figure}[H]
\centering
\begin{tikzpicture}{thick}
\begin{scope}[every node/.style={circle,draw,inner sep=1}]
	\node (w1) at (-1,0) {$w_1$};	
	\node (u1) at (0,0) {$u_1$};
	\node (u2) at (1,0) {$u_2$};
	\node (g'1) at (2.5,0) {\Large $G_{\neq}$};
	\node (v2) at (4,0) {$v_2$};
	\node (v1) at (5,0) {$v_1$};
	\node (x1) at (6,0) {$x_1$};
	\node (w2) at (1,2) {$w_2$};
	\node (g'2) at (2.5,2) {\Large $G_{\neq}$};
	\node (x2) at (4,2) {$x_2$};
	\node (w'2) at (1,-2) {$w'_2$};
	\node (g'3) at (2.5,-2) {\Large $G_{\neq}$};
	\node (x'2) at (4,-2) {$x'_2$};
\end{scope}

\begin{scope}[every edge/.style={draw=black}]
	\path (w1) edge (u1);
	\path (x1) edge (v1);    
    \path (u1) edge (u2);
    \path (v2) edge (v1);
    \path (u2) edge[bend left] (g'1);
    \path (u2) edge[bend right] (g'1);
    \path (v2) edge[bend left] (g'1);
    \path (v2) edge[bend right] (g'1);
    \path (u1) edge[bend left=40] (v1);
    \path (w2) edge[bend left] (g'2);
    \path (w2) edge[bend right] (g'2);
    \path (x2) edge[bend left] (g'2);
    \path (x2) edge[bend right] (g'2);
    \path (w'2) edge[bend left] (g'3);
    \path (w'2) edge[bend right] (g'3);
    \path (x'2) edge[bend left] (g'3);
    \path (x'2) edge[bend right] (g'3);
    \path (w1) edge[bend left] (w2);
    \path (x2) edge[bend left] (x1);
    \path (w1) edge[bend right] (w'2);
    \path (x'2) edge[bend right] (x1);
\end{scope}
\end{tikzpicture}
\caption{\label{fig:non-5-colorable}A non-5-colorable planar subcubic graph of girth 6.}
\end{figure}

\begin{proof}
One can easily verify that the graph $G$ in \Cref{fig:non-5-colorable} is planar, subcubic, and has girth 6. Suppose by contradiction that there exists a $2$-distance $5$-coloring $\phi$ of $G$. Suppose w.l.o.g. that $\phi(w_1)=a$, $\phi(u_1)=b$, $\phi(u_2)=c$, and $\phi(v_1)=d$. By \Cref{Gneq}, $\phi(v_2)= \phi(u_2)=c$ due to $G_{\neq}(u_1,u_2,v_1,v_2)$ and $\phi(x_1)\neq\phi(w_1)$ due to $G_{\neq}(w_1,w_2,x_1,x_2)$. Moreover, since $\phi(x_1)\notin\{\phi(v_1),\phi(v_2),\phi(u_1)\}=\{d,c,b\}$. We must have $\phi(x_1)=e$. By \Cref{Gneq}, we also have $\phi(w_2)=\phi(x_2)$ due to $G_{\neq}(w_1,w_2,x_1,x_2)$. Since $\phi(w_2)\notin\{\phi(w_1),\phi(u_1)\}=\{a,b\}$ and $\phi(x_2)\notin\{\phi(v_1),\phi(x_1)\}=\{d,e\}$, we get $\phi(w_2)=\phi(x_2)=c$. By symmetry, we also get $\phi(w'_2)=\phi(x'_2)=c$, which is impossible since $w_2$ sees $w'_2$.
\end{proof}

\section{Generalization of the vertices-to-faces discharging verification algorithm\label{sec5}}

In \Cref{subsec:second_round_discharging}, we presented an algorithm (\Cref{algorithm}) that automates the discharging procedure with a given set of reducible configurations. This becomes extremely helpful for proofs where the discharging procedure involve a large case analysis. For the input we efficiently encode a face, the set of reducible configurations, as well as the amount of charge of a vertex depending on its neighborhood.  The corresponding computer program was written in Python. The source code and its documentation is publically available on \url{https://gite.lirmm.fr/discharging/planar-graphs}. In the case of \Cref{thm:girth8}, the execution time takes few seconds on a standard machine. 

On the public repository, we also provide another example where we proved the 2-distance 8-choosability of planar graphs with maximum degree 4 and girth at least 7, a result by Cranston \textit{et al.} in \cite{cra13}, using a small amount configurations that can be easily reduced (by hand or by computer) and very naive discharging rules. The idea is to move towards a computer automation of proofs using the discharging method.

Our approach can be applied to other problems on planar graphs by concentrating charges on the vertices of the graph when the distribution of charges is made (according to the Euler formula). First, one has to obtain a non-negative sum of charges on the vertices (by realizing an easy discharging procedure for example). This concentrates the difficulty of the problem on the second round of discharging. In this round, one has to redistribute the remaining charge of the vertices to the faces with negative charge and that is where our algorithm can come in handy. Note that the way our algorithm is designed, a vertex can also take charge from a face by giving it a negative charge.

The encoding of a face with a number-word and a letter-word can be done in the same way. In our case, since $G$ has no $2^+$-paths, the number-word of a face is composed of integers in $\{0,1\}$. But this alphabet can be extended to $\{0,1,\dots,k-1\}$ if $G$ has no $k^+$-paths. Observe that one can partition a face into $i$-paths ($0\leq i\leq k$) and consider that each path contains only one endvertex. Therefore, in order to obtain the starting number-words for a face of size $d(f)$, it suffices to decompose $d(f)$ into sums where each term corresponds to the number of vertices in an $i$-path.

As for the letter-words, it suffices to choose a letter for each different neighborhood of interest outside the considered face. In our case, three letters are sufficient but one can always work with a larger alphabet to suit the considered problems. Once the convention for the encoding of a face is fixed, the reducible configurations and entries of the dictionary of charges can be done in the same way.

There are a few details to note about the entries of the dictionary. First, the position of the vertex $v$ holding the charge must be in the center of the entry (or just left of the center). Second, the encoding has to  start and end with a number. These properties can be guaranteed by extending the encoding with every possible sequence up to a certain length. Finally, one has to be mindful that $v$ is in the center when the length of the encoding is congruent to 3 modulo 4, and left of the center when it is congruent to 1 modulo 4.

Moreover, we would like to note that, when a discharging procedure along with the given reducible configurations does not prove the desired result, \Cref{algorithm} returns a sufficient set of missing configurations (to be reduced). This helps to pinpoint the possible difficulty of the proof using discharging. In practice, we started out with a simple discharging procedure. Then, we proceed by reducing the missing configurations returned by \Cref{algorithm}. When there are non-reducible configurations, we further refine our discharging procedure and repeat the process until we reach a sufficient set of discharging rules and reducible configurations. This is how we obtained the configurations in \Cref{lemma:first_reducible_cycles,lemma:reducible,lemma:reducible_cycles,lemma:special_reducible}. In that sense, \Cref{algorithm} is not only a tool to verify a proof but also a tool to assist the research process. 

We also wanted to prove that subcubic planar graphs with girth at least 11 are $2$-distance 5-colorable (which would have improved the non-list version of the result in~\cite{bi12} by Borodin and Ivanova). The computer program returned the problematic configurations which made us realize the difficulty of finding the right discharging rules and reducible configurations.

\section*{Acknowledgements}

We would like to thank Mickael Montassier and Alexandre Pinlou for the helpful discussions on earlier versions of this paper. The second author was supported by the French ANR project DISTANCIA: ANR-17-CE40-0015.

\bibliographystyle{plain}
\bibliography{References.bib}

\begin{appendices}
\section{Reducible cycles}

\begin{figure}[H]
\centering
\hspace*{-1cm}
\subfloat[\texttt{1a1a0c0a0b0c}]{
\label{subfig:1a1a0c0a0b0c}
\scalebox{0.65}{
	\begin{tikzpicture}[join=bevel,vertex/.style={circle,draw, minimum size=0.6cm},inner sep=0mm,scale=0.8]
\foreach \i in {0,...,7}  \node[vertex] (\i) at (-45.0*\i:3) {$v_\i$};
\foreach \i in {0,...,7}  \draw let \n1={int(mod(\i+1,8))} in (\i) -- (\n1);
\node at (-45.0*0:2.1) {\small $5$};
\node at (-45.0*2:2.1) {\small $4$};
\node (u1) at (-45.0*1:4) {};
\draw (1) -- (u1);
\node at (-45.0*1:2.1) {\small $3$};
\node (u3) at (-45.0*3:4) {};
\draw (3) -- (u3);
\node at (-45.0*3:2.1) {\small $3$};
\node[vertex] (v'4) at (-45.0*4:4.5) {$v'_4$};
\node (u'4) at (-45.0*4:5.5) {};
\node at (-45.0*4-11:4.5) {\small $3$};
\draw (4) -- (v'4);
\draw (v'4) -- (u'4);
\node at (-45.0*4:2.1) {\small $3$};
\node (u5) at (-45.0*5:4) {};
\draw (5) -- (u5);
\node at (-45.0*5:2.1) {\small $3$};
\node[vertex] (v'6) at (-45.0*6:4.5) {$v'_6$};
\node[vertex] (v''6) at (-45.0*6-7:6) {$v''_6$};
\node[vertex] (v'''6) at (-45.0*6+7:6) {$v'''_6$};
\node (u''6) at (-45.0*6-7:7) {};
\node (u'''6) at (-45.0*6+7:7) {};
\node at (-45.0*6-11:4.5) {\small $4$};
\node at (-45.0*6-14:6) {\small $3$};
\node at (-45.0*6+14:6) {\small $3$};
\draw (6) -- (v'6);
\draw (v'6) -- (v''6);
\draw (v'6) -- (v'''6);
\draw (v''6) -- (u''6);
\draw (v'''6) -- (u'''6);
\node at (-45.0*6:2.1) {\small $5$};
\node[vertex] (v'7) at (-45.0*7:4.5) {$v'_7$};
\node (u'7) at (-45.0*7:5.5) {};
\node at (-45.0*7-11:4.5) {\small $3$};
\draw (7) -- (v'7);
\draw (v'7) -- (u'7);
\node at (-45.0*7:2.1) {\small $5$};
\end{tikzpicture}
}
}
\subfloat[\texttt{1c0a0c0b0c0b0c}]{
\label{subfig:1c0a0c0b0c0b0c}
\scalebox{0.65}{
	\begin{tikzpicture}[join=bevel,vertex/.style={circle,draw, minimum size=0.6cm},inner sep=0mm,scale=0.8]
\foreach \i in {0,...,7}  \node[vertex] (\i) at (-45.0*\i:3) {$v_\i$};
\foreach \i in {0,...,7}  \draw let \n1={int(mod(\i+1,8))} in (\i) -- (\n1);
\node at (-45.0*0:2.1) {\small $6$};
\node[vertex] (v'1) at (-45.0*1:4.5) {$v'_1$};
\node (u'1) at (-45.0*1:5.5) {};
\node at (-45.0*1-11:4.5) {\small $3$};
\draw (1) -- (v'1);
\draw (v'1) -- (u'1);
\node at (-45.0*1:2.1) {\small $4$};
\node (u2) at (-45.0*2:4) {};
\draw (2) -- (u2);
\node at (-45.0*2:2.1) {\small $3$};
\node[vertex] (v'3) at (-45.0*3:4.5) {$v'_3$};
\node (u'3) at (-45.0*3:5.5) {};
\node at (-45.0*3-11:4.5) {\small $3$};
\draw (3) -- (v'3);
\draw (v'3) -- (u'3);
\node at (-45.0*3:2.1) {\small $4$};
\node[vertex] (v'4) at (-45.0*4:4.5) {$v'_4$};
\node[vertex] (v''4) at (-45.0*4-7:6) {$v''_4$};
\node[vertex] (v'''4) at (-45.0*4+7:6) {$v'''_4$};
\node (u''4) at (-45.0*4-7:7) {};
\node (u'''4) at (-45.0*4+7:7) {};
\node at (-45.0*4-11:4.5) {\small $4$};
\node at (-45.0*4-14:6) {\small $3$};
\node at (-45.0*4+14:6) {\small $3$};
\draw (4) -- (v'4);
\draw (v'4) -- (v''4);
\draw (v'4) -- (v'''4);
\draw (v''4) -- (u''4);
\draw (v'''4) -- (u'''4);
\node at (-45.0*4:2.1) {\small $6$};
\node[vertex] (v'5) at (-45.0*5:4.5) {$v'_5$};
\node (u'5) at (-45.0*5:5.5) {};
\node at (-45.0*5-11:4.5) {\small $3$};
\draw (5) -- (v'5);
\draw (v'5) -- (u'5);
\node at (-45.0*5:2.1) {\small $5$};
\node[vertex] (v'6) at (-45.0*6:4.5) {$v'_6$};
\node[vertex] (v''6) at (-45.0*6-7:6) {$v''_6$};
\node[vertex] (v'''6) at (-45.0*6+7:6) {$v'''_6$};
\node (u''6) at (-45.0*6-7:7) {};
\node (u'''6) at (-45.0*6+7:7) {};
\node at (-45.0*6-11:4.5) {\small $4$};
\node at (-45.0*6-14:6) {\small $3$};
\node at (-45.0*6+14:6) {\small $3$};
\draw (6) -- (v'6);
\draw (v'6) -- (v''6);
\draw (v'6) -- (v'''6);
\draw (v''6) -- (u''6);
\draw (v'''6) -- (u'''6);
\node at (-45.0*6:2.1) {\small $6$};
\node[vertex] (v'7) at (-45.0*7:4.5) {$v'_7$};
\node (u'7) at (-45.0*7:5.5) {};
\node at (-45.0*7-11:4.5) {\small $3$};
\draw (7) -- (v'7);
\draw (v'7) -- (u'7);
\node at (-45.0*7:2.1) {\small $5$};
\end{tikzpicture}
}
}
\subfloat[\texttt{1c0a0c0c0b0b0c}]{
\label{subfig:1c0a0c0c0b0b0c}
\scalebox{0.65}{
	\begin{tikzpicture}[join=bevel,vertex/.style={circle,draw, minimum size=0.6cm},inner sep=0mm,scale=0.8]
\foreach \i in {0,...,7}  \node[vertex] (\i) at (-45.0*\i:3) {$v_\i$};
\foreach \i in {0,...,7}  \draw let \n1={int(mod(\i+1,8))} in (\i) -- (\n1);
\node at (-45.0*0:2.1) {\small $6$};
\node[vertex] (v'1) at (-45.0*1:4.5) {$v'_1$};
\node (u'1) at (-45.0*1:5.5) {};
\node at (-45.0*1-11:4.5) {\small $3$};
\draw (1) -- (v'1);
\draw (v'1) -- (u'1);
\node at (-45.0*1:2.1) {\small $4$};
\node (u2) at (-45.0*2:4) {};
\draw (2) -- (u2);
\node at (-45.0*2:2.1) {\small $3$};
\node[vertex] (v'3) at (-45.0*3:4.5) {$v'_3$};
\node (u'3) at (-45.0*3:5.5) {};
\node at (-45.0*3-11:4.5) {\small $3$};
\draw (3) -- (v'3);
\draw (v'3) -- (u'3);
\node at (-45.0*3:2.1) {\small $4$};
\node[vertex] (v'4) at (-45.0*4:4.5) {$v'_4$};
\node (u'4) at (-45.0*4:5.5) {};
\node at (-45.0*4-11:4.5) {\small $3$};
\draw (4) -- (v'4);
\draw (v'4) -- (u'4);
\node at (-45.0*4:2.1) {\small $5$};
\node[vertex] (v'5) at (-45.0*5:4.5) {$v'_5$};
\node[vertex] (v''5) at (-45.0*5-7:6) {$v''_5$};
\node[vertex] (v'''5) at (-45.0*5+7:6) {$v'''_5$};
\node (u''5) at (-45.0*5-7:7) {};
\node (u'''5) at (-45.0*5+7:7) {};
\node at (-45.0*5-11:4.5) {\small $4$};
\node at (-45.0*5-14:6) {\small $3$};
\node at (-45.0*5+14:6) {\small $3$};
\draw (5) -- (v'5);
\draw (v'5) -- (v''5);
\draw (v'5) -- (v'''5);
\draw (v''5) -- (u''5);
\draw (v'''5) -- (u'''5);
\node at (-45.0*5:2.1) {\small $6$};
\node[vertex] (v'6) at (-45.0*6:4.5) {$v'_6$};
\node[vertex] (v''6) at (-45.0*6-7:6) {$v''_6$};
\node[vertex] (v'''6) at (-45.0*6+7:6) {$v'''_6$};
\node (u''6) at (-45.0*6-7:7) {};
\node (u'''6) at (-45.0*6+7:7) {};
\node at (-45.0*6-11:4.5) {\small $4$};
\node at (-45.0*6-14:6) {\small $3$};
\node at (-45.0*6+14:6) {\small $3$};
\draw (6) -- (v'6);
\draw (v'6) -- (v''6);
\draw (v'6) -- (v'''6);
\draw (v''6) -- (u''6);
\draw (v'''6) -- (u'''6);
\node at (-45.0*6:2.1) {\small $6$};
\node[vertex] (v'7) at (-45.0*7:4.5) {$v'_7$};
\node (u'7) at (-45.0*7:5.5) {};
\node at (-45.0*7-11:4.5) {\small $3$};
\draw (7) -- (v'7);
\draw (v'7) -- (u'7);
\node at (-45.0*7:2.1) {\small $5$};
\end{tikzpicture}
}
}

\hspace*{0.4cm}
\subfloat[\texttt{1c0b0c0a0c0b0c}]{
\label{subfig:1c0b0c0a0c0b0c}
\scalebox{0.65}{
	\begin{tikzpicture}[join=bevel,vertex/.style={circle,draw, minimum size=0.6cm},inner sep=0mm,scale=0.8]
\foreach \i in {0,...,7}  \node[vertex] (\i) at (-45.0*\i:3) {$v_\i$};
\foreach \i in {0,...,7}  \draw let \n1={int(mod(\i+1,8))} in (\i) -- (\n1);
\node at (-45.0*0:2.1) {\small $6$};
\node[vertex] (v'1) at (-45.0*1:4.5) {$v'_1$};
\node (u'1) at (-45.0*1:5.5) {};
\node at (-45.0*1-11:4.5) {\small $3$};
\draw (1) -- (v'1);
\draw (v'1) -- (u'1);
\node at (-45.0*1:2.1) {\small $5$};
\node[vertex] (v'2) at (-45.0*2:4.5) {$v'_2$};
\node[vertex] (v''2) at (-45.0*2-7:6) {$v''_2$};
\node[vertex] (v'''2) at (-45.0*2+7:6) {$v'''_2$};
\node (u''2) at (-45.0*2-7:7) {};
\node (u'''2) at (-45.0*2+7:7) {};
\node at (-45.0*2-11:4.5) {\small $4$};
\node at (-45.0*2-14:6) {\small $3$};
\node at (-45.0*2+14:6) {\small $3$};
\draw (2) -- (v'2);
\draw (v'2) -- (v''2);
\draw (v'2) -- (v'''2);
\draw (v''2) -- (u''2);
\draw (v'''2) -- (u'''2);
\node at (-45.0*2:2.1) {\small $6$};
\node[vertex] (v'3) at (-45.0*3:4.5) {$v'_3$};
\node (u'3) at (-45.0*3:5.5) {};
\node at (-45.0*3-11:4.5) {\small $3$};
\draw (3) -- (v'3);
\draw (v'3) -- (u'3);
\node at (-45.0*3:2.1) {\small $4$};
\node (u4) at (-45.0*4:4) {};
\draw (4) -- (u4);
\node at (-45.0*4:2.1) {\small $3$};
\node[vertex] (v'5) at (-45.0*5:4.5) {$v'_5$};
\node (u'5) at (-45.0*5:5.5) {};
\node at (-45.0*5-11:4.5) {\small $3$};
\draw (5) -- (v'5);
\draw (v'5) -- (u'5);
\node at (-45.0*5:2.1) {\small $4$};
\node[vertex] (v'6) at (-45.0*6:4.5) {$v'_6$};
\node[vertex] (v''6) at (-45.0*6-7:6) {$v''_6$};
\node[vertex] (v'''6) at (-45.0*6+7:6) {$v'''_6$};
\node (u''6) at (-45.0*6-7:7) {};
\node (u'''6) at (-45.0*6+7:7) {};
\node at (-45.0*6-11:4.5) {\small $4$};
\node at (-45.0*6-14:6) {\small $3$};
\node at (-45.0*6+14:6) {\small $3$};
\draw (6) -- (v'6);
\draw (v'6) -- (v''6);
\draw (v'6) -- (v'''6);
\draw (v''6) -- (u''6);
\draw (v'''6) -- (u'''6);
\node at (-45.0*6:2.1) {\small $6$};
\node[vertex] (v'7) at (-45.0*7:4.5) {$v'_7$};
\node (u'7) at (-45.0*7:5.5) {};
\node at (-45.0*7-11:4.5) {\small $3$};
\draw (7) -- (v'7);
\draw (v'7) -- (u'7);
\node at (-45.0*7:2.1) {\small $5$};
\end{tikzpicture}
}
}
\subfloat[\texttt{1a0a0c1c0b0a}]{
\label{subfig:1a0a0c1c0b0a}
\scalebox{0.65}{
	\begin{tikzpicture}[join=bevel,vertex/.style={circle,draw, minimum size=0.6cm},inner sep=0mm,scale=0.8]
\foreach \i in {0,...,7}  \node[vertex] (\i) at (-45.0*\i:3) {$v_\i$};
\foreach \i in {0,...,7}  \draw let \n1={int(mod(\i+1,8))} in (\i) -- (\n1);
\node at (-45.0*0:2.1) {\small $4$};
\node at (-45.0*4:2.1) {\small $6$};
\node (u1) at (-45.0*1:4) {};
\draw (1) -- (u1);
\node at (-45.0*1:2.1) {\small $2$};
\node (u2) at (-45.0*2:4) {};
\draw (2) -- (u2);
\node at (-45.0*2:2.1) {\small $2$};
\node[vertex] (v'3) at (-45.0*3:4.5) {$v'_3$};
\node (u'3) at (-45.0*3:5.5) {};
\node at (-45.0*3-11:4.5) {\small $3$};
\draw (3) -- (v'3);
\draw (v'3) -- (u'3);
\node at (-45.0*3:2.1) {\small $4$};
\node[vertex] (v'5) at (-45.0*5:4.5) {$v'_5$};
\node (u'5) at (-45.0*5:5.5) {};
\node at (-45.0*5-11:4.5) {\small $3$};
\draw (5) -- (v'5);
\draw (v'5) -- (u'5);
\node at (-45.0*5:2.1) {\small $5$};
\node[vertex] (v'6) at (-45.0*6:4.5) {$v'_6$};
\node[vertex] (v''6) at (-45.0*6-7:6) {$v''_6$};
\node[vertex] (v'''6) at (-45.0*6+7:6) {$v'''_6$};
\node (u''6) at (-45.0*6-7:7) {};
\node (u'''6) at (-45.0*6+7:7) {};
\node at (-45.0*6-11:4.5) {\small $4$};
\node at (-45.0*6-14:6) {\small $3$};
\node at (-45.0*6+14:6) {\small $3$};
\draw (6) -- (v'6);
\draw (v'6) -- (v''6);
\draw (v'6) -- (v'''6);
\draw (v''6) -- (u''6);
\draw (v'''6) -- (u'''6);
\node at (-45.0*6:2.1) {\small $5$};
\node (u7) at (-45.0*7:4) {};
\draw (7) -- (u7);
\node at (-45.0*7:2.1) {\small $3$};
\end{tikzpicture}
}
}
\hspace*{0.4cm}
\subfloat[\texttt{1a0a1c0a0c0c}]{
\label{subfig:1a0a1c0a0c0c}
\scalebox{0.65}{
	\begin{tikzpicture}[join=bevel,vertex/.style={circle,draw, minimum size=0.6cm},inner sep=0mm,scale=0.8]
\foreach \i in {0,...,7}  \node[vertex] (\i) at (-45.0*\i:3) {$v_\i$};
\foreach \i in {0,...,7}  \draw let \n1={int(mod(\i+1,8))} in (\i) -- (\n1);
\node at (-45.0*0:2.1) {\small $5$};
\node at (-45.0*3:2.1) {\small $5$};
\node (u1) at (-45.0*1:4) {};
\draw (1) -- (u1);
\node at (-45.0*1:2.1) {\small $2$};
\node (u2) at (-45.0*2:4) {};
\draw (2) -- (u2);
\node at (-45.0*2:2.1) {\small $2$};
\node[vertex] (v'4) at (-45.0*4:4.5) {$v'_4$};
\node (u'4) at (-45.0*4:5.5) {};
\node at (-45.0*4-11:4.5) {\small $3$};
\draw (4) -- (v'4);
\draw (v'4) -- (u'4);
\node at (-45.0*4:2.1) {\small $4$};
\node (u5) at (-45.0*5:4) {};
\draw (5) -- (u5);
\node at (-45.0*5:2.1) {\small $3$};
\node[vertex] (v'6) at (-45.0*6:4.5) {$v'_6$};
\node (u'6) at (-45.0*6:5.5) {};
\node at (-45.0*6-11:4.5) {\small $3$};
\draw (6) -- (v'6);
\draw (v'6) -- (u'6);
\node at (-45.0*6:2.1) {\small $4$};
\node[vertex] (v'7) at (-45.0*7:4.5) {$v'_7$};
\node (u'7) at (-45.0*7:5.5) {};
\node at (-45.0*7-11:4.5) {\small $3$};
\draw (7) -- (v'7);
\draw (v'7) -- (u'7);
\node at (-45.0*7:2.1) {\small $5$};
\end{tikzpicture}
}
}

\subfloat[\texttt{1a0c0a1c0a0c}]{
\label{subfig:1a0c0a1c0a0c}
\scalebox{0.65}{
	\begin{tikzpicture}[join=bevel,vertex/.style={circle,draw, minimum size=0.6cm},inner sep=0mm,scale=0.8]
\foreach \i in {0,...,7}  \node[vertex] (\i) at (-45.0*\i:3) {$v_\i$};
\foreach \i in {0,...,7}  \draw let \n1={int(mod(\i+1,8))} in (\i) -- (\n1);
\node at (-45.0*0:2.1) {\small $5$};
\node at (-45.0*4:2.1) {\small $5$};
\node (u1) at (-45.0*1:4) {};
\draw (1) -- (u1);
\node at (-45.0*1:2.1) {\small $3$};
\node[vertex] (v'2) at (-45.0*2:4.5) {$v'_2$};
\node (u'2) at (-45.0*2:5.5) {};
\node at (-45.0*2-11:4.5) {\small $3$};
\draw (2) -- (v'2);
\draw (v'2) -- (u'2);
\node at (-45.0*2:2.1) {\small $3$};
\node (u3) at (-45.0*3:4) {};
\draw (3) -- (u3);
\node at (-45.0*3:2.1) {\small $3$};
\node[vertex] (v'5) at (-45.0*5:4.5) {$v'_5$};
\node (u'5) at (-45.0*5:5.5) {};
\node at (-45.0*5-11:4.5) {\small $3$};
\draw (5) -- (v'5);
\draw (v'5) -- (u'5);
\node at (-45.0*5:2.1) {\small $4$};
\node (u6) at (-45.0*6:4) {};
\draw (6) -- (u6);
\node at (-45.0*6:2.1) {\small $3$};
\node[vertex] (v'7) at (-45.0*7:4.5) {$v'_7$};
\node (u'7) at (-45.0*7:5.5) {};
\node at (-45.0*7-11:4.5) {\small $3$};
\draw (7) -- (v'7);
\draw (v'7) -- (u'7);
\node at (-45.0*7:2.1) {\small $4$};
\end{tikzpicture}
}
}
\subfloat[\texttt{1a0a0c0c0c0c0a}]{
\label{subfig:1a0a0c0c0c0c0a}
\scalebox{0.65}{
	\begin{tikzpicture}[join=bevel,vertex/.style={circle,draw, minimum size=0.6cm},inner sep=0mm,scale=0.8]
\foreach \i in {0,...,7}  \node[vertex] (\i) at (-45.0*\i:3) {$v_\i$};
\foreach \i in {0,...,7}  \draw let \n1={int(mod(\i+1,8))} in (\i) -- (\n1);
\node at (-45.0*0:2.1) {\small $4$};
\node (u1) at (-45.0*1:4) {};
\draw (1) -- (u1);
\node at (-45.0*1:2.1) {\small $2$};
\node at (-45.0*2:2.1) {\small $2$};
\node (u2) at (-45.0*2:4) {};
\draw (2) -- (u2);
\node[vertex] (v'3) at (-45.0*3:4.5) {$v'_3$};
\node (u'3) at (-45.0*3:5.5) {};
\node at (-45.0*3-11:4.5) {\small $3$};
\draw (3) -- (v'3);
\draw (v'3) -- (u'3);
\node at (-45.0*3:2.1) {\small $4$};
\node[vertex] (v'4) at (-45.0*4:4.5) {$v'_4$};
\node (u'4) at (-45.0*4:5.5) {};
\node at (-45.0*4-11:4.5) {\small $3$};
\draw (4) -- (v'4);
\draw (v'4) -- (u'4);
\node at (-45.0*4:2.1) {\small $5$};
\node[vertex] (v'5) at (-45.0*5:4.5) {$v'_5$};
\node (u'5) at (-45.0*5:5.5) {};
\node at (-45.0*5-11:4.5) {\small $3$};
\draw (5) -- (v'5);
\draw (v'5) -- (u'5);
\node at (-45.0*5:2.1) {\small $5$};
\node[vertex] (v'6) at (-45.0*6:4.5) {$v'_6$};
\node (u'6) at (-45.0*6:5.5) {};
\node at (-45.0*6-11:4.5) {\small $3$};
\draw (6) -- (v'6);
\draw (v'6) -- (u'6);
\node at (-45.0*6:2.1) {\small $4$};
\node (u7) at (-45.0*7:4) {};
\draw (7) -- (u7);
\node at (-45.0*7:2.1) {\small $3$};
\end{tikzpicture}
}
}
\hspace*{0.5cm}
\subfloat[\texttt{1a1a0b0c0a0c}]{
\label{subfig:1a1a0b0c0a0c}
\scalebox{0.65}{
	\begin{tikzpicture}[join=bevel,vertex/.style={circle,draw, minimum size=0.6cm},inner sep=0mm,scale=0.8]
\foreach \i in {0,...,7}  \node[vertex] (\i) at (-45.0*\i:3) {$v_\i$};
\foreach \i in {0,...,7}  \draw let \n1={int(mod(\i+1,8))} in (\i) -- (\n1);
\node at (-45.0*0:2.1) {\small $5$};
\node at (-45.0*2:2.1) {\small $4$};
\node (u1) at (-45.0*1:4) {};
\draw (1) -- (u1);
\node at (-45.0*1:2.1) {\small $3$};
\node (u3) at (-45.0*3:4) {};
\draw (3) -- (u3);
\node at (-45.0*3:2.1) {\small $3$};
\node[vertex] (v'4) at (-45.0*4:4.5) {$v'_4$};
\node[vertex] (v''4) at (-45.0*4-7:6) {$v''_4$};
\node[vertex] (v'''4) at (-45.0*4+7:6) {$v'''_4$};
\node (u''4) at (-45.0*4-7:7) {};
\node (u'''4) at (-45.0*4+7:7) {};
\node at (-45.0*4-11:4.5) {\small $4$};
\node at (-45.0*4-14:6) {\small $3$};
\node at (-45.0*4+14:6) {\small $3$};
\draw (4) -- (v'4);
\draw (v'4) -- (v''4);
\draw (v'4) -- (v'''4);
\draw (v''4) -- (u''4);
\draw (v'''4) -- (u'''4);
\node at (-45.0*4:2.1) {\small $5$};
\node[vertex] (v'5) at (-45.0*5:4.5) {$v'_5$};
\node (u'5) at (-45.0*5:5.5) {};
\node at (-45.0*5-11:4.5) {\small $3$};
\draw (5) -- (v'5);
\draw (v'5) -- (u'5);
\node at (-45.0*5:2.1) {\small $4$};
\node (u6) at (-45.0*6:4) {};
\draw (6) -- (u6);
\node at (-45.0*6:2.1) {\small $3$};
\node[vertex] (v'7) at (-45.0*7:4.5) {$v'_7$};
\node (u'7) at (-45.0*7:5.5) {};
\node at (-45.0*7-11:4.5) {\small $3$};
\draw (7) -- (v'7);
\draw (v'7) -- (u'7);
\node at (-45.0*7:2.1) {\small $4$};
\end{tikzpicture}
}
}

\phantomcaption
\end{figure}

\begin{figure}[H]
\ContinuedFloat

\hspace*{-1.3cm}
\subfloat[\texttt{1a1a0c0b0c0a}]{
\label{subfig:1a1a0c0b0c0a}
\scalebox{0.65}{
	\begin{tikzpicture}[join=bevel,vertex/.style={circle,draw, minimum size=0.6cm},inner sep=0mm,scale=0.8]
\foreach \i in {0,...,7}  \node[vertex] (\i) at (-45.0*\i:3) {$v_\i$};
\foreach \i in {0,...,7}  \draw let \n1={int(mod(\i+1,8))} in (\i) -- (\n1);
\node at (-45.0*0:2.1) {\small $4$};
\node at (-45.0*2:2.1) {\small $4$};
\node (u1) at (-45.0*1:4) {};
\draw (1) -- (u1);
\node at (-45.0*1:2.1) {\small $3$};
\node (u3) at (-45.0*3:4) {};
\draw (3) -- (u3);
\node at (-45.0*3:2.1) {\small $3$};
\node[vertex] (v'4) at (-45.0*4:4.5) {$v'_4$};
\node (u'4) at (-45.0*4:5.5) {};
\node at (-45.0*4-11:4.5) {\small $3$};
\draw (4) -- (v'4);
\draw (v'4) -- (u'4);
\node at (-45.0*4:2.1) {\small $4$};
\node[vertex] (v'5) at (-45.0*5:4.5) {$v'_5$};
\node[vertex] (v''5) at (-45.0*5-7:6) {$v''_5$};
\node[vertex] (v'''5) at (-45.0*5+7:6) {$v'''_5$};
\node (u''5) at (-45.0*5-7:7) {};
\node (u'''5) at (-45.0*5+7:7) {};
\node at (-45.0*5-11:4.5) {\small $4$};
\node at (-45.0*5-14:6) {\small $3$};
\node at (-45.0*5+14:6) {\small $3$};
\draw (5) -- (v'5);
\draw (v'5) -- (v''5);
\draw (v'5) -- (v'''5);
\draw (v''5) -- (u''5);
\draw (v'''5) -- (u'''5);
\node at (-45.0*5:2.1) {\small $6$};
\node[vertex] (v'6) at (-45.0*6:4.5) {$v'_6$};
\node (u'6) at (-45.0*6:5.5) {};
\node at (-45.0*6-11:4.5) {\small $3$};
\draw (6) -- (v'6);
\draw (v'6) -- (u'6);
\node at (-45.0*6:2.1) {\small $4$};
\node (u7) at (-45.0*7:4) {};
\draw (7) -- (u7);
\node at (-45.0*7:2.1) {\small $3$};
\end{tikzpicture}
}
}
\hspace*{0.5cm}
\subfloat[\texttt{1c0a0c0c0c0a0c}]{
\label{subfig:1c0a0c0c0c0a0c}
\scalebox{0.65}{
	\begin{tikzpicture}[join=bevel,vertex/.style={circle,draw, minimum size=0.6cm},inner sep=0mm,scale=0.8]
\foreach \i in {0,...,7}  \node[vertex] (\i) at (-45.0*\i:3) {$v_\i$};
\foreach \i in {0,...,7}  \draw let \n1={int(mod(\i+1,8))} in (\i) -- (\n1);
\node at (-45.0*0:2.1) {\small $6$};
\node[vertex] (v'1) at (-45.0*1:4.5) {$v'_1$};
\node (u'1) at (-45.0*1:5.5) {};
\node at (-45.0*1-11:4.5) {\small $3$};
\draw (1) -- (v'1);
\draw (v'1) -- (u'1);
\node at (-45.0*1:2.1) {\small $4$};
\node (u2) at (-45.0*2:4) {};
\draw (2) -- (u2);
\node at (-45.0*2:2.1) {\small $3$};
\node[vertex] (v'3) at (-45.0*3:4.5) {$v'_3$};
\node (u'3) at (-45.0*3:5.5) {};
\node at (-45.0*3-11:4.5) {\small $3$};
\draw (3) -- (v'3);
\draw (v'3) -- (u'3);
\node at (-45.0*3:2.1) {\small $4$};
\node[vertex] (v'4) at (-45.0*4:4.5) {$v'_4$};
\node (u'4) at (-45.0*4:5.5) {};
\node at (-45.0*4-11:4.5) {\small $3$};
\draw (4) -- (v'4);
\draw (v'4) -- (u'4);
\node at (-45.0*4:2.1) {\small $5$};
\node[vertex] (v'5) at (-45.0*5:4.5) {$v'_5$};
\node (u'5) at (-45.0*5:5.5) {};
\node at (-45.0*5-11:4.5) {\small $3$};
\draw (5) -- (v'5);
\draw (v'5) -- (u'5);
\node at (-45.0*5:2.1) {\small $4$};
\node (u6) at (-45.0*6:4) {};
\draw (6) -- (u6);
\node at (-45.0*6:2.1) {\small $3$};
\node[vertex] (v'7) at (-45.0*7:4.5) {$v'_7$};
\node (u'7) at (-45.0*7:5.5) {};
\node at (-45.0*7-11:4.5) {\small $3$};
\draw (7) -- (v'7);
\draw (v'7) -- (u'7);
\node at (-45.0*7:2.1) {\small $4$};
\end{tikzpicture}
}
}
\subfloat[\texttt{0a0a0c0c0c0c0c0c}]{
\label{subfig:0a0a0c0c0c0c0c0c}
\scalebox{0.65}{
	\begin{tikzpicture}[join=bevel,vertex/.style={circle,draw, minimum size=0.6cm},inner sep=0mm,scale=0.8]
\foreach \i in {0,...,7}  \node[vertex] (\i) at (-45.0*\i:3) {$v_\i$};
\foreach \i in {0,...,7}  \draw let \n1={int(mod(\i+1,8))} in (\i) -- (\n1);
\node (u0) at (-45.0*0:4) {};
\draw (0) -- (u0);
\node at (-45.0*0:2.1) {\small $2$};
\node (u1) at (-45.0*1:4) {};
\draw (1) -- (u1);
\node at (-45.0*1:2.1) {\small $2$};
\node[vertex] (v'2) at (-45.0*2:4.5) {$v'_2$};
\node (u'2) at (-45.0*2:5.5) {};
\node at (-45.0*2-11:4.5) {\small $3$};
\draw (2) -- (v'2);
\draw (v'2) -- (u'2);
\node at (-45.0*2:2.1) {\small $4$};
\node[vertex] (v'3) at (-45.0*3:4.5) {$v'_3$};
\node (u'3) at (-45.0*3:5.5) {};
\node at (-45.0*3-11:4.5) {\small $3$};
\draw (3) -- (v'3);
\draw (v'3) -- (u'3);
\node at (-45.0*3:2.1) {\small $5$};
\node[vertex] (v'4) at (-45.0*4:4.5) {$v'_4$};
\node (u'4) at (-45.0*4:5.5) {};
\node at (-45.0*4-11:4.5) {\small $3$};
\draw (4) -- (v'4);
\draw (v'4) -- (u'4);
\node at (-45.0*4:2.1) {\small $5$};
\node[vertex] (v'5) at (-45.0*5:4.5) {$v'_5$};
\node (u'5) at (-45.0*5:5.5) {};
\node at (-45.0*5-11:4.5) {\small $3$};
\draw (5) -- (v'5);
\draw (v'5) -- (u'5);
\node at (-45.0*5:2.1) {\small $5$};
\node[vertex] (v'6) at (-45.0*6:4.5) {$v'_6$};
\node (u'6) at (-45.0*6:5.5) {};
\node at (-45.0*6-11:4.5) {\small $3$};
\draw (6) -- (v'6);
\draw (v'6) -- (u'6);
\node at (-45.0*6:2.1) {\small $5$};
\node[vertex] (v'7) at (-45.0*7:4.5) {$v'_7$};
\node (u'7) at (-45.0*7:5.5) {};
\node at (-45.0*7-11:4.5) {\small $3$};
\draw (7) -- (v'7);
\draw (v'7) -- (u'7);
\node at (-45.0*7:2.1) {\small $4$};
\end{tikzpicture}
}
}

\subfloat[\texttt{1c1a0a0c0b0a}]{
\label{subfig:1c1a0a0c0b0a}
\scalebox{0.65}{
	\begin{tikzpicture}[join=bevel,vertex/.style={circle,draw, minimum size=0.6cm},inner sep=0mm,scale=0.8]
\foreach \i in {0,...,7}  \node[vertex] (\i) at (-45.0*\i:3) {$v_\i$};
\foreach \i in {0,...,7}  \draw let \n1={int(mod(\i+1,8))} in (\i) -- (\n1);
\node at (-45.0*0:2.1) {\small $5$};
\node at (-45.0*2:2.1) {\small $5$};
\node[vertex] (v'1) at (-45.0*1:4.5) {$v'_1$};
\node (u'1) at (-45.0*1:5.5) {};
\node at (-45.0*1-11:4.5) {\small $3$};
\draw (1) -- (v'1);
\draw (v'1) -- (u'1);
\node at (-45.0*1:2.1) {\small $5$};
\node (u3) at (-45.0*3:4) {};
\draw (3) -- (u3);
\node at (-45.0*3:2.1) {\small $2$};
\node (u4) at (-45.0*4:4) {};
\draw (4) -- (u4);
\node at (-45.0*4:2.1) {\small $2$};
\node[vertex] (v'5) at (-45.0*5:4.5) {$v'_5$};
\node (u'5) at (-45.0*5:5.5) {};
\node at (-45.0*5-11:4.5) {\small $3$};
\draw (5) -- (v'5);
\draw (v'5) -- (u'5);
\node at (-45.0*5:2.1) {\small $4$};
\node[vertex] (v'6) at (-45.0*6:4.5) {$v'_6$};
\node[vertex] (v''6) at (-45.0*6-7:6) {$v''_6$};
\node[vertex] (v'''6) at (-45.0*6+7:6) {$v'''_6$};
\node (u''6) at (-45.0*6-7:7) {};
\node (u'''6) at (-45.0*6+7:7) {};
\node at (-45.0*6-11:4.5) {\small $4$};
\node at (-45.0*6-14:6) {\small $3$};
\node at (-45.0*6+14:6) {\small $3$};
\draw (6) -- (v'6);
\draw (v'6) -- (v''6);
\draw (v'6) -- (v'''6);
\draw (v''6) -- (u''6);
\draw (v'''6) -- (u'''6);
\node at (-45.0*6:2.1) {\small $5$};
\node (u7) at (-45.0*7:4) {};
\draw (7) -- (u7);
\node at (-45.0*7:2.1) {\small $3$};
\end{tikzpicture}
}
}
\subfloat[\texttt{1c1a0b0a0c0a}]{
\label{subfig:1c1a0b0a0c0a}
\scalebox{0.65}{
	\begin{tikzpicture}[join=bevel,vertex/.style={circle,draw, minimum size=0.6cm},inner sep=0mm,scale=0.8]
\foreach \i in {0,...,7}  \node[vertex] (\i) at (-45.0*\i:3) {$v_\i$};
\foreach \i in {0,...,7}  \draw let \n1={int(mod(\i+1,8))} in (\i) -- (\n1);
\node at (-45.0*0:2.1) {\small $5$};
\node at (-45.0*2:2.1) {\small $5$};
\node[vertex] (v'1) at (-45.0*1:4.5) {$v'_1$};
\node (u'1) at (-45.0*1:5.5) {};
\node at (-45.0*1-11:4.5) {\small $3$};
\draw (1) -- (v'1);
\draw (v'1) -- (u'1);
\node at (-45.0*1:2.1) {\small $5$};
\node (u3) at (-45.0*3:4) {};
\draw (3) -- (u3);
\node at (-45.0*3:2.1) {\small $3$};
\node[vertex] (v'4) at (-45.0*4:4.5) {$v'_4$};
\node[vertex] (v''4) at (-45.0*4-7:6) {$v''_4$};
\node[vertex] (v'''4) at (-45.0*4+7:6) {$v'''_4$};
\node (u''4) at (-45.0*4-7:7) {};
\node (u'''4) at (-45.0*4+7:7) {};
\node at (-45.0*4-11:4.5) {\small $4$};
\node at (-45.0*4-14:6) {\small $3$};
\node at (-45.0*4+14:6) {\small $3$};
\draw (4) -- (v'4);
\draw (v'4) -- (v''4);
\draw (v'4) -- (v'''4);
\draw (v''4) -- (u''4);
\draw (v'''4) -- (u'''4);
\node at (-45.0*4:2.1) {\small $4$};
\node (u5) at (-45.0*5:4) {};
\draw (5) -- (u5);
\node at (-45.0*5:2.1) {\small $3$};
\node[vertex] (v'6) at (-45.0*6:4.5) {$v'_6$};
\node (u'6) at (-45.0*6:5.5) {};
\node at (-45.0*6-11:4.5) {\small $3$};
\draw (6) -- (v'6);
\draw (v'6) -- (u'6);
\node at (-45.0*6:2.1) {\small $3$};
\node (u7) at (-45.0*7:4) {};
\draw (7) -- (u7);
\node at (-45.0*7:2.1) {\small $3$};
\end{tikzpicture}
}
}
\subfloat[\texttt{1c1a0a0b0c0a}]{
\label{subfig:1c1a0a0b0c0a}
\scalebox{0.65}{
	\begin{tikzpicture}[join=bevel,vertex/.style={circle,draw, minimum size=0.6cm},inner sep=0mm,scale=0.8]
\foreach \i in {0,...,7}  \node[vertex] (\i) at (-45.0*\i:3) {$v_\i$};
\foreach \i in {0,...,7}  \draw let \n1={int(mod(\i+1,8))} in (\i) -- (\n1);
\node at (-45.0*0:2.1) {\small $5$};
\node at (-45.0*2:2.1) {\small $5$};
\node[vertex] (v'1) at (-45.0*1:4.5) {$v'_1$};
\node (u'1) at (-45.0*1:5.5) {};
\node at (-45.0*1-11:4.5) {\small $3$};
\draw (1) -- (v'1);
\draw (v'1) -- (u'1);
\node at (-45.0*1:2.1) {\small $5$};
\node (u3) at (-45.0*3:4) {};
\draw (3) -- (u3);
\node at (-45.0*3:2.1) {\small $2$};
\node (u4) at (-45.0*4:4) {};
\draw (4) -- (u4);
\node at (-45.0*4:2.1) {\small $2$};
\node[vertex] (v'5) at (-45.0*5:4.5) {$v'_5$};
\node[vertex] (v''5) at (-45.0*5-7:6) {$v''_5$};
\node[vertex] (v'''5) at (-45.0*5+7:6) {$v'''_5$};
\node (u''5) at (-45.0*5-7:7) {};
\node (u'''5) at (-45.0*5+7:7) {};
\node at (-45.0*5-11:4.5) {\small $4$};
\node at (-45.0*5-14:6) {\small $3$};
\node at (-45.0*5+14:6) {\small $3$};
\draw (5) -- (v'5);
\draw (v'5) -- (v''5);
\draw (v'5) -- (v'''5);
\draw (v''5) -- (u''5);
\draw (v'''5) -- (u'''5);
\node at (-45.0*5:2.1) {\small $5$};
\node[vertex] (v'6) at (-45.0*6:4.5) {$v'_6$};
\node (u'6) at (-45.0*6:5.5) {};
\node at (-45.0*6-11:4.5) {\small $3$};
\draw (6) -- (v'6);
\draw (v'6) -- (u'6);
\node at (-45.0*6:2.1) {\small $4$};
\node (u7) at (-45.0*7:4) {};
\draw (7) -- (u7);
\node at (-45.0*7:2.1) {\small $3$};
\end{tikzpicture}
}
}

\subfloat[\texttt{0a0a0c0b0c0c0c0c}]{
\label{subfig:0a0a0c0b0c0c0c0c}
\scalebox{0.65}{
	\begin{tikzpicture}[join=bevel,vertex/.style={circle,draw, minimum size=0.6cm},inner sep=0mm,scale=0.8]
\foreach \i in {0,...,7}  \node[vertex] (\i) at (-45.0*\i:3) {$v_\i$};
\foreach \i in {0,...,7}  \draw let \n1={int(mod(\i+1,8))} in (\i) -- (\n1);
\node (u0) at (-45.0*0:4) {};
\draw (0) -- (u0);
\node at (-45.0*0:2.1) {\small $2$};
\node (u1) at (-45.0*1:4) {};
\draw (1) -- (u1);
\node at (-45.0*1:2.1) {\small $2$};
\node[vertex] (v'2) at (-45.0*2:4.5) {$v'_2$};
\node (u'2) at (-45.0*2:5.5) {};
\node at (-45.0*2-11:4.5) {\small $3$};
\draw (2) -- (v'2);
\draw (v'2) -- (u'2);
\node at (-45.0*2:2.1) {\small $4$};
\node[vertex] (v'3) at (-45.0*3:4.5) {$v'_3$};
\node[vertex] (v''3) at (-45.0*3-7:6) {$v''_3$};
\node[vertex] (v'''3) at (-45.0*3+7:6) {$v'''_3$};
\node (u''3) at (-45.0*3-7:7) {};
\node (u'''3) at (-45.0*3+7:7) {};
\node at (-45.0*3-11:4.5) {\small $4$};
\node at (-45.0*3-14:6) {\small $3$};
\node at (-45.0*3+14:6) {\small $3$};
\draw (3) -- (v'3);
\draw (v'3) -- (v''3);
\draw (v'3) -- (v'''3);
\draw (v''3) -- (u''3);
\draw (v'''3) -- (u'''3);
\node at (-45.0*3:2.1) {\small $6$};
\node[vertex] (v'4) at (-45.0*4:4.5) {$v'_4$};
\node (u'4) at (-45.0*4:5.5) {};
\node at (-45.0*4-11:4.5) {\small $3$};
\draw (4) -- (v'4);
\draw (v'4) -- (u'4);
\node at (-45.0*4:2.1) {\small $5$};
\node[vertex] (v'5) at (-45.0*5:4.5) {$v'_5$};
\node (u'5) at (-45.0*5:5.5) {};
\node at (-45.0*5-11:4.5) {\small $3$};
\draw (5) -- (v'5);
\draw (v'5) -- (u'5);
\node at (-45.0*5:2.1) {\small $5$};
\node[vertex] (v'6) at (-45.0*6:4.5) {$v'_6$};
\node (u'6) at (-45.0*6:5.5) {};
\node at (-45.0*6-11:4.5) {\small $3$};
\draw (6) -- (v'6);
\draw (v'6) -- (u'6);
\node at (-45.0*6:2.1) {\small $5$};
\node[vertex] (v'7) at (-45.0*7:4.5) {$v'_7$};
\node (u'7) at (-45.0*7:5.5) {};
\node at (-45.0*7-11:4.5) {\small $3$};
\draw (7) -- (v'7);
\draw (v'7) -- (u'7);
\node at (-45.0*7:2.1) {\small $4$};
\end{tikzpicture}
}
}
\subfloat[\texttt{0a0c0c0c0b0c0c0c}]{
\label{subfig:0a0c0c0c0b0c0c0c}
\scalebox{0.65}{
	\begin{tikzpicture}[join=bevel,vertex/.style={circle,draw, minimum size=0.6cm},inner sep=0mm,scale=0.8]
\foreach \i in {0,...,7}  \node[vertex] (\i) at (-45.0*\i:3) {$v_\i$};
\foreach \i in {0,...,7}  \draw let \n1={int(mod(\i+1,8))} in (\i) -- (\n1);
\node (u0) at (-45.0*0:4) {};
\draw (0) -- (u0);
\node at (-45.0*0:2.1) {\small $3$};
\node[vertex] (v'1) at (-45.0*1:4.5) {$v'_1$};
\node (u'1) at (-45.0*1:5.5) {};
\node at (-45.0*1-11:4.5) {\small $3$};
\draw (1) -- (v'1);
\draw (v'1) -- (u'1);
\node at (-45.0*1:2.1) {\small $4$};
\node[vertex] (v'2) at (-45.0*2:4.5) {$v'_2$};
\node (u'2) at (-45.0*2:5.5) {};
\node at (-45.0*2-11:4.5) {\small $3$};
\draw (2) -- (v'2);
\draw (v'2) -- (u'2);
\node at (-45.0*2:2.1) {\small $5$};
\node[vertex] (v'3) at (-45.0*3:4.5) {$v'_3$};
\node (u'3) at (-45.0*3:5.5) {};
\node at (-45.0*3-11:4.5) {\small $3$};
\draw (3) -- (v'3);
\draw (v'3) -- (u'3);
\node at (-45.0*3:2.1) {\small $5$};
\node[vertex] (v'4) at (-45.0*4:4.5) {$v'_4$};
\node[vertex] (v''4) at (-45.0*4-7:6) {$v''_4$};
\node[vertex] (v'''4) at (-45.0*4+7:6) {$v'''_4$};
\node (u''4) at (-45.0*4-7:7) {};
\node (u'''4) at (-45.0*4+7:7) {};
\node at (-45.0*4-11:4.5) {\small $4$};
\node at (-45.0*4-14:6) {\small $3$};
\node at (-45.0*4+14:6) {\small $3$};
\draw (4) -- (v'4);
\draw (v'4) -- (v''4);
\draw (v'4) -- (v'''4);
\draw (v''4) -- (u''4);
\draw (v'''4) -- (u'''4);
\node at (-45.0*4:2.1) {\small $6$};
\node[vertex] (v'5) at (-45.0*5:4.5) {$v'_5$};
\node (u'5) at (-45.0*5:5.5) {};
\node at (-45.0*5-11:4.5) {\small $3$};
\draw (5) -- (v'5);
\draw (v'5) -- (u'5);
\node at (-45.0*5:2.1) {\small $5$};
\node[vertex] (v'6) at (-45.0*6:4.5) {$v'_6$};
\node (u'6) at (-45.0*6:5.5) {};
\node at (-45.0*6-11:4.5) {\small $3$};
\draw (6) -- (v'6);
\draw (v'6) -- (u'6);
\node at (-45.0*6:2.1) {\small $5$};
\node[vertex] (v'7) at (-45.0*7:4.5) {$v'_7$};
\node (u'7) at (-45.0*7:5.5) {};
\node at (-45.0*7-11:4.5) {\small $3$};
\draw (7) -- (v'7);
\draw (v'7) -- (u'7);
\node at (-45.0*7:2.1) {\small $4$};
\end{tikzpicture}
}
}
\subfloat[\texttt{1a0c0c0a0c0b0c}]{
\label{subfig:1a0c0c0a0c0b0c}
\scalebox{0.65}{
	\begin{tikzpicture}[join=bevel,vertex/.style={circle,draw, minimum size=0.6cm},inner sep=0mm,scale=0.8]
\foreach \i in {0,...,7}  \node[vertex] (\i) at (-45.0*\i:3) {$v_\i$};
\foreach \i in {0,...,7}  \draw let \n1={int(mod(\i+1,8))} in (\i) -- (\n1);
\node at (-45.0*0:2.1) {\small $5$};
\node (u1) at (-45.0*1:4) {};
\draw (1) -- (u1);
\node at (-45.0*1:2.1) {\small $3$};
\node[vertex] (v'2) at (-45.0*2:4.5) {$v'_2$};
\node (u'2) at (-45.0*2:5.5) {};
\node at (-45.0*2-11:4.5) {\small $3$};
\draw (2) -- (v'2);
\draw (v'2) -- (u'2);
\node at (-45.0*2:2.1) {\small $4$};
\node[vertex] (v'3) at (-45.0*3:4.5) {$v'_3$};
\node (u'3) at (-45.0*3:5.5) {};
\node at (-45.0*3-11:4.5) {\small $3$};
\draw (3) -- (v'3);
\draw (v'3) -- (u'3);
\node at (-45.0*3:2.1) {\small $4$};
\node (u4) at (-45.0*4:4) {};
\draw (4) -- (u4);
\node at (-45.0*4:2.1) {\small $3$};
\node[vertex] (v'5) at (-45.0*5:4.5) {$v'_5$};
\node (u'5) at (-45.0*5:5.5) {};
\node at (-45.0*5-11:4.5) {\small $3$};
\draw (5) -- (v'5);
\draw (v'5) -- (u'5);
\node at (-45.0*5:2.1) {\small $4$};
\node[vertex] (v'6) at (-45.0*6:4.5) {$v'_6$};
\node[vertex] (v''6) at (-45.0*6-7:6) {$v''_6$};
\node[vertex] (v'''6) at (-45.0*6+7:6) {$v'''_6$};
\node (u''6) at (-45.0*6-7:7) {};
\node (u'''6) at (-45.0*6+7:7) {};
\node at (-45.0*6-11:4.5) {\small $4$};
\node at (-45.0*6-14:6) {\small $3$};
\node at (-45.0*6+14:6) {\small $3$};
\draw (6) -- (v'6);
\draw (v'6) -- (v''6);
\draw (v'6) -- (v'''6);
\draw (v''6) -- (u''6);
\draw (v'''6) -- (u'''6);
\node at (-45.0*6:2.1) {\small $6$};
\node[vertex] (v'7) at (-45.0*7:4.5) {$v'_7$};
\node (u'7) at (-45.0*7:5.5) {};
\node at (-45.0*7-11:4.5) {\small $3$};
\draw (7) -- (v'7);
\draw (v'7) -- (u'7);
\node at (-45.0*7:2.1) {\small $5$};
\end{tikzpicture}
}
}

\phantomcaption
\end{figure}

\begin{figure}[H]
\ContinuedFloat

\hspace*{-1.5cm}
\subfloat[\texttt{1a0b0c0c0c0a0c}]{
\label{subfig:1a0b0c0c0c0a0c}
\scalebox{0.65}{
	\begin{tikzpicture}[join=bevel,vertex/.style={circle,draw, minimum size=0.6cm},inner sep=0mm,scale=0.8]
\foreach \i in {0,...,7}  \node[vertex] (\i) at (-45.0*\i:3) {$v_\i$};
\foreach \i in {0,...,7}  \draw let \n1={int(mod(\i+1,8))} in (\i) -- (\n1);
\node at (-45.0*0:2.1) {\small $5$};
\node (u1) at (-45.0*1:4) {};
\draw (1) -- (u1);
\node at (-45.0*1:2.1) {\small $3$};
\node[vertex] (v'2) at (-45.0*2:4.5) {$v'_2$};
\node[vertex] (v''2) at (-45.0*2-7:6) {$v''_2$};
\node[vertex] (v'''2) at (-45.0*2+7:6) {$v'''_2$};
\node (u''2) at (-45.0*2-7:7) {};
\node (u'''2) at (-45.0*2+7:7) {};
\node at (-45.0*2-11:4.5) {\small $4$};
\node at (-45.0*2-14:6) {\small $3$};
\node at (-45.0*2+14:6) {\small $3$};
\draw (2) -- (v'2);
\draw (v'2) -- (v''2);
\draw (v'2) -- (v'''2);
\draw (v''2) -- (u''2);
\draw (v'''2) -- (u'''2);
\node at (-45.0*2:2.1) {\small $5$};
\node[vertex] (v'3) at (-45.0*3:4.5) {$v'_3$};
\node (u'3) at (-45.0*3:5.5) {};
\node at (-45.0*3-11:4.5) {\small $3$};
\draw (3) -- (v'3);
\draw (v'3) -- (u'3);
\node at (-45.0*3:2.1) {\small $5$};
\node[vertex] (v'4) at (-45.0*4:4.5) {$v'_4$};
\node (u'4) at (-45.0*4:5.5) {};
\node at (-45.0*4-11:4.5) {\small $3$};
\draw (4) -- (v'4);
\draw (v'4) -- (u'4);
\node at (-45.0*4:2.1) {\small $5$};
\node[vertex] (v'5) at (-45.0*5:4.5) {$v'_5$};
\node (u'5) at (-45.0*5:5.5) {};
\node at (-45.0*5-11:4.5) {\small $3$};
\draw (5) -- (v'5);
\draw (v'5) -- (u'5);
\node at (-45.0*5:2.1) {\small $4$};
\node (u6) at (-45.0*6:4) {};
\draw (6) -- (u6);
\node at (-45.0*6:2.1) {\small $3$};
\node[vertex] (v'7) at (-45.0*7:4.5) {$v'_7$};
\node (u'7) at (-45.0*7:5.5) {};
\node at (-45.0*7-11:4.5) {\small $3$};
\draw (7) -- (v'7);
\draw (v'7) -- (u'7);
\node at (-45.0*7:2.1) {\small $4$};
\end{tikzpicture}
}
}
\subfloat[\texttt{1a0c0b0c0c0a0c}]{
\label{subfig:1a0c0b0c0c0a0c}
\scalebox{0.65}{
	\begin{tikzpicture}[join=bevel,vertex/.style={circle,draw, minimum size=0.6cm},inner sep=0mm,scale=0.8]
\foreach \i in {0,...,7}  \node[vertex] (\i) at (-45.0*\i:3) {$v_\i$};
\foreach \i in {0,...,7}  \draw let \n1={int(mod(\i+1,8))} in (\i) -- (\n1);
\node at (-45.0*0:2.1) {\small $5$};
\node (u1) at (-45.0*1:4) {};
\draw (1) -- (u1);
\node at (-45.0*1:2.1) {\small $3$};
\node[vertex] (v'2) at (-45.0*2:4.5) {$v'_2$};
\node (u'2) at (-45.0*2:5.5) {};
\node at (-45.0*2-11:4.5) {\small $3$};
\draw (2) -- (v'2);
\draw (v'2) -- (u'2);
\node at (-45.0*2:2.1) {\small $45$};
\node[vertex] (v'3) at (-45.0*3:4.5) {$v'_3$};
\node[vertex] (v''3) at (-45.0*3-7:6) {$v''_3$};
\node[vertex] (v'''3) at (-45.0*3+7:6) {$v'''_3$};
\node (u''3) at (-45.0*3-7:7) {};
\node (u'''3) at (-45.0*3+7:7) {};
\node at (-45.0*3-11:4.5) {\small $45$};
\node at (-45.0*3-14:6) {\small $3$};
\node at (-45.0*3+14:6) {\small $3$};
\draw (3) -- (v'3);
\draw (v'3) -- (v''3);
\draw (v'3) -- (v'''3);
\draw (v''3) -- (u''3);
\draw (v'''3) -- (u'''3);
\node at (-45.0*3:2.1) {\small $6$};
\node[vertex] (v'4) at (-45.0*4:4.5) {$v'_4$};
\node (u'4) at (-45.0*4:5.5) {};
\node at (-45.0*4-11:4.5) {\small $3$};
\draw (4) -- (v'4);
\draw (v'4) -- (u'4);
\node at (-45.0*4:2.1) {\small $5$};
\node[vertex] (v'5) at (-45.0*5:4.5) {$v'_5$};
\node (u'5) at (-45.0*5:5.5) {};
\node at (-45.0*5-11:4.5) {\small $3$};
\draw (5) -- (v'5);
\draw (v'5) -- (u'5);
\node at (-45.0*5:2.1) {\small $4$};
\node (u6) at (-45.0*6:4) {};
\draw (6) -- (u6);
\node at (-45.0*6:2.1) {\small $3$};
\node[vertex] (v'7) at (-45.0*7:4.5) {$v'_7$};
\node (u'7) at (-45.0*7:5.5) {};
\node at (-45.0*7-11:4.5) {\small $3$};
\draw (7) -- (v'7);
\draw (v'7) -- (u'7);
\node at (-45.0*7:2.1) {\small $4$};
\end{tikzpicture}
}
}
\subfloat[\texttt{1a0c0c0c0c0a0b}]{
\label{subfig:1a0c0c0c0c0a0b}
\scalebox{0.65}{
	\begin{tikzpicture}[join=bevel,vertex/.style={circle,draw, minimum size=0.6cm},inner sep=0mm,scale=0.8]
\foreach \i in {0,...,7}  \node[vertex] (\i) at (-45.0*\i:3) {$v_\i$};
\foreach \i in {0,...,7}  \draw let \n1={int(mod(\i+1,8))} in (\i) -- (\n1);
\node at (-45.0*0:2.1) {\small $5$};
\node (u1) at (-45.0*1:4) {};
\draw (1) -- (u1);
\node at (-45.0*1:2.1) {\small $3$};
\node[vertex] (v'2) at (-45.0*2:4.5) {$v'_2$};
\node (u'2) at (-45.0*2:5.5) {};
\node at (-45.0*2-11:4.5) {\small $3$};
\draw (2) -- (v'2);
\draw (v'2) -- (u'2);
\node at (-45.0*2:2.1) {\small $4$};
\node[vertex] (v'3) at (-45.0*3:4.5) {$v'_3$};
\node (u'3) at (-45.0*3:5.5) {};
\node at (-45.0*3-11:4.5) {\small $3$};
\draw (3) -- (v'3);
\draw (v'3) -- (u'3);
\node at (-45.0*3:2.1) {\small $5$};
\node[vertex] (v'4) at (-45.0*4:4.5) {$v'_4$};
\node (u'4) at (-45.0*4:5.5) {};
\node at (-45.0*4-11:4.5) {\small $3$};
\draw (4) -- (v'4);
\draw (v'4) -- (u'4);
\node at (-45.0*4:2.1) {\small $5$};
\node[vertex] (v'5) at (-45.0*5:4.5) {$v'_5$};
\node (u'5) at (-45.0*5:5.5) {};
\node at (-45.0*5-11:4.5) {\small $3$};
\draw (5) -- (v'5);
\draw (v'5) -- (u'5);
\node at (-45.0*5:2.1) {\small $4$};
\node (u6) at (-45.0*6:4) {};
\draw (6) -- (u6);
\node at (-45.0*6:2.1) {\small $3$};
\node[vertex] (v'7) at (-45.0*7:4.5) {$v'_7$};
\node[vertex] (v''7) at (-45.0*7-7:6) {$v''_7$};
\node[vertex] (v'''7) at (-45.0*7+7:6) {$v'''_7$};
\node (u''7) at (-45.0*7-7:7) {};
\node (u'''7) at (-45.0*7+7:7) {};
\node at (-45.0*7-11:4.5) {\small $4$};
\node at (-45.0*7-14:6) {\small $3$};
\node at (-45.0*7+14:6) {\small $3$};
\draw (7) -- (v'7);
\draw (v'7) -- (v''7);
\draw (v'7) -- (v'''7);
\draw (v''7) -- (u''7);
\draw (v'''7) -- (u'''7);
\node at (-45.0*7:2.1) {\small $5$};
\end{tikzpicture}
}
}

\hspace*{-1.5cm}
\subfloat[\texttt{1a0c0c0b0c0a0c}]{
\label{subfig:1a0c0c0b0c0a0c}
\scalebox{0.65}{
	\begin{tikzpicture}[join=bevel,vertex/.style={circle,draw, minimum size=0.6cm},inner sep=0mm,scale=0.8]
\foreach \i in {0,...,7}  \node[vertex] (\i) at (-45.0*\i:3) {$v_\i$};
\foreach \i in {0,...,7}  \draw let \n1={int(mod(\i+1,8))} in (\i) -- (\n1);
\node at (-45.0*0:2.1) {\small $5$};
\node (u1) at (-45.0*1:4) {};
\draw (1) -- (u1);
\node at (-45.0*1:2.1) {\small $3$};
\node[vertex] (v'2) at (-45.0*2:4.5) {$v'_2$};
\node (u'2) at (-45.0*2:5.5) {};
\node at (-45.0*2-11:4.5) {\small $3$};
\draw (2) -- (v'2);
\draw (v'2) -- (u'2);
\node at (-45.0*2:2.1) {\small $4$};
\node[vertex] (v'3) at (-45.0*3:4.5) {$v'_3$};
\node (u'3) at (-45.0*3:5.5) {};
\node at (-45.0*3-11:4.5) {\small $3$};
\draw (3) -- (v'3);
\draw (v'3) -- (u'3);
\node at (-45.0*3:2.1) {\small $5$};
\node[vertex] (v'4) at (-45.0*4:4.5) {$v'_4$};
\node[vertex] (v''4) at (-45.0*4-7:6) {$v''_4$};
\node[vertex] (v'''4) at (-45.0*4+7:6) {$v'''_4$};
\node (u''4) at (-45.0*4-7:7) {};
\node (u'''4) at (-45.0*4+7:7) {};
\node at (-45.0*4-11:4.5) {\small $4$};
\node at (-45.0*4-14:6) {\small $3$};
\node at (-45.0*4+14:6) {\small $3$};
\draw (4) -- (v'4);
\draw (v'4) -- (v''4);
\draw (v'4) -- (v'''4);
\draw (v''4) -- (u''4);
\draw (v'''4) -- (u'''4);
\node at (-45.0*4:2.1) {\small $6$};
\node[vertex] (v'5) at (-45.0*5:4.5) {$v'_5$};
\node (u'5) at (-45.0*5:5.5) {};
\node at (-45.0*5-11:4.5) {\small $3$};
\draw (5) -- (v'5);
\draw (v'5) -- (u'5);
\node at (-45.0*5:2.1) {\small $4$};
\node (u6) at (-45.0*6:4) {};
\draw (6) -- (u6);
\node at (-45.0*6:2.1) {\small $3$};
\node[vertex] (v'7) at (-45.0*7:4.5) {$v'_7$};
\node (u'7) at (-45.0*7:5.5) {};
\node at (-45.0*7-11:4.5) {\small $3$};
\draw (7) -- (v'7);
\draw (v'7) -- (u'7);
\node at (-45.0*7:2.1) {\small $4$};
\end{tikzpicture}
}
}
\subfloat[\texttt{1a0c0b0c0c0c0a}]{
\label{subfig:1a0c0b0c0c0c0a}
\scalebox{0.65}{
	\begin{tikzpicture}[join=bevel,vertex/.style={circle,draw, minimum size=0.6cm},inner sep=0mm,scale=0.8]
\foreach \i in {0,...,7}  \node[vertex] (\i) at (-45.0*\i:3) {$v_\i$};
\foreach \i in {0,...,7}  \draw let \n1={int(mod(\i+1,8))} in (\i) -- (\n1);
\node at (-45.0*0:2.1) {\small $4$};
\node (u1) at (-45.0*1:4) {};
\draw (1) -- (u1);
\node at (-45.0*1:2.1) {\small $3$};
\node[vertex] (v'2) at (-45.0*2:4.5) {$v'_2$};
\node (u'2) at (-45.0*2:5.5) {};
\node at (-45.0*2-11:4.5) {\small $3$};
\draw (2) -- (v'2);
\draw (v'2) -- (u'2);
\node at (-45.0*2:2.1) {\small $4$};
\node[vertex] (v'3) at (-45.0*3:4.5) {$v'_3$};
\node[vertex] (v''3) at (-45.0*3-7:6) {$v''_3$};
\node[vertex] (v'''3) at (-45.0*3+7:6) {$v'''_3$};
\node (u''3) at (-45.0*3-7:7) {};
\node (u'''3) at (-45.0*3+7:7) {};
\node at (-45.0*3-11:4.5) {\small $4$};
\node at (-45.0*3-14:6) {\small $3$};
\node at (-45.0*3+14:6) {\small $3$};
\draw (3) -- (v'3);
\draw (v'3) -- (v''3);
\draw (v'3) -- (v'''3);
\draw (v''3) -- (u''3);
\draw (v'''3) -- (u'''3);
\node at (-45.0*3:2.1) {\small $6$};
\node[vertex] (v'4) at (-45.0*4:4.5) {$v'_4$};
\node (u'4) at (-45.0*4:5.5) {};
\node at (-45.0*4-11:4.5) {\small $3$};
\draw (4) -- (v'4);
\draw (v'4) -- (u'4);
\node at (-45.0*4:2.1) {\small $5$};
\node[vertex] (v'5) at (-45.0*5:4.5) {$v'_5$};
\node (u'5) at (-45.0*5:5.5) {};
\node at (-45.0*5-11:4.5) {\small $3$};
\draw (5) -- (v'5);
\draw (v'5) -- (u'5);
\node at (-45.0*5:2.1) {\small $5$};
\node[vertex] (v'6) at (-45.0*6:4.5) {$v'_6$};
\node (u'6) at (-45.0*6:5.5) {};
\node at (-45.0*6-11:4.5) {\small $3$};
\draw (6) -- (v'6);
\draw (v'6) -- (u'6);
\node at (-45.0*6:2.1) {\small $4$};
\node (u7) at (-45.0*7:4) {};
\draw (7) -- (u7);
\node at (-45.0*7:2.1) {\small $3$};
\end{tikzpicture}
}
}
\subfloat[\texttt{1a0c0c0b0c0c0a}]{
\label{subfig:1a0c0c0b0c0c0a}
\scalebox{0.65}{
	\begin{tikzpicture}[join=bevel,vertex/.style={circle,draw, minimum size=0.6cm},inner sep=0mm,scale=0.8]
\foreach \i in {0,...,7}  \node[vertex] (\i) at (-45.0*\i:3) {$v_\i$};
\foreach \i in {0,...,7}  \draw let \n1={int(mod(\i+1,8))} in (\i) -- (\n1);
\node at (-45.0*0:2.1) {\small $4$};
\node (u1) at (-45.0*1:4) {};
\draw (1) -- (u1);
\node at (-45.0*1:2.1) {\small $3$};
\node[vertex] (v'2) at (-45.0*2:4.5) {$v'_2$};
\node (u'2) at (-45.0*2:5.5) {};
\node at (-45.0*2-11:4.5) {\small $3$};
\draw (2) -- (v'2);
\draw (v'2) -- (u'2);
\node at (-45.0*2:2.1) {\small $4$};
\node[vertex] (v'3) at (-45.0*3:4.5) {$v'_3$};
\node (u'3) at (-45.0*3:5.5) {};
\node at (-45.0*3-11:4.5) {\small $3$};
\draw (3) -- (v'3);
\draw (v'3) -- (u'3);
\node at (-45.0*3:2.1) {\small $5$};
\node[vertex] (v'4) at (-45.0*4:4.5) {$v'_4$};
\node[vertex] (v''4) at (-45.0*4-7:6) {$v''_4$};
\node[vertex] (v'''4) at (-45.0*4+7:6) {$v'''_4$};
\node (u''4) at (-45.0*4-7:7) {};
\node (u'''4) at (-45.0*4+7:7) {};
\node at (-45.0*4-11:4.5) {\small $4$};
\node at (-45.0*4-14:6) {\small $3$};
\node at (-45.0*4+14:6) {\small $3$};
\draw (4) -- (v'4);
\draw (v'4) -- (v''4);
\draw (v'4) -- (v'''4);
\draw (v''4) -- (u''4);
\draw (v'''4) -- (u'''4);
\node at (-45.0*4:2.1) {\small $6$};
\node[vertex] (v'5) at (-45.0*5:4.5) {$v'_5$};
\node (u'5) at (-45.0*5:5.5) {};
\node at (-45.0*5-11:4.5) {\small $3$};
\draw (5) -- (v'5);
\draw (v'5) -- (u'5);
\node at (-45.0*5:2.1) {\small $5$};
\node[vertex] (v'6) at (-45.0*6:4.5) {$v'_6$};
\node (u'6) at (-45.0*6:5.5) {};
\node at (-45.0*6-11:4.5) {\small $3$};
\draw (6) -- (v'6);
\draw (v'6) -- (u'6);
\node at (-45.0*6:2.1) {\small $4$};
\node (u7) at (-45.0*7:4) {};
\draw (7) -- (u7);
\node at (-45.0*7:2.1) {\small $3$};
\end{tikzpicture}
}
}

\subfloat[\texttt{1a0a1c0b0b0c}]{
\label{subfig:1a0a1c0b0b0c}
\scalebox{0.65}{
	\begin{tikzpicture}[join=bevel,vertex/.style={circle,draw, minimum size=0.6cm},inner sep=0mm,scale=0.8]
\foreach \i in {0,...,7}  \node[vertex] (\i) at (-45.0*\i:3) {$v_\i$};
\foreach \i in {0,...,7}  \draw let \n1={int(mod(\i+1,8))} in (\i) -- (\n1);
\node at (-45.0*0:2.1) {\small $5$};
\node at (-45.0*3:2.1) {\small $5$};
\node (u1) at (-45.0*1:4) {};
\draw (1) -- (u1);
\node at (-45.0*1:2.1) {\small $2$};
\node (u2) at (-45.0*2:4) {};
\draw (2) -- (u2);
\node at (-45.0*2:2.1) {\small $2$};
\node[vertex] (v'4) at (-45.0*4:4.5) {$v'_4$};
\node (u'4) at (-45.0*4:5.5) {};
\node at (-45.0*4-11:4.5) {\small $3$};
\draw (4) -- (v'4);
\draw (v'4) -- (u'4);
\node at (-45.0*4:2.1) {\small $5$};
\node[vertex] (v'5) at (-45.0*5:4.5) {$v'_5$};
\node[vertex] (v''5) at (-45.0*5-7:6) {$v''_5$};
\node[vertex] (v'''5) at (-45.0*5+7:6) {$v'''_5$};
\node (u''5) at (-45.0*5-7:7) {};
\node (u'''5) at (-45.0*5+7:7) {};
\node at (-45.0*5-11:4.5) {\small $4$};
\node at (-45.0*5-14:6) {\small $3$};
\node at (-45.0*5+14:6) {\small $3$};
\draw (5) -- (v'5);
\draw (v'5) -- (v''5);
\draw (v'5) -- (v'''5);
\draw (v''5) -- (u''5);
\draw (v'''5) -- (u'''5);
\node at (-45.0*5:2.1) {\small $6$};
\node[vertex] (v'6) at (-45.0*6:4.5) {$v'_6$};
\node[vertex] (v''6) at (-45.0*6-7:6) {$v''_6$};
\node[vertex] (v'''6) at (-45.0*6+7:6) {$v'''_6$};
\node (u''6) at (-45.0*6-7:7) {};
\node (u'''6) at (-45.0*6+7:7) {};
\node at (-45.0*6-11:4.5) {\small $4$};
\node at (-45.0*6-14:6) {\small $3$};
\node at (-45.0*6+14:6) {\small $3$};
\draw (6) -- (v'6);
\draw (v'6) -- (v''6);
\draw (v'6) -- (v'''6);
\draw (v''6) -- (u''6);
\draw (v'''6) -- (u'''6);
\node at (-45.0*6:2.1) {\small $6$};
\node[vertex] (v'7) at (-45.0*7:4.5) {$v'_7$};
\node (u'7) at (-45.0*7:5.5) {};
\node at (-45.0*7-11:4.5) {\small $3$};
\draw (7) -- (v'7);
\draw (v'7) -- (u'7);
\node at (-45.0*7:2.1) {\small $5$};
\end{tikzpicture}
}
}
\subfloat[\texttt{1a0c1a0c0a0b}]{
\label{subfig:1a0c1a0c0a0b}
\scalebox{0.65}{
	\begin{tikzpicture}[join=bevel,vertex/.style={circle,draw, minimum size=0.6cm},inner sep=0mm,scale=0.8]
\foreach \i in {0,...,7}  \node[vertex] (\i) at (-45.0*\i:3) {$v_\i$};
\foreach \i in {0,...,7}  \draw let \n1={int(mod(\i+1,8))} in (\i) -- (\n1);
\node at (-45.0*0:2.1) {\small $5$};
\node at (-45.0*3:2.1) {\small $5$};
\node (u1) at (-45.0*1:4) {};
\draw (1) -- (u1);
\node at (-45.0*1:2.1) {\small $3$};
\node[vertex] (v'2) at (-45.0*2:4.5) {$v'_2$};
\node (u'2) at (-45.0*2:5.5) {};
\node at (-45.0*2-11:4.5) {\small $3$};
\draw (2) -- (v'2);
\draw (v'2) -- (u'2);
\node at (-45.0*2:2.1) {\small $4$};
\node (u4) at (-45.0*4:4) {};
\draw (4) -- (u4);
\node at (-45.0*4:2.1) {\small $3$};
\node[vertex] (v'5) at (-45.0*5:4.5) {$v'_5$};
\node (u'5) at (-45.0*5:5.5) {};
\node at (-45.0*5-11:4.5) {\small $3$};
\draw (5) -- (v'5);
\draw (v'5) -- (u'5);
\node at (-45.0*5:2.1) {\small $3$};
\node (u6) at (-45.0*6:4) {};
\draw (6) -- (u6);
\node at (-45.0*6:2.1) {\small $3$};
\node[vertex] (v'7) at (-45.0*7:4.5) {$v'_7$};
\node[vertex] (v''7) at (-45.0*7-7:6) {$v''_7$};
\node[vertex] (v'''7) at (-45.0*7+7:6) {$v'''_7$};
\node (u''7) at (-45.0*7-7:7) {};
\node (u'''7) at (-45.0*7+7:7) {};
\node at (-45.0*7-11:4.5) {\small $4$};
\node at (-45.0*7-14:6) {\small $3$};
\node at (-45.0*7+14:6) {\small $3$};
\draw (7) -- (v'7);
\draw (v'7) -- (v''7);
\draw (v'7) -- (v'''7);
\draw (v''7) -- (u''7);
\draw (v'''7) -- (u'''7);
\node at (-45.0*7:2.1) {\small $5$};
\end{tikzpicture}
}
}
\subfloat[\texttt{1a0c1a0b0c0a}]{
\label{subfig:1a0c1a0b0c0a}
\scalebox{0.65}{
	\begin{tikzpicture}[join=bevel,vertex/.style={circle,draw, minimum size=0.6cm},inner sep=0mm,scale=0.8]
\foreach \i in {0,...,7}  \node[vertex] (\i) at (-45.0*\i:3) {$v_\i$};
\foreach \i in {0,...,7}  \draw let \n1={int(mod(\i+1,8))} in (\i) -- (\n1);
\node at (-45.0*0:2.1) {\small $4$};
\node at (-45.0*3:2.1) {\small $5$};
\node (u1) at (-45.0*1:4) {};
\draw (1) -- (u1);
\node at (-45.0*1:2.1) {\small $3$};
\node[vertex] (v'2) at (-45.0*2:4.5) {$v'_2$};
\node (u'2) at (-45.0*2:5.5) {};
\node at (-45.0*2-11:4.5) {\small $3$};
\draw (2) -- (v'2);
\draw (v'2) -- (u'2);
\node at (-45.0*2:2.1) {\small $4$};
\node (u4) at (-45.0*4:4) {};
\draw (4) -- (u4);
\node at (-45.0*4:2.1) {\small $3$};
\node[vertex] (v'5) at (-45.0*5:4.5) {$v'_5$};
\node[vertex] (v''5) at (-45.0*5-7:6) {$v''_5$};
\node[vertex] (v'''5) at (-45.0*5+7:6) {$v'''_5$};
\node (u''5) at (-45.0*5-7:7) {};
\node (u'''5) at (-45.0*5+7:7) {};
\node at (-45.0*5-11:4.5) {\small $4$};
\node at (-45.0*5-14:6) {\small $3$};
\node at (-45.0*5+14:6) {\small $3$};
\draw (5) -- (v'5);
\draw (v'5) -- (v''5);
\draw (v'5) -- (v'''5);
\draw (v''5) -- (u''5);
\draw (v'''5) -- (u'''5);
\node at (-45.0*5:2.1) {\small $5$};
\node[vertex] (v'6) at (-45.0*6:4.5) {$v'_6$};
\node (u'6) at (-45.0*6:5.5) {};
\node at (-45.0*6-11:4.5) {\small $3$};
\draw (6) -- (v'6);
\draw (v'6) -- (u'6);
\node at (-45.0*6:2.1) {\small $4$};
\node (u7) at (-45.0*7:4) {};
\draw (7) -- (u7);
\node at (-45.0*7:2.1) {\small $3$};
\end{tikzpicture}
}
}

\phantomcaption
\end{figure}

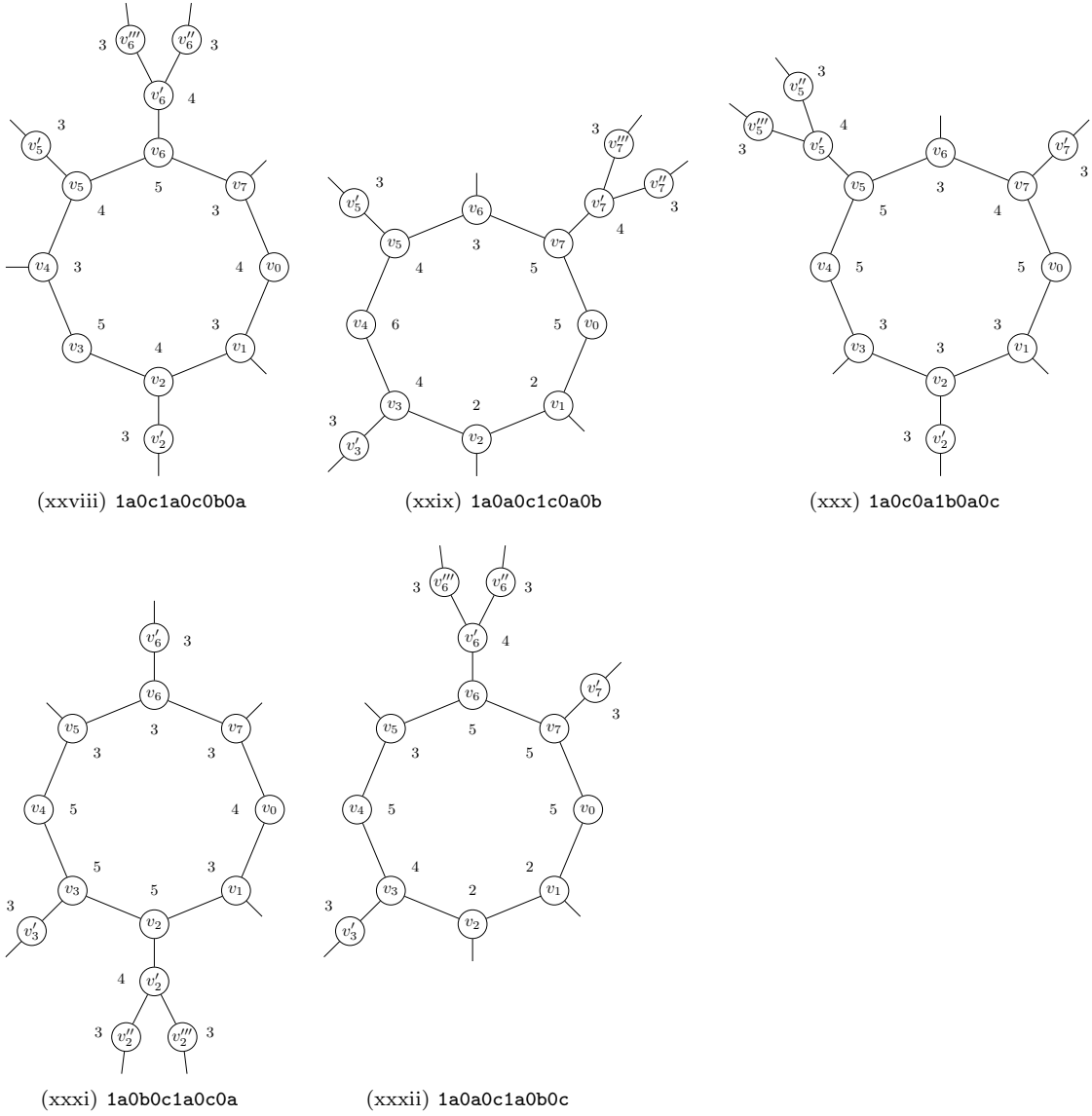
\begin{figure}[H]
\ContinuedFloat

\subfloat[\texttt{1a0c1a0c0b0a}]{
\label{subfig:1a0c1a0c0b0a}
\scalebox{0.65}{
	\begin{tikzpicture}[join=bevel,vertex/.style={circle,draw, minimum size=0.6cm},inner sep=0mm,scale=0.8]
\foreach \i in {0,...,7}  \node[vertex] (\i) at (-45.0*\i:3) {$v_\i$};
\foreach \i in {0,...,7}  \draw let \n1={int(mod(\i+1,8))} in (\i) -- (\n1);
\node at (-45.0*0:2.1) {\small $4$};
\node at (-45.0*3:2.1) {\small $5$};
\node (u1) at (-45.0*1:4) {};
\draw (1) -- (u1);
\node at (-45.0*1:2.1) {\small $3$};
\node[vertex] (v'2) at (-45.0*2:4.5) {$v'_2$};
\node (u'2) at (-45.0*2:5.5) {};
\node at (-45.0*2-11:4.5) {\small $3$};
\draw (2) -- (v'2);
\draw (v'2) -- (u'2);
\node at (-45.0*2:2.1) {\small $4$};
\node (u4) at (-45.0*4:4) {};
\draw (4) -- (u4);
\node at (-45.0*4:2.1) {\small $3$};
\node[vertex] (v'5) at (-45.0*5:4.5) {$v'_5$};
\node (u'5) at (-45.0*5:5.5) {};
\node at (-45.0*5-11:4.5) {\small $3$};
\draw (5) -- (v'5);
\draw (v'5) -- (u'5);
\node at (-45.0*5:2.1) {\small $4$};
\node[vertex] (v'6) at (-45.0*6:4.5) {$v'_6$};
\node[vertex] (v''6) at (-45.0*6-7:6) {$v''_6$};
\node[vertex] (v'''6) at (-45.0*6+7:6) {$v'''_6$};
\node (u''6) at (-45.0*6-7:7) {};
\node (u'''6) at (-45.0*6+7:7) {};
\node at (-45.0*6-11:4.5) {\small $4$};
\node at (-45.0*6-14:6) {\small $3$};
\node at (-45.0*6+14:6) {\small $3$};
\draw (6) -- (v'6);
\draw (v'6) -- (v''6);
\draw (v'6) -- (v'''6);
\draw (v''6) -- (u''6);
\draw (v'''6) -- (u'''6);
\node at (-45.0*6:2.1) {\small $5$};
\node (u7) at (-45.0*7:4) {};
\draw (7) -- (u7);
\node at (-45.0*7:2.1) {\small $3$};
\end{tikzpicture}
}
}
\subfloat[\texttt{1a0a0c1c0a0b}]{
\label{subfig:1a0a0c1c0a0b}
\scalebox{0.65}{
	\begin{tikzpicture}[join=bevel,vertex/.style={circle,draw, minimum size=0.6cm},inner sep=0mm,scale=0.8]
\foreach \i in {0,...,7}  \node[vertex] (\i) at (-45.0*\i:3) {$v_\i$};
\foreach \i in {0,...,7}  \draw let \n1={int(mod(\i+1,8))} in (\i) -- (\n1);
\node at (-45.0*0:2.1) {\small $5$};
\node at (-45.0*4:2.1) {\small $6$};
\node (u1) at (-45.0*1:4) {};
\draw (1) -- (u1);
\node at (-45.0*1:2.1) {\small $2$};
\node (u2) at (-45.0*2:4) {};
\draw (2) -- (u2);
\node at (-45.0*2:2.1) {\small $2$};
\node[vertex] (v'3) at (-45.0*3:4.5) {$v'_3$};
\node (u'3) at (-45.0*3:5.5) {};
\node at (-45.0*3-11:4.5) {\small $3$};
\draw (3) -- (v'3);
\draw (v'3) -- (u'3);
\node at (-45.0*3:2.1) {\small $4$};
\node[vertex] (v'5) at (-45.0*5:4.5) {$v'_5$};
\node (u'5) at (-45.0*5:5.5) {};
\node at (-45.0*5-11:4.5) {\small $3$};
\draw (5) -- (v'5);
\draw (v'5) -- (u'5);
\node at (-45.0*5:2.1) {\small $4$};
\node (u6) at (-45.0*6:4) {};
\draw (6) -- (u6);
\node at (-45.0*6:2.1) {\small $3$};
\node[vertex] (v'7) at (-45.0*7:4.5) {$v'_7$};
\node[vertex] (v''7) at (-45.0*7-7:6) {$v''_7$};
\node[vertex] (v'''7) at (-45.0*7+7:6) {$v'''_7$};
\node (u''7) at (-45.0*7-7:7) {};
\node (u'''7) at (-45.0*7+7:7) {};
\node at (-45.0*7-11:4.5) {\small $4$};
\node at (-45.0*7-14:6) {\small $3$};
\node at (-45.0*7+14:6) {\small $3$};
\draw (7) -- (v'7);
\draw (v'7) -- (v''7);
\draw (v'7) -- (v'''7);
\draw (v''7) -- (u''7);
\draw (v'''7) -- (u'''7);
\node at (-45.0*7:2.1) {\small $5$};
\end{tikzpicture}
}
}
\subfloat[\texttt{1a0c0a1b0a0c}]{
\label{subfig:1a0c0a1b0a0c}
\scalebox{0.65}{
	\begin{tikzpicture}[join=bevel,vertex/.style={circle,draw, minimum size=0.6cm},inner sep=0mm,scale=0.8]
\foreach \i in {0,...,7}  \node[vertex] (\i) at (-45.0*\i:3) {$v_\i$};
\foreach \i in {0,...,7}  \draw let \n1={int(mod(\i+1,8))} in (\i) -- (\n1);
\node at (-45.0*0:2.1) {\small $5$};
\node at (-45.0*4:2.1) {\small $5$};
\node (u1) at (-45.0*1:4) {};
\draw (1) -- (u1);
\node at (-45.0*1:2.1) {\small $3$};
\node[vertex] (v'2) at (-45.0*2:4.5) {$v'_2$};
\node (u'2) at (-45.0*2:5.5) {};
\node at (-45.0*2-11:4.5) {\small $3$};
\draw (2) -- (v'2);
\draw (v'2) -- (u'2);
\node at (-45.0*2:2.1) {\small $3$};
\node (u3) at (-45.0*3:4) {};
\draw (3) -- (u3);
\node at (-45.0*3:2.1) {\small $3$};
\node[vertex] (v'5) at (-45.0*5:4.5) {$v'_5$};
\node[vertex] (v''5) at (-45.0*5-7:6) {$v''_5$};
\node[vertex] (v'''5) at (-45.0*5+7:6) {$v'''_5$};
\node (u''5) at (-45.0*5-7:7) {};
\node (u'''5) at (-45.0*5+7:7) {};
\node at (-45.0*5-11:4.5) {\small $4$};
\node at (-45.0*5-14:6) {\small $3$};
\node at (-45.0*5+14:6) {\small $3$};
\draw (5) -- (v'5);
\draw (v'5) -- (v''5);
\draw (v'5) -- (v'''5);
\draw (v''5) -- (u''5);
\draw (v'''5) -- (u'''5);
\node at (-45.0*5:2.1) {\small $5$};
\node (u6) at (-45.0*6:4) {};
\draw (6) -- (u6);
\node at (-45.0*6:2.1) {\small $3$};
\node[vertex] (v'7) at (-45.0*7:4.5) {$v'_7$};
\node (u'7) at (-45.0*7:5.5) {};
\node at (-45.0*7-11:4.5) {\small $3$};
\draw (7) -- (v'7);
\draw (v'7) -- (u'7);
\node at (-45.0*7:2.1) {\small $4$};
\end{tikzpicture}
}
}

\subfloat[\texttt{1a0b0c1a0c0a}]{
\label{subfig:1a0b0c1a0c0a}
\scalebox{0.65}{
	\begin{tikzpicture}[join=bevel,vertex/.style={circle,draw, minimum size=0.6cm},inner sep=0mm,scale=0.8]
\foreach \i in {0,...,7}  \node[vertex] (\i) at (-45.0*\i:3) {$v_\i$};
\foreach \i in {0,...,7}  \draw let \n1={int(mod(\i+1,8))} in (\i) -- (\n1);
\node at (-45.0*0:2.1) {\small $4$};
\node at (-45.0*4:2.1) {\small $5$};
\node (u1) at (-45.0*1:4) {};
\draw (1) -- (u1);
\node at (-45.0*1:2.1) {\small $3$};
\node[vertex] (v'2) at (-45.0*2:4.5) {$v'_2$};
\node[vertex] (v''2) at (-45.0*2-7:6) {$v''_2$};
\node[vertex] (v'''2) at (-45.0*2+7:6) {$v'''_2$};
\node (u''2) at (-45.0*2-7:7) {};
\node (u'''2) at (-45.0*2+7:7) {};
\node at (-45.0*2-11:4.5) {\small $4$};
\node at (-45.0*2-14:6) {\small $3$};
\node at (-45.0*2+14:6) {\small $3$};
\draw (2) -- (v'2);
\draw (v'2) -- (v''2);
\draw (v'2) -- (v'''2);
\draw (v''2) -- (u''2);
\draw (v'''2) -- (u'''2);
\node at (-45.0*2:2.1) {\small $5$};
\node[vertex] (v'3) at (-45.0*3:4.5) {$v'_3$};
\node (u'3) at (-45.0*3:5.5) {};
\node at (-45.0*3-11:4.5) {\small $3$};
\draw (3) -- (v'3);
\draw (v'3) -- (u'3);
\node at (-45.0*3:2.1) {\small $5$};
\node (u5) at (-45.0*5:4) {};
\draw (5) -- (u5);
\node at (-45.0*5:2.1) {\small $3$};
\node[vertex] (v'6) at (-45.0*6:4.5) {$v'_6$};
\node (u'6) at (-45.0*6:5.5) {};
\node at (-45.0*6-11:4.5) {\small $3$};
\draw (6) -- (v'6);
\draw (v'6) -- (u'6);
\node at (-45.0*6:2.1) {\small $3$};
\node (u7) at (-45.0*7:4) {};
\draw (7) -- (u7);
\node at (-45.0*7:2.1) {\small $3$};
\end{tikzpicture}
}
}
\subfloat[\texttt{1a0a0c1a0b0c}]{
\label{subfig:1a0a0c1a0b0c}
\scalebox{0.65}{
	\begin{tikzpicture}[join=bevel,vertex/.style={circle,draw, minimum size=0.6cm},inner sep=0mm,scale=0.8]
\foreach \i in {0,...,7}  \node[vertex] (\i) at (-45.0*\i:3) {$v_\i$};
\foreach \i in {0,...,7}  \draw let \n1={int(mod(\i+1,8))} in (\i) -- (\n1);
\node at (-45.0*0:2.1) {\small $5$};
\node at (-45.0*4:2.1) {\small $5$};
\node (u1) at (-45.0*1:4) {};
\draw (1) -- (u1);
\node at (-45.0*1:2.1) {\small $2$};
\node (u''2) at (-45.0*2-7:7) {};
\node (u'''2) at (-45.0*2+7:7) {};
\node (u2) at (-45.0*2:4) {};
\draw (2) -- (u2);
\node at (-45.0*2:2.1) {\small $2$};
\node[vertex] (v'3) at (-45.0*3:4.5) {$v'_3$};
\node (u'3) at (-45.0*3:5.5) {};
\node at (-45.0*3-11:4.5) {\small $3$};
\draw (3) -- (v'3);
\draw (v'3) -- (u'3);
\node at (-45.0*3:2.1) {\small $4$};
\node (u5) at (-45.0*5:4) {};
\draw (5) -- (u5);
\node at (-45.0*5:2.1) {\small $3$};
\node[vertex] (v'6) at (-45.0*6:4.5) {$v'_6$};
\node[vertex] (v''6) at (-45.0*6-7:6) {$v''_6$};
\node[vertex] (v'''6) at (-45.0*6+7:6) {$v'''_6$};
\node (u''6) at (-45.0*6-7:7) {};
\node (u'''6) at (-45.0*6+7:7) {};
\node at (-45.0*6-11:4.5) {\small $4$};
\node at (-45.0*6-14:6) {\small $3$};
\node at (-45.0*6+14:6) {\small $3$};
\draw (6) -- (v'6);
\draw (v'6) -- (v''6);
\draw (v'6) -- (v'''6);
\draw (v''6) -- (u''6);
\draw (v'''6) -- (u'''6);
\node at (-45.0*6:2.1) {\small $5$};
\node[vertex] (v'7) at (-45.0*7:4.5) {$v'_7$};
\node (u'7) at (-45.0*7:5.5) {};
\node at (-45.0*7-11:4.5) {\small $3$};
\draw (7) -- (v'7);
\draw (v'7) -- (u'7);
\node at (-45.0*7:2.1) {\small $5$};
\end{tikzpicture}
}
}

\caption{Reducible cycles.}
\label{fig:reducible_cycles}
\end{figure}

Proof of \Cref{lemma:reducible_cycles}.

\begin{proof}
The outline of the following proofs uses the same conventions as in the proof \Cref{lemma:special_reducible}.

\figureproof{\Cref{subfig:1a1a0c0a0b0c}}
We have \gs. Now, we redefine $S=\{v'_7,v_7,v_0,v_1,v_2\}$ and consider a coloring $\phi$ of $G-S$. The list of remaining colors for $v'_7$, $v_7$, $v_0$, $v_1$, and $v_2$ are at least 2, 2, 4, 2, and 2 respectively and we can assume w.l.o.g. that they are respectively $\{a,b\}$, $\{a,b\}$, $\{a,b,c,d\}$, $\{c,d\}$, and $\{c,d\}$ by \Cref{fig:forced_non_colorable_P5}. Now, we uncolor $v_3$, $v_4$, $v'_4$, $v_5$, $v_6$, $v'_6$, $v''_6$, and $v'''_6$. The lower bounds on the lists of available colors for every vertex now corresponds to the ones indicated on the figure. Let $\phi(v_3)=x$, $\phi(v_4)=y$, and $\phi(v_6)=z$. We deduce that $L(v'_7)=\{a,b,z\}$, $\{a,b,z\}\subset L(v_7)$, $L(v_0)=\{a,b,c,d,z\}$, $L(v_1)=\{c,d,x\}$, and $L(v_2)=\{c,d,x,y\}$. Moreover, we claim that $L(v'_4)\neq L(v_3)$. Otherwise, we can simply switch the colors of $v_3$ and $v'_4$ in $\phi$ and we can extend this coloring to $S$ by \Cref{fig:forced_non_colorable_P5} as the remaining colors for $v_1$ and $v_2$ would no longer be $\{c,d\}$ while the remaining colors for $v'_7$, $v_7$, and $v_0$ stay the same.

Thanks to the observations above, we can color these vertices as follow. First, we restrict $L(v_5)$ to $L(v_5)\setminus\{a,b\}$. Since $|L(v_4)|\geq 3$, we color $v_4$ with a color different from $x$ and $y$. Now, we color $v_5$, $v_6$, $v'_6$, $v''_6$, and $v''_6$ by \Cref{subfig:config5}. Vertices $v_3$ and $v'_4$ are colorable since $L(v'_4)\neq L(v_3)$. 

If $v_3$ is not colored $c$, $d$, or $x$, then the number of colors remaining for $v'_7$, $v_7$, $v_0$, $v_1$, and $v_2$ are at least 2, 2, 4, 3, and 2 respectively since $L(v_1)=\{c,d,x\}$. So, $S$ is colorable thanks to \Cref{fig:forced_non_colorable_P5}.

If $v_3$ is colored $c$, $d$, or $x$, then neither $v_3$ nor $v_4$ is colored $y$. In other words, $v_2$ has $y$ as an available color while $v_1$ does not. Thus, $v'_7$, $v_7$, $v_0$, $v_1$, and $v_2$ can be colored by \Cref{fig:forced_non_colorable_P5} as they have at least 2, 2, 4, 2, and 2 remaining colors respectively.

\hfill\smallqed\medskip

\figureproof{\Cref{subfig:1c0a0c0b0c0b0c}}
If $v''_4$ sees $v''_6$ or $v'''_6$, then they must be at distance exactly 2 since $G$ has girth 8. Say $v_8$ is the common neighbor between $v''_4$ and $v''_6$, then $v'''_4$, $v'_4$, $v''_4$, $v_8$, $v''_6$, $v'_6$, and $v'''_6$ form the reducible configuration from \Cref{subfig:c1a1c}. 

If $v''_4$ sees $v'_7$, then they must be at distance exactly 2 since $G$ has girth 8. Say $v_8$ is their common neighbor, then $v'''_4$, $v'_4$, $v''_4$, $v_8$, $v'_7$, $v_7$, and $v_0$ form the reducible configuration from \Cref{subfig:c1a1c}. The same holds if $v''_4$ sees $v'_1$, or if $v''_6$ sees $v'_1$.

If $v''_6$ sees $v'_3$, then they must be at distance exactly 2 since $G$ has girth 8. Say $v_8$ is their common neighbor, then $v''_6$, $v_8$, $v'_3$, $v_3$, $v_4$, $v'_4$, $v''_4$, $v'''_4$, $v_5$, $v'_5$, $v_6$, $v'_6$, and $v'''_6$ form the reducible configuration from \Cref{subfig:1a1a0b0c0a0c}. 

Symmetrically, these observations also hold for $v'''_4$ and $v'''_6$.

If $v'_3$ sees $v'_7$, then they must be at distance exactly 2 since $G$ has girth 8. Say $v_8$ is their common neighbor, then $v'_3$, $v_8$, $v'_7$, $v_7$, $v_0$, $v_1$, and $v'_1$ form the reducible configuration from \Cref{subfig:c1a1c}. The same holds if $v'_5$ sees $v'_1$.

Therefore, we have \gs. We color $v'_6$ such that $v''_6$ has 3 colors left and $v'_4$ such that $v''_4$ has 3 colors left. Then, we color $v_7$ such that $v'_7$ has 3 colors left. Now, we color $v_2$, $v_3$, $v'_3$, $v_5$, $v_4$, $v'''_4$, $v''_4$, $v'_5$, $v'_1$, $v_1$, $v_0$, $v_6$, $v'_7$, $v'''_6$, and $v''_6$ in this order.

\hfill\smallqed\medskip

\figureproof{\Cref{subfig:1c0a0c0c0b0b0c}}
If $v''_6$ sees $v'_3$, then they must be at distance exactly 2 since $G$ has girth 8. Say $v_8$ is their common neighbor, then $v''_6$, $v_8$, $v'_3$, $v_3$, $v_4$, $v'_4$, $v_5$, $v'_5$, $v''_5$, $v'''_5$, $v_6$, $v'_6$, and $v'''_6$ form the reducible configuration from \Cref{subfig:1a1a0c0b0a0c}.

If $v''_6$ sees $v'_1$, then  they must be at distance exactly 2 since $G$ has girth 8. Say $v_8$ is their common neighbor, then $v''_6$, $v_8$, $v'_1$, $v_1$, $v_0$, $v_7$, and $v'_7$ form the reducible configuration from \Cref{subfig:c1a1c}.

The same observations hold for $v'''_6$ by symmetry.

If $v'_3$ sees $v'_7$, then  they must be at distance exactly 2 since $G$ has girth 8. Say $v_8$ is their common neighbor, then $v'_3$, $v_8$, $v'_7$, $v_7$, $v_0$, $v_1$, and $v'_1$ form the reducible configuration from \Cref{subfig:c1a1c}.

Now, observe that at least one vertex among $v''_5$ and $v'''_5$ do not see $v'_1$, say $v'''_5$. Also note that, if $v''_5$ sees $v'_1$, then $|L(v''_5)|\geq 4$ and $|L(v'_1)|\geq 4$. We color $v'_5$ such that $v'''_5$ has 3 colors left, $v'_6$ such that $v'''_6$ has 3 colors left, and $v_3$ such that $v'_3$ has 3 colors left. Then, we color $v_7$ such that $v'_7$ has 3 colors left. We finish by coloring $v'_4$, $v_5$, $v_4$, $v_2$, $v'_3$, $v''_5$, $v'''_5$, $v'_1$, $v_1$, $v_0$, $v_6$, $v'_7$, $v''_6$, and $v'''_6$ in this order.

\hfill\smallqed\medskip

\figureproof{\Cref{subfig:1c0b0c0a0c0b0c}}
If $v'_7$ sees $v'_3$, then they must be at distance exactly 2 since $G$ has girth $8$. Say $v_8$ is their common neighbor, then $v'_1$, $v_1$, $v_0$, $v_7$, $v'_7$, $v_8$, and $v'_3$ form the reducible configuration from \Cref{subfig:c1a1c}. Similarly, the same holds when $v'_7$ sees $v''_2$ or $v'''_2$.

If $v'_1$ sees $v'_5$, then they must be at distance exactly 2 since $G$ has girth $8$. Say $v_8$ is their common neighbor, then $v'_5$, $v_8$, $v'_1$, $v_1$, $v_0$, $v_7$, and $v'_7$ form the reducible configuration from \Cref{subfig:c1a1c}. Similarly, the same holds when $v'_1$ sees $v''_6$ or $v'''_6$.

If $v'_3$ sees $v'''_6$, then they must be at distance exactly 2 since $G$ has girth $8$. Say $v_8$ is their common neighbor, then $v'''_6$, $v_8$, $v'_3$, $v_3$, $v_4$, $v_5$, $v'_5$, $v_6$, $v_7$, $v'_7$, $v_0$, $v'_6$, and $v''_6$ form the reducible configuration from \Cref{subfig:1a1a0a0c0b0c}. Similarly, the same holds when $v'_3$ sees $v''_6$, or when $v'_5$ sees $v''_2$ or $v'''_2$.

If $v''_2=v''_6$, then $v''_2$, $v'_2$, $v_2$, $v_1$, $v'_1$, $v_0$, $v_7$, $v'_7$, $v_6$, $v'_6$, and $v'''_6$ form the reducible configuration from \Cref{subfig:1a0a0c1c0a0c}. Similarly, the same holds when $v''_2=v'''_6$, or when $v'''_2=v''_6$ or $v'''_6$.

Now, if $v''_2$ sees $v''_6$, then they must be at distance exactly 2 since $G$ has no $2^+-path$ by \Cref{lemma:no 2-path}. The same holds when $v''_2$ sees $v'''_6$, or when $v'''_2$ sees $v''_6$ or $v'''_6$. So, there is a vertex among $v''_2$ and $v'''_2$ that does not see $v''_6$ nor $v'''_6$, say $v''_2$. The same holds for $v''_6$ and $v'''_6$, so say $v''_6$ does not see $v''_2$ nor $v'''_2$. Observe that $|L(v''_6)|=|L(v''_2)|=3$, so we can color $v'_6$ differently from $L(v''_6)$, and $v'_2$ differently from $L(v''_2)$. By the pigeonhole principle, we can color $v_1$ and $v'_7$ with the same color since we have six colors in total. We finish by coloring $v_3$, $v_4$, $v_5$, $v'_3$, $v'_5$, $v_2$, $v'_1$, $v_6$, $v_7$, $v_0$, $v'''_6$, $v''_6$, $v'''_2$, and $v''_2$ in this order. 
\hfill\smallqed\medskip

\figureproof{\Cref{subfig:1a0a0c1c0b0a}}
If $v'_3$ sees $v''_6$, then they must be at distance exactly 2 since $G$ has girth $8$. Say $v_8$ is their common neighbor, then $v''_6$, $v_8$, $v'_3$, $v_3$, $v_4$, $v_5$, $v_6$, and $v'_6$ form the reducible configuration from \Cref{subfig:1a1a1a0a0a}. By symmetry the same holds when  $v'_3$ sees $v'''_6$.
Thus, we have \gs. Color vertex $v'_6$ with $x\notin L(v''_6)$, and afterwards color $v_5$ with $y\notin L(v'_5)$. We finish by coloring the remaining vertices in the following order: $v_7$, $v_1$, $v_2$, $v_0$, $v_6$, $v'''_6$, $v''_6$, $v_3$, $v'_3$, $v_4$, and $v'_5$.

\hfill\smallqed\medskip

\figureproof{\Cref{subfig:1a0a1c0a0c0c}}
Note that \gs. Color $v_4$ with $a\notin L(v'_4)$, then color $v_2$ and $v_1$ greedily. Color $v_0$, $v_7$, $v'_7$, $v_6$, $v'_6$, $v_5$ by \Cref{subfig:config7} and finish by coloring $v_3$ and $v'_4$ in this order.
\hfill\smallqed\medskip

\figureproof{\Cref{subfig:1a0c0a1c0a0c}}
Note that \gs. Here, we redefine $S=\{v_4,v_5,v'_5\}$. Consider $\phi$ a coloring of $G-S$. We uncolor $v_0$, $v_1$, $v_2$, $v'_2$, $v_3$, $v_6$, $v_7$, $v'_7$. By \Cref{fig:forced_non_colorable_3path}, we must have $L(v'_5)=\{a,b,c\}$, $L(v_4)=\{a,b,c,d,e\}$, $\phi(v_6)=c$, $\phi(v_3)=d$ and $\phi(v_2)=e$, as otherwise $\phi$ would be extendable to $G$. Note that $L(v_3)\neq L(v'_2)$ or we could have switched their colors in $\phi$ and $d$ would be an available color for $v_4$ and we could extend $\phi$ to $G$. Now, we color $v_2$ and $v_6$ with colors not in $\{d,e\}$. Then, we color $v'_7$, $v_7$, $v_0$, $v_1$ by \Cref{subfig:config1}. Color $v'_2$ and $v_3$ greedily, which is possible since $L(v'_2)\neq L(v_3)$. Finally, since at least $d$ or $e$ is available for $v_4$ and $d,e\notin L(v'_5)$, by \Cref{fig:forced_non_colorable_3path}, we can color $v_4$, $v_5$, $v'_5$.

\hfill\smallqed\medskip

\figureproof{\Cref{subfig:1a0a0c0c0c0c0a}}
Note that \gs. Observe that it is always possible to color $v_1$, $v_2$ and $v_7$ such that afterwards $v_0$ has at least two available colors. Indeed, either $v_2$ and $v_7$ can be colored with the same color, or $|L(v_2)\cup L(v_7)|\geq 5$.

Then color vertices $v'_6$, $v_6$, $v_5$, $v'_5$, $v_4$, $v'_4$, $v_3$, $v'_3$ by \Cref{subfig:config9} and finish by coloring $v_0$.
\hfill\smallqed\medskip

\figureproof{\Cref{subfig:1a1a0b0c0a0c}}
If $v''_4$ sees $v'_7$ by sharing a common neighbor, say $v_8$, then vertices $v_0$, $v_7$, $v'_7$, $v_8$, $v''_4$, $v'_4$, $v'''_4$ form the reducible configuration of \Cref{subfig:c1a1c}. The case when $v'''_4$ sees $v'_7$ is symmetric.

Therefore, we can suppose that \gs. We color $v'_4$ with a color $x\notin L(v''_4)$ and $v_7$ with a color $y\in L(v'_7)$. Now color $v_3$, $v_4$, $v_5$, $v'_5$, $v_6$ by \Cref{subfig:config4}. Finish by coloring $v'''_4$, $v''_4$, $v_1$, $v_2$, $v_0$, $v'_7$ in this order.
\hfill\smallqed\medskip

\figureproof{\Cref{subfig:1a1a0c0b0c0a}}
If $v''_5$ sees $v_1$ by sharing a common neighbor, say $v_8$, then vertices $v'''_5$, $v'_5$, $v''_5$, $v_8$, $v_1$, $v_2$, $v_0$ form the reducible configuration of \Cref{subfig:1a1a0c1}. The case when $v'''_5$ sees $v_1$ is symmetric.

Therefore, know that \gs. We prove first the following observations.

\begin{itemize}
\item $L(v_3)=\{a,b,c\}$ and $L(v_7)=\{d,e,f\}$. Suppose to the contrary that we can color $v_3$ and $v_7$ with the same color. Then color $v_0$ with $x$ such that $|L(v_1)\setminus\{x\}|\geq 2$. Color vertices $v'_4$, $v_4$, $v_5$, $v_6$, $v'_6$, $v'_5$, $v''_5$, $v'''_5$ by \Cref{subfig:config10} and finish by coloring $v_2$, $v_1$ in this order.

\item Observe that vertices $v_3$ and $v_7$ are symmetric and thus by pigeonhole principle w.l.o.g. we have $\{a,b\}\subset L(v_1)$.

\item $L(v_1)=L(v_3)=\{a,b,c\}$. If not, that is $c\notin L(v_1)$, then color $v_3$ with $c$ and $v_7$ with $x\notin L(v_1)$. Color vertices $v'_4$, $v_4$, $v_5$, $v_6$, $v'_6$, $v'_5$, $v''_5$, $v'''_5$ by \Cref{subfig:config10} and finish by coloring $v_0$, $v_2$, $v_1$ in this order.

\item $\{a,b,c\}\subset L(v_0)$. Otherwise, color $v_1$ with $x\notin L(v_0)$. Then color $v_3$, $v_4$, $v'_4$, $v_5$, $v'_5$, $v''_5$, $v'''_5$, $v_6$, $v'_6$, $v_7$ by \Cref{subfig:config16}. Finish by coloring $v_2$, $v_0$ in this order.
\end{itemize}

By the last item, w.l.o.g. we can assume that $|L(v_0)\setminus\{d,e\}| \geq 4$. Thus we restrict $L(v_7)$ to $\{d,e\}$. Then color $v_3$, $v_4$, $v'_4$, $v_5$, $v'_5$, $v''_5$, $v'''_5$, $v_6$, $v'_6$, $v_7$ by \Cref{subfig:config16}. Finish by coloring $v_1$, $v_2$, $v_0$ in this order.
\hfill\smallqed\medskip

\figureproof{\Cref{subfig:1c0a0c0c0c0a0c}}
If $v'_1$ sees $v'_5$, then the are at distance exactly 2 and share a common neighbor, say $v_8$. Then vertices $v_0$, $v_1$, $v'_1$, $v_8$, $v'_5$, $v_5$, $v_6$, $v_7$ correspond to the reducible configuration of \Cref{subfig:1a1a1a0a0a}. The case when $v'_7$ sees $v'_3$ is symmetric.

Therefore, we can assume that \gs. Color $v_1$ with $x\notin L(v'_1)$ and $v_5$ with $y\notin L(v'_5)$. Then color vertices $v'_4$, $v_4$, $v_3$, $v'_3$, $v_2$ by \Cref{subfig:config4}. Finish by coloring $v_6$, $v'_5$, $v_7$, $v'_7$, $v_0$, $v'_1$ in this order.
\hfill\smallqed\medskip

\figureproof{\Cref{subfig:0a0a0c0c0c0c0c0c}}
If $v'_2$ sees $v'_6$, then they must be at distance exactly 2 since $G$ has girth $8$. Say $v_8$ is their common neighbor, then $v'_6$, $v_8$, $v'_2$, $v_2$, $v_3$, $v'_3$, $v_4$, $v'_4$, $v_5$, $v'_5$ and $v_6$ form the reducible configuration from \Cref{subfig:1a1a0a0c0c0a}. The same holds when $v'_3$ sees $v'_7$.

Now, \gs.  We restrict $L(v_3)$ to $L(v_3)\setminus L(v'_4)$. We color $v_0$, $v_1$, $v_2$, $v'_2$, $v_3$ by \Cref{subfig:config5}, then we color $v'_3$. After that, we color $v_4$, $v_5$, $v'_5$, $v_6$, $v'_6$, $v_7$, $v'_7$ by \Cref{subfig:config8} and finish by coloring $v'_4$.
\hfill\smallqed\medskip

\figureproof{\Cref{subfig:1c1a0a0c0b0a}}
If $v'_1$ sees $v'_5$, then they must be at distance exactly 2 since $G$ has girth $8$. Say $v_8$ is their common neighbor, then $v_2$, $v_1$, $v'_1$, $v_8$, $v'_5$, $v_5$, $v_4$, $v_3$ form the reducible configuration from \Cref{subfig:1a1a1a0a0a}.

If $v'_1$ sees $v''_6$, then they must be at distance exactly 2 since $G$ has girth $8$ and say $v_8$ is their common neighbor. Then $v_0$, $v_1$, $v'_1$, $v_8$, $v''_6$, $v'_6$, $v_6$, $v_7$ form the reducible configuration from \Cref{subfig:1a1a1a0a0a}.

So we have \gs. Take a coloring $\phi$ where $v_3$ and $v_4$ are colored greedily, then  $v''_6$, $v'_6$, $v'''_6$, $v_6$, $v_7$, $v_5$, $v'_5$ are colored by \Cref{subfig:config8}. Then the remaining non-colored vertices are $v_0$, $v_1$, $v'_1$ and $v_2$. By \Cref{fig:forced_non_colorable_claw} we conclude that initially $L(v'_1)=\{a,b,c\}$, $L(v_0)=\{a,b,c,\phi(v_6),\phi(v_7)\}$, $L(v_1)=\{a,b,c,\phi(v_3),\phi(v_7)\}$ and $L(v_2)=\{a,b,c,\phi(v_3),\phi(v_4)\}$. Without loss of generality $\phi(v_3)=d$ and $\phi(v_7)=e$. Now observe that the color of $v_3$ was chosen arbitrarily, thus there exists a similar coloring $\phi'$ where $\phi'(v_3)\neq d$. Moreover, using again \Cref{fig:forced_non_colorable_claw}, $\phi'(v_3)=e$. Thus we deduce that $L(v_0)=L(v_1)=L(v_2)=\{a,b,c,d,e\}$ and $L(v_4)\supset\{d,e\}\subset L(v_3)$.

With all the remarks of the previous paragraphs, we give a coloring of the configuration: restrict $L(v_6)$ to $L(v_6)\setminus\{d,e\}$ and restrict $L(v_5)$ to $L(v_5)\setminus\{d,e\}$. Then color $v''_6$, $v'_6$, $v'''_6$, $v_7$, $v_5$ by \Cref{subfig:config7} and color $v'_5$, $v_4$, $v_3$ in this order. Recall that $v_6$ was colored $x\notin\{d,e\}$. Thus the list of remaining available colors for $v_0$ is not $\{a,b,c\}=L(v'_1)$ and hence by \Cref{fig:forced_non_colorable_claw} we are done.
\hfill\smallqed\medskip

\figureproof{\Cref{subfig:1c1a0b0a0c0a}}
If $v''_4$ sees $v'_1$, then they must be at distance exactly 2 since $G$ has girth $8$. Say $v_8$ is their common neighbor, then $v_2$, $v_1$, $v'_1$, $v_8$, $v''_4$, $v'_4$, $v_4$, $v_3$ form the reducible configuration from \Cref{subfig:1a1a1a0a0a}.

So we have \gs. We redefine $S=\{v_0,v_1,v'_1,v_2\}$ and take a coloring $\phi$ of $G-S$. By \Cref{fig:forced_non_colorable_claw} we know that $L(v'_1)=L(v_1)=L(v_0)=L(v_2)=\{a,b,c\}$ and the colors of $v'_6$ and $v_7$ cannot be interchanged. Having that said, we uncolor vertices $v_3$, $v_4$, $v'_4$, $v''_4$, $v'''_4$, $v_5$, $v_6$, $v'_6$ and $v_7$ and the number of available colors for each vertex correspond to the numbers depicted on \Cref{subfig:1c1a0b0a0c0a}. We know now that $L(v_0)=\{a,b,c,\phi(v_6),\phi(v_7)\}$ and furthermore $L(v'_6)\neq L(v_7)$.

We color $v_6$ with $x\notin\{phi(v_6),\phi(v_7)\}$ and $v''_4$, $v'_4$, $v'''_4$, $v_4$, $v_3$, $v_5$ by \Cref{subfig:config5}. Then we color $v'_6$ and $v_7$ since $L(v'_6)\neq L(v_7)$. Now observe that $L(v_0)\neq\{a,b,c\}=L(v'_1)$ and thus by \Cref{fig:forced_non_colorable_claw} we are done.
\hfill\smallqed\medskip

\figureproof{\Cref{subfig:1c1a0a0b0c0a}}
Suppose $v''_5$ sees $v'_1$. By \Cref{lemma:no 2-path} they are at distance exactly two and therefore $|L(v''_5)|=|L(v'_1)|=4$. By pigeonhole principle we color vertices $v_0$ and $v_4$ with the same color $x$ and show the following:
\begin{itemize}
\item $x\notin L(v'_6)$. If not then we color $v'_6$ with $x$ as well and finish by coloring $v_3$, $v_7$, $v_5$, $v_6$, $v'_5$, $v'''_5$, $v''_5$, $v'_1$, $v_1$, $v_2$ in order.
\item $x\in L(v_6)$. If not then we color $v'_6$ arbitrarily and finish by coloring $v_3$, $v_7$, $v_5$, $v_6$, $v'_5$, $v'''_5$, $v''_5$, $v'_1$, $v_1$, $v_2$ in order.
\item $x\in L(v_3)$. If not then we color $v'_6$ arbitrarily and finish by coloring $v_7$, $v_6$, $v_5$, $v_3$, $v'_5$, $v'''_5$, $v''_5$, $v'_1$, $v_1$, $v_2$ in order.
\end{itemize}

We recolor the whole configuration as follows. Color $v_3$ and $v_6$ with $x$, then color $v_4$. Color $v_5$ such that vertex $v_7$ has at least two available colors left. Color $v'_5$, $v'''_5$, $v''_5$, $v'_1$ in this order. Color $v_7$, $v_0$, $v_1$, $v_2$ by \Cref{subfig:config1}.

The case when $v'''_5$ sees $v'_1$ is symmetric.

So we have \gs. Redefine $S=\{v_0,v_1,v'_1,v_2\}$. Take a coloring $\phi$ of $G-S$. If vertices of $S$ are colorable, then we are done. Hence by using \Cref{fig:forced_non_colorable_claw} we uncolor vertices $v_3$, $v_4$, $v_5$, $v'_5$, $v''_5$, $v'''_5$, $v_6$, $v'_6$, $v_7$ and conclude that after uncoloring $L(v'_1)=\{a,b,c\}$, $L(v_0)=\{a,b,c,\phi(v_6),\phi(v_7)\}$, $L(v_1)=\{a,b,c,\phi(v_3),\phi(v_7)\}$ and $L(v_2)=\{a,b,c,\phi(v_3),\phi(v_4)\}$. Observe that $L(v'_6)\neq L(v_7)$ as their color could be permuted and $\phi$ could be extended to $S$. 

Without loss of generality $\phi(v_6)=d$ and $\phi(v_7)=e$. Now one could restrict $L(v_6)$ to $L(v_6)\setminus\{d\}$, and give another coloring $\phi'$ of $G-S$ where first vertices $v''_5$, $v'_5$, $v'''_5$, $v_5$, $v_6$, $v_3$, $v_4$ are colored using \Cref{subfig:config8} and then since $L(v'_6)\neq L(v_7)$, vertices $v'_6$ and $v_7$ are colored greedily. Note that since $\phi'(v_6)\neq d$, using again \Cref{fig:forced_non_colorable_claw}, we necessarily have $\phi'(v_6)=e$. Thus we deduce  $L(v_0)=L(v_1)=L(v_2)=\{a,b,c,d,e\}$ and $L(v_4)\supset\{d,e\}\subset L(v_3)$.

With all the remarks of the previous paragraphs, we give a coloring of the configuration: restrict $L(v_6)$ to $L(v_6)=\{x,y\} \cap\{d,e\}=\emptyset$. Since $|L(v_5)| = 5$, we color $v_5$ with $z\notin\{x,y,d,e\}$. Then we color $v''_5$, $v'''_5$, $v'_5$, $v_4$, $v_3$ in this order. Now recall that initially $L(v'_6)\neq L(v_7)$ and that $v_5$ and $v_4$ were colored with colors other than $x$ and $y$. Therefore we color vertices $v_6$, $v'_6$, $v_7$ in this order. Recall that $v_6$ was colored say $x\notin\{d,e\}$. Thus the list of remaining available colors for $v_0$ is not $\{a,b,c\}=L(v'_1)$ and hence by \Cref{fig:forced_non_colorable_claw} we are done.
\hfill\smallqed\medskip

\figureproof{\Cref{subfig:0a0a0c0b0c0c0c0c}}
Restrict $L(v_3)$ to $L(v_3)\setminus L(v''_3)$. Then color vertices $v_0$, $v_1$, $v_2$, $v'_2$, $v_3$ by \Cref{subfig:config5}. Color vertices $v'_4$, $v_4$, $v_5$, $v'_5$, $v_6$, $v'_6$, $v_7$, $v'_7$ by \Cref{subfig:config9}. Finish by coloring $v'_3$, $v'''_3$, $v''_3$ in this order.
Note that this coloring procedure works even when $v''_3$ (resp. $v'''_3$) sees $v'_6$ or $v'_7$ and when $v'_2$ sees $v'_6$.
\hfill\smallqed\medskip

\figureproof{\Cref{subfig:0a0c0c0c0b0c0c0c}}
If $v'_1$ sees $v'_5$, then they must be at distance exactly 2 since $G$ has girth $8$. Say $v_8$ is their common neighbor, then $v'_5$, $v_8$, $v'_1$, $v_1$, $v_0$, $v_7$, $v'_7$, $v_6$, $v'_6$, $v_5$ form the reducible configuration from \Cref{subfig:1a1a0a0c0c0a}. Symmetrically, the same holds when $v'_3$ sees $v'_7$.

If $v'_2$ sees $v'_6$, then they must be at distance exactly 2 since $G$ has girth $8$. Say $v_8$ is their common neighbor, then $v'_6$, $v_8$, $v'_2$, $v_2$, $v_3$, $v'_3$, $v_4$, $v'_4$, $v''_4$, $v'''_4$, $v_5$, $v'_5$, $v_6$ form the reducible configuration from \Cref{subfig:1a1a0c0b0c0a}.

Color $v_7$ with a color that is not in $L(v'_7)$ and color $v_6$ such that $v'_6$ has at least two colors left. Color $v_0$ greedily. Color $v'_4$ such that $v''_4$ has at least three colors left. Now, $2=|L(v'_5)|\leq |L(v_5)|\leq 3$. If there exists $x\in L(v'_5)\setminus L(v_5)$, then we can color $v'_5$ with $x$, then color $v_5$, $v_4$, $v_3$, $v'_3$, $v_2$, $v'_2$, $v_1$, $v'_1$ by \Cref{subfig:config9}. We can finish by coloring $v'''_4$, $v''_4$, $v'_6$, $v'_7$ in this order. As a result, $L(v'_5)\subseteq L(v_5)$, in which case, we restrict $L(v_4)$ to $L(v_4)\setminus L(v'_5)$ and color $v_4$, $v_3$, $v'_3$, $v_2$, $v'_2$, $v_1$, $v'_1$ by \Cref{subfig:config8}. Finally, we finish by coloring $v_5$, $v'_5$, $v'''_4$, $v''_4$, $v'_6$, $v'_7$ in this order. Note that this coloring procedure works even when $v_0$ sees $v''_4$ or $v'''_4$. 
\hfill\smallqed\medskip

\figureproof{\Cref{subfig:1a0c0c0a0c0b0c}}
If $v'_3$ sees $v'_7$, then they must be at distance exactly 2 since $G$ has girth $8$. Say $v_8$ is their common neighbor, then $v_0$, $v_7$, $v'_7$, $v_8$, $v'_3$, $v_2$, $v_1$ form the reducible configuration from \Cref{subfig:1a1a1a0a0a}.

Restrict $L(v_0)$ to $L(v_0)\setminus L(v'_7)$. Color $v_4$, $v_3$, $v'_3$, $v_2$, $v'_2$, $v_1$, $v_0$ by \Cref{subfig:config8}. Color $v''_6$, $v'_6$, $v'''_6$, $v_6$, $v_7$, $v_5$, $v'_5$ by \Cref{subfig:config8} and finish by coloring $v'_7$. Note that this coloring procedure works even when $v'_2$ (resp. $v'_3$) sees $v''_6$ or $v'''_6$.
\hfill\smallqed\medskip

\figureproof{\Cref{subfig:1a0c0c0a0c0b0c}}
If $v'_3$ sees $v'_7$, then they must be at distance exactly 2 since $G$ has girth $8$. Say $v_8$ is their common neighbor, then $v_0$, $v_7$, $v'_7$, $v_8$, $v'_3$, $v_3$, $v_2$, $v_1$ form the reducible configuration from \Cref{subfig:1a1a1a0a0a}.

Restrict $L(v_0)$ to $L(v_0)\setminus L(v'_7)$. Color $v_4$, $v_3$, $v'_3$, $v_2$, $v'_2$, $v_1$, $v_0$ by \Cref{subfig:config8}. Color $v''_6$, $v'_6$, $v'''_6$, $v_6$, $v_7$, $v_5$, $v'_5$ by \Cref{subfig:config8} and finish by coloring $v'_7$. Note that this coloring procedure works even when $v'_2$ (resp. $v'_3$) sees $v''_6$ or $v'''_6$ at distance 2 since there are no 2-paths due to \Cref{lemma:no 2-path} (resp. since $G$ has girth at least 8).
\hfill\smallqed\medskip

\figureproof{\Cref{subfig:1a0b0c0c0c0a0c}}
If $v''_2$ sees $v'_7$, then they must be at distance exactly 2 since $G$ has girth $8$. Say $v_8$ is their common neighbor, then $v_0$, $v_7$, $v'_7$, $v_8$, $v''_2$, $v'_2$, $v_2$, $v_1$ form the reducible configuration from \Cref{subfig:1a1a1a0a0a}. Symmetrically, the same holds when $v'''_2$ sees $v'_7$.

If $v''_2$ sees $v'_5$, then they must be at distance exactly 2 since $G$ has girth $8$. Say $v_8$ is their common neighbor, then $v'_5$, $v_8$, $v''_2$, $v'_2$, $v_2$, $v_3$, $v'_3$, $v_4$, $v'_4$, $v_5$ form the reducible configuration from \Cref{subfig:1a1a0a0c0c0a}. Symmetrically, the same holds when $v'''_2$ sees $v'_5$.

Between $v''_2$ and $v'''_2$, there always exists one vertex that does not see $v_6$, say $v''_2$. Color $v'_2$ with a color that is not in $L(v''_2)$. By pigeonhole principle, since we have 6 colors, color $v_2$ and $v_6$ with the same color. Color $v_1$ greedily. Color $v'_3$, $v_3$, $v_4$, $v'_4$, $v_5$, $v'_5$ by \Cref{subfig:config6}. Finish by coloring $v_7$, $v'_7$, $v_0$, $v'''_6$, $v''_6$ in this order. 
Note that this coloring procedure works even when $v'_3$ sees $v'_7$.
\hfill\smallqed\medskip

\figureproof{\Cref{subfig:1a0c0b0c0c0a0c}}
Between $v''_3$ and $v'''_3$, there always exists one vertex that does not see $v'_7$, say $v''_3$.

If $L(v_1)=L(v'_2)$, then color $v_2$ with $x\notin L(v'_2)$. Restrict $L(v_3)$ to $L(v_3)\setminus L(v''_3)$. Color $v_3$, $v_4$, $v'_4$, $v_5$, $v'_5$, $v_6$ by \Cref{subfig:config7}. Color $v'_3$, $v'''_3$, $v''_3$ in this order. Then, color $v'_7$, $v_7$, $v_0$, $v_1$ by \Cref{subfig:config1} and finish by coloring $v'_2$.

If $L(v_1)\neq L(v'_2)$, then, by pigeonhole principle, color $v_2$ and $v_6$ with the same color. Restrict $L(v_3)$ to $L(v_3)\setminus L(v''_3)$. Color $v'_5$, $v_5$, $v_4$, $v'_4$, $v_3$ by \Cref{subfig:config5}. Color $v'_3$, $v'''_3$, $v''_3$ in this order. Then, color $v_1$ and $v_2$, which is possible since $L(v_1)\neq L(v'_2)$. Finish by coloring $v_7$, $v'_7$ and $v_0$ in this order. 

Note that this coloring procedure works even when $v'''_3$ sees $v'_7$ (at distance 2 since there are no 2-paths by \Cref{lemma:no 2-path}).

\hfill\smallqed\medskip

\figureproof{\Cref{subfig:1a0c0c0c0c0a0b}}
Between $v''_7$ and $v'''_7$, there always exists on vertex that does not see $v'_3$, say $v''_7$. Color $v'_7$ with a color not in $v''_7$. Color $v_2$ with a color such that $v'_3$ still retains three available colors. Color $v_0$ such that $v_6$ still retain two available colors. Color $v_1$ and $v'_2$ greedily. Color $v_3$, $v_4$, $v'_4$, $v_5$, $v'_5$, $v_6$, $v_7$ by \Cref{subfig:config8}. Finish by coloring $v'_3$, $v'''_7$ and $v''_7$ in this order. Note that this coloring procedure works even when $v''_7$ (resp. $v'''_7$) sees $v'_2$ or $v'_4$.
\hfill\smallqed\medskip

\figureproof{\Cref{subfig:1a0c0c0b0c0a0c}}
Between $v''_4$ and $v'''_4$, there always exists one vertex that does not see $v'_7$, say $v''_4$. Color $v'_4$ with a color $x\notin L(v''_4)$. Color $v_7$ such that $v'_7$ has at least three colors left. Color $v_6$, $v_5$, $v'_5$ in this order. Color $v_1$, $v_2$, $v'_2$, $v_3$, $v'_3$, $v_4$ by \Cref{subfig:config7}. Finish by coloring $v'''_4$, $v''_4$, $v_0$ and $v'_7$ in this order. Note that this coloring procedure works even when $v'''_4$ sees $v'_7$ (at distance 2 since $G$ has girth at least 8).
\hfill\smallqed\medskip

\figureproof{\Cref{subfig:1a0c0b0c0c0c0a}}
Restrict $L(v_3)$ to $L(v_3)\setminus L(v''_3)$. Color $v_2$ such that $v'_2$ has at least three colors left. Color $v_3$ then $v_1$ greedily. Color $v'_4$, $v_4$, $v_5$, $v'_5$, $v_6$, $v'_6$, $v_7$, $v_0$ by \Cref{subfig:config9}. Finish by coloring $v'_2$, $v'_3$, $v'''_3$ and $v''_3$ in this order. Note that this coloring procedure works even when $v''_3$ (resp. $v'''_3$) sees $v'_6$ or $v_7$, and when $v'_2$ sees $v'_6$.
\hfill\smallqed\medskip

\figureproof{\Cref{subfig:1a0c0c0b0c0c0a}}
Note that there always exists $ x \in L(v''_4) \cap L(v'_4)$. We start by showing the following observations:
\begin{itemize}
\item $x \in L(v_3)=L(v_5)$.  Now suppose w.l.o.g. that $x\notin L(v_3)$ or $L(v_3)\neq L(v_5)$. Color $v'_4$ with $x$. Color $v_5$ such that $v_3$ has at least four colors left. Color $v'_6$, $v_7$, $v_6$, $v'_5$ in this order. Color $v_4$, $v_3$, $v'_3$, $v_2$, $v'_2$, $v_1$, $v_0$ by \Cref{subfig:config8}. Finish by coloring $v'''_4$ and $v''_4$ in this order.
\item $x \notin L(v_2)$. Suppose that $x\in L(v_2)$. We color $v'_4$ and $v_2$ with $x$. Color $v_1$ and $v'_2$.
\begin{itemize}
\item If we can color $v'_3$ such that $v_3$ has at least two colors left, then we can color $v_3$, $v_4$, $v_5$, $v'_5$, $v_6$, $v'_6$, $v_7$, $v_0$ by \Cref{subfig:config9}, and finish by coloring $v'''_4$ and $v''_4$ in this order.
\item Otherwise, we must have $|L(v'_3)| = |L(v_3)| = 2$, in which case, we restrict $L(v_4)$ to $L(v_4)\setminus L(v_3)$. Now, we can color $v_4$, $v_5$, $v'_5$, $v_6$, $v'_6$, $v_7$, $v_0$ by \Cref{subfig:config8}, and finish by coloring $v_3$, $v'_3$, $v'''_4$ and $v''_4$ in this order.
\end{itemize}
\end{itemize}
With the observations above, we color $v_4$ with $x$ (since $L(v_4)$ contains all available colors). Color $v'_5$, $v_5$, $v_6$, $v'_6$, $v_7$ by \Cref{subfig:config5}. Color $v_0$, $v_1$, $v_2$, $v'_2$, $v_3$, $v'_3$ by \Cref{subfig:config6}. Finish by coloring $v'_4$, $v'''_4$ and $v''_4$ in this order. 

Note that in all of the above-mentioned coloring procedure, there is no problem even when $v'_2$ sees $v'_6$. 

\hfill\smallqed\medskip

\figureproof{\Cref{subfig:1a0a1c0b0b0c}}
Between $v''_5$ and $v'''_5$ (resp. $v''_6$ and $v'''_6$), there always exists one vertex that does not see $v_1$ (resp. $v_2$), say $v''_5$ (resp. $v''_6$). Restrict $L(v_5)$ to $L(v_5)\setminus L(v''_5)$. Color $v'_6$ with a color not in $L(v''_6)$. Color $v_7$ with a color not in $L(v'_7)$. Finish by coloring $v_1$, $v_2$, $v_5$, $v_3$, $v'_4$, $v_4$, $v_6$, $v_0$, $v'_7$, $v'''_6$, $v''_6$, $v'_5$, $v'''_5$, $v''_5$ in this order.
\hfill\smallqed\medskip

\figureproof{\Cref{subfig:1a0c1a0c0a0b}}
If $v''_7$ sees $v'_2$, then they must be at distance exactly 2 since $G$ has girth $8$. Say $v_8$ is their common neighbor, then $v'''_7$, $v'_7$, $v''_7$, $v_8$, $v'_2$, $v_2$, $v_3$ form the reducible configuration from \Cref{subfig:c1a1c}. Symmetrically, the same holds when $v'''_7$ sees $v'_2$.

So, we have \gs. Here, we redefine $S=\{v_0,v_7,v'_7,v''_7,v'''_7\}$ and consider $\phi$ a coloring of $G-S$. Note that the lists of available colors of vertices of $S$ correspond to \Cref{fig:forced_non_colorable_clawpath} or $\phi$ would be extendable to $G$. We uncolor $v_1$, $v_2$, $v'_2$, $v_3$, $v_4$, $v_5$, $v'_5$, $v_6$. By \Cref{fig:forced_non_colorable_clawpath}, we must have $L(v''_7)=L(v'''_7) = \{a,b,c\}$, $L(v'_7)=\{a,b,c,d\}$, $\phi(v_6)=d$, $e\in L(v_7)$ and $\phi(v_5) = e$ or $\phi(v_1) = e$. Note that $L(v_6)\neq L(v'_5)$, otherwise, it suffices to switch their colors in $\phi$ to extend it to $G$ by \Cref{fig:forced_non_colorable_clawpath}. Restrict $L(v_5)$ to $L(v_5)\setminus\{e\}$ and $L(v_1)$ to $L(v_1)\setminus\{e\}$. Color $v_1$, then $v'_2$, $v_2$, $v_3$, $v_4$, $v_5$ by \Cref{subfig:config2}. Color $v'_5$ and $v_6$ greedily (which is possible since $L(v_6)\neq L(v_5)$). Now, observe that $v_6$ must still be colored $d$ or by \Cref{fig:forced_non_colorable_clawpath}, we can extend this coloring. However, we know that $v_1$ and $v_5$ are not colored $e$, thus $e$ is still an available color for $v_7$. By \Cref{fig:forced_non_colorable_clawpath}, $S$ is colorable.
\hfill\smallqed\medskip

\figureproof{\Cref{subfig:1a0c1a0b0c0a}}
If $v''_5$ sees $v'_2$, then they must be at distance exactly 2 since $G$ has girth $8$. Say $v_8$ is their common neighbor, then $v''_5$, $v_8$, $v'_2$, $v_2$, $v'_2$, $v_2$, $v_3$, $v_4$, $v_5$, $v'_5$ form the reducible configuration from \Cref{subfig:1a1a1a0a0a}. Symmetrically, the same holds when $v'''_5$ sees $v'_2$.

Between $v''_5$ and $v'''_5$, there always exists one vertex that does not see $v_1$, say $v''_5$. We restrict $L(v_5)$ to $L(v_5)\setminus L(v''_5)$. There exists a color $x\in L(v_2)\setminus L(v'_2)$, we restrict $L(v_4)$ to $L(v_4)\setminus\{x\}$. We color $v_4$, $v_5$, $v_6$, $v'_6$, $v_7$ by \Cref{subfig:config5}. Then, we color $v'_5$, $v'''_5$, $v''_5$ in this order. Now, observe that by \Cref{fig:forced_non_colorable_clawpath}, we can color $v_0$, $v_1$, $v_2$, $v'_2$, $v_3$ since $x$ is an available color for $v_2$ but not $v'_2$. Note that this coloring procedure works even when $v'''_5$ sees $v_1$.
\hfill\smallqed\medskip

\figureproof{\Cref{subfig:1a0c1a0c0b0a}}
If $v''_6$ sees $v'_2$, then they must be at distance exactly two since there are no $2$-path by \Cref{lemma:no 2-path}. We restrict $L(v_6)$ to $L(v_6)\setminus L(v'''_6)$, then we color $v_6$, $v_4$, $v'_5$, $v_5$, $v_7$, $v_0$, $v_1$, $v_2$, $v_3$, $v'_2$, $v'_6$, $v''_6$, $v'''_6$ in this order. Symmetrically, the same holds when $v'''_6$ sees $v'_2$.

There exists a color $x\in L(v_2)\setminus L(v'_2)$, we restrict $L(v_4)$ to $L(v_4)\setminus\{x\}$. We restrict $L(v_6)$ to $L(v_6)\setminus L(v''_6)$. We color $v_7$, $v_6$, $v_5$, $v'_5$, $v_4$ by \Cref{subfig:config5}. Then, we color $v'_6$, $v'''_6$, $v''_6$ in this order. Now, observe that by \Cref{fig:forced_non_colorable_clawpath}, we can color $v_0$, $v_1$, $v_2$, $v'_2$, $v_3$ since $x$ is an available color for $v_2$ but not $v'_2$.
\hfill\smallqed\medskip

\figureproof{\Cref{subfig:1a0a0c1c0a0b}}
Between $v''_7$ and $v'''_7$, there always exists one vertex that does not see $v'_3$, say $v''_7$. Color $v'_7$ with a color not in $L(v''_7)$. Color $v_5$ with a color not in $L(v'_5)$. Color $v_6$ greedily. Color $v_7$, $v_0$, $v_1$, $v_2$ by \Cref{subfig:config1}. Finish by coloring $v_3$, $v'_3$, $v_4$, $v'_5$, $v'''_7$, $v''_7$ in this order. Note that this coloring procedure works even when $v'''_7$ sees $v'_3$.
\hfill\smallqed\medskip

\figureproof{\Cref{subfig:1a0c0a1b0a0c}}
If $v''_5$ sees $v_1$, then they must be at distance exactly 2 since $G$ has girth $8$. Say $v_8$ is their common neighbor, then $v''_5$, $v_8$, $v_1$, $v_0$, $v_7$, $v'_7$, $v_6$, $_5$, $v'_5$, $v'''_5$ form the reducible configuration from \Cref{subfig:1a0a1c0a0c0c}. Symmetrically, the same holds when $v'''_5$ sees $v_1$.

If $v''_5$ sees $v'_2$, then we restrict $L(v_5)$ to $L(v_5)\setminus L(v'''_5)$. We color $v_3$ with a color not in $L(v_5)$. Color $v_2$, $v_1$, $v_0$, $v_7$, $v'_7$, $v_6$, $v_5$ by \Cref{subfig:config11}. Finish by coloring $v'_2$, $v_4$, $v'_5$, $v''_5$, $v'''_5$ in this order. Symmetrically, the same holds when $v'''_5$ sees $v'_2$.

Now, suppose that \gs. We redefine $S=\{v_4,v_5,v'_5,v''_5,v'''_5\}$ and consider $\phi$ a coloring of $G-S$. Note that the lists of available colors of vertices of $S$ correspond to \Cref{fig:forced_non_colorable_clawpath} as otherwise $\phi$ would be extendable to $G$. We uncolor $v_0$ $v_1$, $v_2$, $v'_2$, $v_3$, $v_6$, $v_7$, $v'_7$. By \Cref{fig:forced_non_colorable_clawpath}, we must have $L(v''_5)=L(v'''_5) = \{a,b,c\}$, $L(v'_5)=\{a,b,c,d\}$ and $\phi(v_6)=d$. Restrict $L(v_6)$ to $L(v_6)\setminus\{d\}$. Color $v_7$ with a color not in $L(v_1)$. Color $v_6$ and $v'_7$ greedily. Observe that  $|L(v_1)\cup L(v_2)\cup L(v'_2)\cup L(v_3)|>3$ since they were colorable with $\phi$, and also note that these vertices do not see $v_6$, $v_7$, $v'_7$. Therefore, by \Cref{fig:forced_non_colorable_clawpath}, we color $v_0$, $v_1$, $v_2$, $v'_2$, $v_3$. What remains is $S$ and since $v_6$ is not colored $d$, we have $ d\in L(v'_5) \cap L(v''_5)$ so $S$ is colorable by \Cref{fig:forced_non_colorable_clawpath}.

\hfill\smallqed\medskip

\figureproof{\Cref{subfig:1a0b0c1a0c0a}}
If $v''_2$ sees $v'_6$, then they must be at distance exactly 2, since there are no $2$-path (\Cref{lemma:no 2-path}). Restrict $L(v_2)$ to $L(v_2)\setminus L(v'_6)$. Color $v_3$ with a color not in $v'_3$. Color $v_2$ and $v_1$ greedily. Color $v_5$, $v_6$, $v_7$, $v_1$ by \Cref{subfig:config1}. Finish by coloring $v'_6$, $v_4$, $v'_3$, $v'_2$, $v''_2$, $v'''_2$ in this order.

Now, suppose that \gs. We redefine $S = \{v_0, v_1, v_2, v'_2, v''_2, v'''_2, v_3, v'_3, v_4\}$ and consider $\phi$ a coloring of $G-S$. If there exists $x\in L(v_4)\notin L(v_3)$, then color $v_4$ with $x$. Color $v'''_2$, $v'_2$, $v''_2$, $v_2$, $v_3$, $v_1$, $v_0$ by \Cref{subfig:config8}. Finish by coloring $v'_3$. Thus, we uncolor $v_5$, $v_6$, $v_7$ and conclude that $L(v'_3)=\{a,b,c\} \subset L(v_4) = \{a,b,c,d,e\}$, $\phi(v_5)=d$ and $\phi(v_6)=e$. Also note that $L(v_5)\neq L(v'_6)$ or we can simply switch $v_5$'s and $v_6$'s colors and $d$ would still be available for $v_4$.

Now, we color $v_6$ with a color different from $d$ and $e$. We color $v_7$ greedily. We color $v_5$ and $v'_6$ (which is possible since $L(v_5)\neq L(v'_6)$. Note that either $d$ or $e$ must still be available for $v_4$ and $d,e\notin L(v'_3)$ so we refer to the above-mentioned coloring.

\hfill\smallqed\medskip

\figureproof{\Cref{subfig:1a0a0c1a0b0c}}
If $v'_3$ sees $v'_7$, then they must be at distance exactly 2 since $G$ has girth $8$. Say $v_8$ is their common neighbor, then $v_0$, $v_7$, $v'_7$, $v_8$, $v'_3$, $v_3$, $v_2$, $v_1$ form the reducible configuration from \Cref{subfig:1a1a1a0a0a}. If $v'_3$ sees $v''_6$, then they share a common neighbor $v_8$ and $v''_6$, $v_8$, $v'_3$, $v_3$, $v_4$, $v_5$, $v_6$, $v'_6$ form the reducible configuration from \Cref{subfig:1a1a1a0a0a}. Symmetrically, the same holds when $v'_3$ sees $v'''_6$.

If $v_2$ sees $v''_6$, then they must be at distance exactly 2. Note that, in this case, $|L(v_2)|\geq 3$ so we can color $v_2$ then $v_3$ such that $v_1$ retains at least 2 available colors. We color $v'_6$ with $x\notin L(v''_6)$. Then, we color $v'_3$, $v_5$, $v_4$ in this order. Afterwards, we color $v_1$, $v_0$, $v_7$, $v'_7$ and $v_6$ by \Cref{subfig:config4}. Finish by coloring $v'''_6$ then $v''_6$ in this order. Symmetrically, the same holds when $v_2$ sees $v'''_6$.

So we have \gs. Observe that in the previous case, it suffices to color $v_2$ such that $v_1$ still retain at least 2 available colors to be able to extend the coloring to $G$. Thus, $L(v_2)=L(v_1)=\{a,b\}$. 

Now, if $a\in L(v'_3)$, then we color $v'_3$ and $v_1$ with $a$. Restrict $L(v_6)$ to $L=L(v_6)\setminus L(v''_6)$. Color $v'_7$ with $v\notin L$. Then, we color $v_0$, $v_7$, $v_6$, $v_5$, $v_4$, $v_3$ with \Cref{subfig:config3}. and finish by coloring $v'_6$, $v'''_6$ and $v''_6$ in this order.

If $a\notin L(v'3)$, color $v_2$ with $a$. By pigeonhole principle, we color $v'_6$ and $v'_7$ with the same color. Then, we color $v''_6$ and $v'''_6$ greedily. Afterwards, we color $v_5$, $v_6$, $v_7$, $v_0$ by \Cref{subfig:config1}. We finish by coloring $v_3$, $v_4$, $v'_3$ in this order.
\hfill\smallqed\medskip

\end{proof}

\end{appendices}

\end{document}